\documentclass[11pt]{amsart}

\usepackage{amsmath}
\usepackage{fullpage}
\usepackage{xspace}
\usepackage[psamsfonts]{amssymb}
\usepackage[latin1]{inputenc}
\usepackage{graphicx,color}
\usepackage[curve]{xypic} 
\usepackage{hyperref}
\usepackage{graphicx}
\usepackage{amsmath}%
\usepackage{amsthm}%
\usepackage{amscd}
\usepackage{amsfonts}%
\usepackage{amssymb}%
\usepackage{graphicx}

\usepackage{mathrsfs}

\usepackage{tikz}
\usetikzlibrary{matrix,arrows}

\usepackage{tikz-cd}

\newtheorem{theorem}{Theorem}[section]

\newtheorem{corollary}[theorem]{Corollary}

\newtheorem{lemma}[theorem]{Lemma}

\newtheorem{proposition}[theorem]{Proposition}

\theoremstyle{remark}
\newtheorem{remark}[theorem]{Remark}

\numberwithin{equation}{section}

\newcommand{\pfrak}{\mathfrak{p}}

\newcommand{\qfrak}{\mathfrak{q}}

\newcommand{\gfrak}{\mathfrak{g}}

\newcommand{\mfrak}{\mathfrak{m}}

\newcommand{\efrak}{\mathfrak{e}}
\newcommand{\ffrak}{\mathfrak{f}}

\newcommand{\Acal}{\mathscr{A}}

\newcommand{\Ecal}{\mathscr{E}}
\newcommand{\Fcal}{\mathscr{F}}
\newcommand{\Hcal}{\mathscr{H}}
\newcommand{\Ical}{\mathscr{I}}

\newcommand{\Ocal}{\mathscr{O}}
\newcommand{\Pcal}{\mathscr{P}}

\newcommand{\Scal}{\mathscr{S}}
\newcommand{\Tcal}{\mathscr{T}}
\newcommand{\Vcal}{\mathscr{V}}
\newcommand{\Xcal}{\mathscr{X}}
\newcommand{\Ycal}{\mathscr{Y}}
\newcommand{\Zcal}{\mathscr{Z}}

\newcommand{\Z}{\mathbb{Z}}
\newcommand{\C}{\mathbb{C}}

\newcommand{\F}{\mathbb{F}}
\newcommand{\Q}{\mathbb{Q}}
\newcommand{\R}{\mathbb{R}}
\newcommand{\N}{\mathbb{N}}
\newcommand{\A}{\mathbb{A}}

\newcommand{\kk}{\mathbf{k}}
\newcommand{\nn}{\mathbf{n}}

\newcommand{\xx}{\mathbf{x}}
\newcommand{\yy}{\mathbf{y}}
\newcommand{\ttt}{\mathbf{t}}
\newcommand{\vv}{\mathbf{v}}
\newcommand{\uu}{\mathbf{u}}

\newcommand{\ord}{\mathrm{ord}}
\newcommand{\End}{\mathrm{End}}

\newcommand{\Spec}{\mathrm{Spec}\,}

\newcommand{\Lie}{\mathrm{Lie}}

\newcommand{\rk}{\mathrm{rank}\,}
\newcommand{\supp}{\mathrm{supp}}

\newcommand{\Exp}{\mathrm{Exp}}
\newcommand{\red}{\mathrm{red}}
\newcommand{\Log}{\mathrm{Log}}
\newcommand{\divi}{\mathrm{div}}
\newcommand{\mult}{\mathrm{mult}}
\newcommand{\GL}{\mathrm{GL}}

\newcommand{\tr}{\mathrm{tr}}

\newcommand{\Spf}{\mathrm{Spf}\,}
\newcommand{\Id}{\mathrm{Id}}
\newcommand{\im}{\mathrm{im}}
\newcommand{\cdeg}{\,\mathrm{cdeg}}

\newcommand{\Sbf}{\mathbf{S}}

  \DeclareFontFamily{U}{wncy}{}
    \DeclareFontShape{U}{wncy}{m}{n}{<->wncyr10}{}
    \DeclareSymbolFont{mcy}{U}{wncy}{m}{n}
    \DeclareMathSymbol{\Sha}{\mathord}{mcy}{"58}

\begin{document}
\title{A Chabauty-Coleman bound for surfaces}

\author{Jerson Caro}
\address{ Departamento de Matem\'aticas,
Pontificia Universidad Cat\'olica de Chile.
Facultad de Matem\'aticas,
4860 Av.\ Vicu\~na Mackenna,
Macul, RM, Chile}
\email[J. Caro]{jocaro@uc.cl }%

\author{Hector Pasten}
\address{ Departamento de Matem\'aticas,
Pontificia Universidad Cat\'olica de Chile.
Facultad de Matem\'aticas,
4860 Av.\ Vicu\~na Mackenna,
Macul, RM, Chile}
\email[H. Pasten]{hpasten@gmail.com}%

\thanks{J.C. was supported by ANID Doctorado Nacional 21190304. H.P. was supported by FONDECYT Regular grant 1190442. }

\date{\today}
\subjclass[2010]{Primary 11G35; Secondary 11G25, 14G05, 14J20} %
\keywords{Chabauty, Coleman, surfaces, rational points}%
%\dedicatory{}
%11 = NumbTh, 14= AlgGeo
%11G35  	Varieties over global fields
%14J20  	Arithmetic ground fields for surfaces or higher-dimensional varieties
%14G05  	Rational points
%14K20  	Analytic theory of abelian varieties; abelian integrals and differentials
%11G25  	Varieties over finite and local fields

%%%%%%%%%
%%%%%%%%%
%%%%%%%%%
%%%%%%%%%

%%%%%%%%%
%%%%%%%%%
%%%%%%%%%
%%%%%%%%%
\begin{abstract}  Building on work by Chabauty from 1941, Coleman proved in 1985 an explicit bound for the number of rational points of a curve $C$ of genus $g\ge 2$ defined over a number field $F$, with Jacobian of rank at most $g-1$. Namely, in the case $F=\mathbb{Q}$, if $p>2g$ is a prime of good reduction, then the number of rational points of $C$ is at most the number of $\mathbb{F}_p$-points plus a contribution coming from the canonical class of $C$.  We prove a result analogous to Coleman's bound in the case of a hyperbolic surface $X$ over a number field, embedded in an abelian variety $A$ of rank at most one, under suitable conditions on the reduction type at the auxiliary prime. This provides the first extension of Coleman's explicit bound beyond the case of curves. The main innovation in our approach is a new method to study the intersection of a $p$-adic analytic subgroup with a subvariety of $A$ by means of overdetermined systems of differential equations in positive characteristic. 
\end{abstract}

%%%%%%%%%
%%%%%%%%%
%%%%%%%%%
%%%%%%%%%

\maketitle

\setcounter{tocdepth}{1}
\tableofcontents

%%%%%%%%%%%%%%%%%%%%%%%%%%%%%%%%%%%%%%
%%%%%%%%%%%%%%%%%%%%%%%%%%%%%%%%%%%%%%
%%%%%%%%%%%%%%%%%%%%%%%%%%%%%%%%%%%%%%
%%%%%%%%%%%%%%%%%%%%%%%%%%%%%%%%%%%%%%
%%%%%%%%%%%%%%%%%%%%%%%%%%%%%%%%%%%%%%
%%%%%%%%%%%%%%%%%%%%%%%%%%%%%%%%%%%%%%

\section{Introduction}

\subsection{The Chabauty-Coleman bound} Let $C$ be a smooth, geometrically irreducible, projective curve of genus $g\ge 2$ defined over a number field $F$, with Jacobian $J$. In the direction of Mordell's conjecture, Chabauty \cite{Chabauty} proved in 1941 that if $\rk J(F)\le g-1$, then the set $C(F)$ of $F$-rational points of $C$ is finite. Let us recall the main ideas for $F=\Q$: After embedding $C$ into $J$, one has $C(\Q)\subseteq C(\Q_p)\cap \Gamma$ where $p$ is an auxiliary prime and $\Gamma$ is the $p$-adic closure of $J(\Q)$ in $J(\Q_p)$. Then $\dim \Gamma \le \rk J(\Q)<g=\dim J(\Q_p)$ which leads to the finiteness of $C(\Q_p)\cap \Gamma$ because $C(\Q_p)$ is not contained in a lower dimensional $p$-adic analytic subgroup of $J(\Q_p)$.

 Coleman \cite{Coleman} proved in 1985 a celebrated explicit version of Chabauty's theorem, which we recall over $\Q$: With the same conditions, if $p>2g$ is a prime of good reduction for $C$, then 
\begin{equation}\label{EqnColeman}
\#C(\Q)\le  \#C'(\F_p) +2g-2
\end{equation} 
where $C'$ is the reduction of $C$ modulo $p$. While Chabauty's finiteness result was superseded by Faltings's proof of Mordell's conjecture \cite{Faltings1}, the method of Chabauty and Coleman has led to a number of striking developments: Explicit determination of the rational points in suitable curves (see \cite{MPsurvey} for an introduction and references), improvements on \eqref{EqnColeman} such as \cite{LorenziniTucker, MPsurvey, Stollr, KatzZB}, progress towards the Caporaso-Harris-Mazur conjecture  \cite{Stoll, KRZ}, and explicit versions of Kim's non-abelian approach \cite{Kim} such as \cite{BBM, BD1, BDeffChKim, BD2,  BDMTV,  EdixhovenLido}. 

%%%
%%%

\subsection{Beyond curves} Despite all these remarkable developments over the last few decades, the problem of proving a version of \eqref{EqnColeman} for the rational points of a higher dimensional variety $X$ contained in an abelian variety $A$ has remained out of reach. Our main results provide such an extension of \eqref{EqnColeman} when $X$ is a hyperbolic surface contained in an abelian variety $A$ of dimension $n\ge 3$, both defined over a number field $F$, under the assumption $\rk A(F)\le 1$. Although we work over number fields (see Sections \ref{SecMainResult} and \ref{SecApplications}), let us keep the discussion over $\Q$ in this introduction to simplify the exposition. Our results, when applicable,  will imply
\begin{equation}\label{EqnRoughly}
\# X(\Q) \le \# X'(\F_p) + 4p\cdot c_1^2(X)
\end{equation}
where $c_1^2(X)$ is the first Chern number of the surface $X$ (the self-intersection of a canonical divisor),  $p$ is a prime of good reduction satisfying some technical assumptions, and  $X'$ is the reduction of $X$ modulo $p$. The coefficient $4p$ in \eqref{EqnRoughly} is in fact a simplification of a slightly more complicated expression that gives a better estimate. See  Theorem \ref{ThmMainIntro} and Remark \ref{RmkRH4p}. 

We observe that Coleman's bound \eqref{EqnColeman} can be written as
$$
\#C(\Q)\le  \#C'(\F_p) +c_1(C)
$$
where $c_1(C)=2g-2$ is the first Chern number of $C$, i.e. the degree of a canonical divisor. Also, we recall that hyperbolic projective curves are precisely those of genus $g\ge 2$. We hope that these remarks clarify the analogy between our results for surfaces and Coleman's bound for curves.

 Although the \emph{shape} of our bound \eqref{EqnRoughly} is analogous to \eqref{EqnColeman}, the methods of proof are quite different. Let $X$ be a subvariety of an abelian variety $A$ over $\Q$. Let  $\Gamma$ be the $p$-adic closure of the group $A(\Q)$ in $A(\Q_p)$. Then $X(\Q)$ is contained in  $\Gamma \cap X(\Q_p)$ and one tries to bound the latter.
 
 In the case of curves ($X=C$ embedded in $A=J$) Coleman used his theory of $p$-adic integration to construct $p$-adic analytic functions on $C(\Q_p)$ that vanish on $C(\Q_p)\cap \Gamma$ and then he bounded the number of zeros of the relevant one-variable $p$-adic power series on residue disks. 

On the other hand, when $d=\dim X>1$, the analogous approach using $p$-adic analytic functions on $X(\Q_p)$ quickly encounters difficulties. Such functions are locally given by power series in $d$ variables. To get finiteness of $\Gamma\cap X(\Q_p)$ one needs to consider at least $d$ different $p$-adic analytic functions on $X(\Q_p)$ whose zero sets (usually, $p$-adic analytic sets of dimension $d-1$) meet properly, and then there is the problem of giving an upper bound for the number of common zeros. So far, this approach has not succeeded in proving a version of \eqref{EqnColeman} other than in the case of curves.

We take a different route instead. When $\rk A(\Q)=1$ we can give a $p$-adic analytic parametrization of $\Gamma$ by power series in one variable. Upon composing with suitably chosen local equations for the surface $X$ in $A$ and working on residue disks, the problem of bounding $\# \Gamma\cap X(\Q_p)$ is reduced to bounding the number of zeros of a certain power series $h(z)=\sum_j c_jz^j\in \Q_p[[z]]$ on a disk. The key difficulty in doing so (and in the whole approach) it so prove the existence of a small $N$ such that $|c_N|\ge 1$.  We achieve this by developing a method based on overdetermined $\omega$-integrality, which in our case is essentially a study of overdetermined systems of differential equations in positive characteristic. See Section \ref{SecIntroMethods} for a more detailed description.

Prior to our work, the efforts on extending \eqref{EqnColeman} to the higher dimensional setting  have focused on the special case when $A=J$ is the Jacobian of a curve $C$  and $X=W_d$ is the image in $J$ of the $d$-th symmetric power of $C$ via the addition map.  Klassen \cite{Klassen} obtained some partial results for the varieties $W_d$, later improved by Siksek \cite{Siksek}. Although Siksek's work does not give an explicit  bound for $\# W_d(\Q)$ such as \eqref{EqnColeman}, it gives a practical method that in many cases computes the set of rational points of $W_d$.  Park \cite{Park} used tropical geometry to obtain a weak analogue of \eqref{EqnColeman} for $W_d$ (at least if $\rk J(\Q)\le 1$),  but the result turns out to be conditional on an unproved technical assumption as explained in \cite{GM}. Uniform extensions of Park's result are studied in \cite{VW} for $W_2$, but these are also conditional due to the same issue discussed in \cite{GM}.

%%%%%%%%%
%%%%%%%%%
%%%%%%%%%

\subsection{Results}\label{SecIntroResults} Our main result is Theorem \ref{ThmMain}, which  proves an analogue of the Chabauty-Coleman bound  \eqref{EqnColeman} for surfaces over finite extensions of $\Q_p$. For simplicity, let us first state here a version just over $\Q_p$ which follows from the more general Theorem \ref{ThmMain}, see Remark \ref{RmkMainToIntro}. As usual, $K_X$ denotes a canonical divisor of a smooth projective variety $X$. We recall that $X$ is said to be  of general type if $K_X$ is big.  For every field $F$, we choose an algebraic closure and denote it by $F^{alg}$.
\begin{theorem}[Main result, case over $\Q_p$]\label{ThmMainIntro} Let $p$ be a prime. Let $X$ be a smooth, geometrically irreducible, projective surface contained in an abelian variety $A$ of dimension $n\ge 3$, both defined over $\Q_p$ and having good reduction. Let $X'$ and $A'$ be the corresponding reductions modulo $p$. Let $G\le A(\Q_p)$ be a finitely generated group with $\rk G\le 1$ and let $\Gamma$ be its $p$-adic closure in $A(\Q_p)$.  Suppose that either of the following conditions holds:
\begin{itemize}
\item[(i)] $n=3$, $X$ is of general type, $X'$ contains no elliptic curves over $\F_p^{alg}$, and $p>(128/9)c_1^2(X)^2$.  
\item[(ii)] $A'$ is simple over $\F_p^{alg}$, $p>3c_1^2(X)+2$, and there is an ample divisor $H$ on $A$ such that
$$
p> \frac{n!\cdot (3\deg(H^{2}.X) + \deg(H.K_X))^n}{n^n\cdot \deg(H^n)}.
$$
\end{itemize}
Then $X(\Q_p)\cap \Gamma$ is finite and we have
\begin{equation}\label{EqnMainIntro}
 \# X(\Q_p)\cap \Gamma\le \# X'(\F_p) + \frac{p-1}{p-2}\cdot \left(p+4p^{1/2} + 3\right)\cdot c_1^2(X).
 \end{equation}
\end{theorem}
\begin{remark}\label{RmkRHbd} By the Riemann Hypothesis for surfaces over finite fields \cite{Deligne} we have the estimate $|\# X'(\F_p) - (p^2+1)|\le b_3p^{3/2} + b_2 p+b_1p^{1/2}$ where $b_j$ is the $j$-th Betti number of $X(\C)$. In particular, one can use $\# X'(\F_p)\le p^2+b_3p^{3/2} + b_2 p+b_1p^{1/2}+1$  in \eqref{EqnMainIntro} to get a more uniform bound.
\end{remark}
\begin{remark}\label{RmkRH4p} When Theorem \ref{ThmMainIntro} applies, $p\ge 7$ and \eqref{EqnMainIntro} implies $ \# X(\Q_p)\cap \Gamma< \# X'(\F_p) + 4p\cdot c_1^2(X)$. Also, Remark \ref{RmkRHbd} shows  that $\# X'(\F_p)$ is roughly of size $p^2$, say, for large $p$ and fixed Betti numbers of $X$. In this way, $\# X'(\F_p)$ can be seen as the main term in the upper bound \eqref{EqnMainIntro}.
\end{remark}
\begin{remark} As it can be relevant in practical applications, we mention that Proposition \ref{PropKeyU} gives a precise upper bound on the number of points of $\Gamma\cap X(\Q_p)$ in a given residue disk, as well as a criterion for this number to be at most $1$. 
\end{remark}

The following two results on rational points of surfaces are deduced from Theorem \ref{ThmMainIntro} by base change to $\Q_p$ and choosing $G$ as the group of $\Q$-rational points of the corresponding abelian variety. We also obtain similar results over any number field, not just $\Q$; see Section \ref{SecApp1}.
\begin{theorem}\label{ThmIntroA} Let $X$ be a smooth, geometrically irreducible, projective surface of general type contained in an abelian threefold $A$, both defined over $\Q$. Let $p> (128/9)\cdot  c_1^2(X)^2$ be a prime of good reduction for $X$ and $A$, and let $X'$ be the reduction of $X$ modulo $p$.  If $\rk A(\Q)\le 1$ and $X'$ contains no elliptic curves over $\F_p^{alg}$, then $X(\Q)$ is finite and
$$
\#X(\Q) \le \# X'(\F_p) + \frac{p-1}{p-2}\cdot \left(p+4p^{1/2} + 3\right)\cdot c_1^2(X).
$$ 
\end{theorem}
\begin{remark}\label{RmkHyp1} A compact complex manifold $M$ is \emph{hyperbolic} if every holomorphic map $f:\C\to M$ is constant. If a complex projective surface is hyperbolic, then it is of general type (see Lemma \ref{LemmaGenTypeHyp}). So, in Theorem \ref{ThmIntroA} we may require that $X(\C)$ be hyperbolic instead of requiring general type ---depending on the application, this might be easier to check. In fact, there is no loss of generality in doing so: Under the assumptions of Theorem \ref{ThmIntroA}, the surface $X$ contains no elliptic curves over $\Q^{alg}$, hence, over  $\C$ (by specialization). Furthermore, $X$ is not an abelian surface because it is of general type. Since $X\subseteq A$, a result of Green \cite{Green} implies that $X(\C)$ is hyperbolic.
\end{remark}

\begin{theorem}\label{ThmIntroB} Let $X$ be a smooth, geometrically irreducible, projective surface contained in an abelian variety $A$ of dimension $n\ge 3$, both defined over $\Q$.  Let $H$ be an ample divisor on $A$ and let $p$ be a prime of good reduction for $X$ and $A$ satisfying
$$
p> \max\left\{3 c_1^2(X)+2,\frac{n!\cdot (3\deg(H^{2}.X) + \deg(H.K_X))^n}{n^n\cdot \deg(H^n)}\right\}.
$$
Let $X'$ and $A'$ be the corresponding reductions modulo $p$ of $X$ and $A$.  If $\rk A(\Q)\le 1$ and $A'$ is simple over $\F_p^{alg}$, then $X(\Q)$ is finite and
$$
\#X(\Q) \le \# X'(\F_p) + \frac{p-1}{p-2}\cdot \left(p+4p^{1/2} + 3\right)\cdot c_1^2(X).
$$ 
\end{theorem}

\begin{remark} Since $A'$ is geometrically simple, so is $A$. Thus, $X(\C)$ is hyperbolic  by \cite{Green} as it was in Theorem \ref{ThmIntroA} (see Remark \ref{RmkHyp1}). Deep conjectures by Bombieri and Lang predict that if $V$ is a smooth projective variety over $\Q$ such that $V(\C)$ is hyperbolic, then $V(\Q)$ is finite. For curves this is Faltings's theorem since hyperbolic projective curves are precisely those of genus $g\ge 2$. When $V$ is contained in an abelian variety and $V(\C)$ is hyperbolic, finiteness of $V(\Q)$ was proved by Faltings \cite{Faltings2} extending methods of Vojta \cite{VojtaCompact}.  Hence, hyperbolicity of $X(\C)$ is natural in our context.
\end{remark}

\begin{remark}\label{RmkRk} It is expected that the rank of abelian varieties over $\Q$ of a fixed positive dimension is $0$ or $1$ a positive proportion of the time each ---ordering the abelian varieties, for instance, by (Faltings or Theta) height--- and this is proved for elliptic curves \cite{BhargavaShankar, BhargavaSkinner}. Thus, one can expect that the rank assumption in Theorems \ref{ThmIntroA} and \ref{ThmIntroB} is often satisfied in examples.
\end{remark}

\begin{remark} For a variety $X$  contained in an abelian variety $A$ over $\Q$, heuristically, one sees that the limit of applicability of an analogue of Chabauty's classical approach should be $\dim X+ \dim \Gamma \le \dim A$ where $\Gamma$ is the $p$-adic closure of $A(\Q)$ in $A(\Q_p)$. In our results in the case $\dim A=3$, this limit rank condition is in fact reached. 
\end{remark}

A simple case where our results are applicable is given by the following.
\begin{corollary}\label{CoroDim3} Let $A$ be an abelian threefold over $\Q$ with $\rk A(\Q)\le 1$ and $\End (A_\C) = \Z$. There is a set of primes $\Pcal$ of density $1$ in the primes  such that the following holds:

Let $X$ be a smooth, geometrically irreducible, projective surface defined over $\Q$ and contained in $A$, and let $p\in \Pcal$ be a prime of good reduction for $X$ with $p>(128/9) \cdot c_1^2(X)^2$. Let $X'$ be the reduction of $X$ modulo $p$. Then
$$
\#X(\Q) \le \# X'(\F_p) + \frac{p-1}{p-2}\cdot \left(p+4p^{1/2} + 3\right)\cdot c_1^2(X).
$$
\end{corollary}

This is deduced in Section \ref{SecApp2} from Theorem \ref{ThmIntroA} and results of Chavdarov \cite{Chavdarov} on absolutely simple reduction of abelian varieties. Chavdarov's results together with Theorem \ref{ThmIntroB} imply analogous corollaries when $\dim A\ge 3$ is odd. Further extensions are discussed in Section \ref{SecApp2}.

\begin{remark} Given $n\ge 1$, a general abelian variety $B$ over $\C$ of  dimension $n$ satisfies $\End(B)\simeq \Z$. In view of Remark \ref{RmkRk}, abelian threefolds satisfying the conditions in Corollary \ref{CoroDim3} should be rather common. And in fact, they are easy to find; the Jacobian of the genus $3$ hyperelliptic curve $y^2=x^7-x-1$ is such an example with Mordell-Weil rank $1$.
\end{remark}

\begin{remark} For any abelian threefold as in Corollary \ref{CoroDim3}, our bounds for the number of rational points apply to \emph{any} smooth surface contained in $A$, e.g. by embedding $A$ in a projective space and intersecting with a general hyperplane (by Bertini's theorem). This gives plenty of examples.
\end{remark}

For a curve $C$ over a field, we let $C^{(n)}$ be its $n$-th symmetric power. If $C$ is defined over $\Q$, the $\Q$-rational points of $C^{(2)}$ are in bijection with Galois orbits of quadratic points and unordered pairs of $\Q$-rational points. If $C$ is a hyperelliptic curve over $\Q$, then we certainly have that $C^{(2)}(\Q)$ is infinite. As a direct application of our results, we can bound $\# C^{(2)}(\Q)$ for non-hyperelliptic curves  whose Jacobian has rank $0$ or  $1$, under some conditions on the reduction type at $p$. 

\begin{corollary}\label{CoroQuadratic} Let $C$ be a smooth, geometrically irreducible, projective curve over $\Q$ of genus $g\ge 3$ which is not hyperelliptic over $\Q^{alg}$ and such that its Jacobian $J$ has $\rk J(\Q)\le 1$.  Let  $p>(8g-10)^g$ be a prime of good reduction for $C$. Let $C'$ and $J'$ denote the reduction of $C$ and $J$ modulo $p$ respectively. Suppose that $C'$ is not hyperelliptic over $\F_p^{alg}$ and that $J'$ is geometrically simple. Then  $C^{(2)}(\Q)$ is finite and
$$
\# C^{(2)}(\Q)\le \# (C')^{(2)}(\F_p) +  \frac{p-1}{p-2}\cdot \left(p+4p^{1/2} + 3\right)\cdot (4g-9)(g-1).
$$
\end{corollary}

This is directly obtained from Theorem \ref{ThmIntroB} in Section \ref{SecApp3} after some computations in intersection theory. In a similar fashion, we will prove the following strengthening of Corollary \ref{CoroQuadratic} for curves of genus $3$, by applying Theorem \ref{ThmIntroA} instead.

\begin{corollary}\label{CoroQuadratic3} Let $C$ be a smooth, geometrically irreducible, projective curve over $\Q$ of genus $3$ which is not hyperelliptic over $\Q^{alg}$ and such that its Jacobian $J$ has $\rk J(\Q)\le 1$.  Let  $p\ge 521$ be a prime of good reduction for $C$ and denote by $C'$ the reduction of $C$ modulo $p$. Suppose that $C'$ is not hyperelliptic over $\F_p^{alg}$ and that $(C')^{(2)}$ does not contain elliptic curves over $\F_p^{alg}$. Then  $C^{(2)}(\Q)$  is finite and
$$
\# C^{(2)}(\Q)\le \# (C')^{(2)}(\F_p) +  6\cdot \frac{p-1}{p-2}\cdot \left(p+4p^{1/2} + 3\right) <  \# (C')^{(2)}(\F_p) + 7.1\cdot p.
$$
\end{corollary}

Finally, we mention that the finiteness aspect of Theorem \ref{ThmMainIntro} (and more generally, Theorem \ref{ThmMain}) does not follow from Faltings's theorem for subvarieties of abelian varieties \cite{Faltings2}, as $\Gamma$ is not a finite rank group when $\rk(G)=1$.  Regarding bounds for the number of rational points, the Diophantine approximation method of Vojta \cite{VojtaCompact} and Faltings \cite{Faltings2, Faltings3} led to explicit bounds such as \cite{Remond1, Remond2} for subvarieties of abelian varieties over number fields, outside the special set. However, as it is the case for the classical Chabauty-Coleman method on curves compared to Diophantine approximation bounds, our $p$-adic approach for surfaces leads to sharper estimates when it applies. 

%%%%%%%%%
%%%%%%%%%
%%%%%%%%%

\subsection{Sketch of the method: overdetermined $\omega$-integrality}\label{SecIntroMethods} We conclude this introduction by presenting an outline of the proof of Theorem \ref{ThmMainIntro} (and more generally, Theorem \ref{ThmMain}). To simplify the notation, let us focus in the case $\dim A=3$ since the key features already appear here. Furthermore, enlarging $G$ we may assume that $\rk G=1$.

First we note that, at least heuristically, the Chabauty-Coleman $p$-adic approach has a chance to succeed since $\dim X+ \dim \Gamma= 2+1=\dim A$.

Consider the reduction map $\red:A(\Q_p)\to A'(\F_p)$ and for each $x\in A'(\F_p)$ let  $U_x=\red^{-1}(x)\subseteq A(\Q_p)$ be the corresponding residue disk. We want to bound $\# \Gamma\cap X(\Q_p)\cap U_x$ and then add these upper bounds as $x$ varies in $X'(\F_p)$. Let us parametrize $\Gamma\cap U_x$ by a $p$-adic analytic map $\gamma : p\Z_p\to U_x\subseteq A(\Q_p)$. If $f$ is a local equation for $X$ in $U_x$, then $\#\Gamma\cap X(\Q_p)\cap U_x$ is the number of zeros of the one-variable $p$-adic analytic function $h=f\circ \gamma$ on $p\Z_p$. We remark that the idea of parametrizing $\Gamma$ can be traced back to work of Flynn \cite{Flyn} and it has been successful in computing the rational points of curves in examples, although it has not previously led to results such as \eqref{EqnColeman}.

Writing $h(z)=c_1z+c_2z^2+...\in \Q_p[[z]]$, the number of zeros can be estimated  \emph{provided} that we have some information on the size of the coefficients $c_j$, see Section \ref{SecZeros}. Namely,  we need:
\begin{itemize} 
\item[(I)] a good upper bound for $|c_j|$ for all $j$, and 
\item[(II)] some small $N$ such that $|c_N|$ is not too small, say, $|c_N|\ge 1$.
\end{itemize}
To achieve (I) we perform the construction of the $p$-adic analytic map $\gamma$ very carefully. We develop a completely explicit theory of $1$-parameter $p$-adic analytic subgroups in Section \ref{Sec1PS} and, with some care in the choice of the local equation $f$, this allows us to prove the desired upper bound. 

The key difficulty in the whole argument, however, is (II). We take a somewhat indirect approach.

Consider the morphism of $p$-adic analytic groups $\Log:A(\Q_p)\to \Lie(A(\Q_p))\simeq H^0(A,\Omega^1_{A/\Q_p})^\vee$ constructed by Coleman integration or by classical theory of $p$-adic Lie groups \cite{Bourbaki}. As $\rk G=1$,  $\Log(\Gamma)$ is contained in a line of $\Lie(A(\Q_p))$ which determines a hyperplane $\Hcal\subseteq H^0(A,\Omega^1_{A/\Q_p})$. We choose $\omega_1,\omega_2\in \Hcal$ so that they reduce to independent differentials $\omega'_1,\omega'_2$ on $A'$ which we restrict to differentials $u_1,u_2$ on $X'$. For $x\in X'(\F_p)$ let $m(x)$ be the supremum of all integers $m$ such that there is a closed immersion $\phi_m:\Spec \F_p^{alg}[z]/(z^{m+1})\to X'$ supported at $x$ which is $\omega$-integral for both $\omega=u_1,u_2$. Roughly speaking, $\omega$-integrality for a differential $\omega$ means that the morphism $\phi_m$ solves the differential equation determined by $\omega$; see Section \ref{Secwint} for a detailed discussion.

If $u_1$ and $u_2$ are independent over the function field of  $X'$, then the maps $\phi_m$ in the definition of $m(x)$ are jet solutions to an \emph{overdetermined system of differential equations}. So, one expects $m(x)$ to be finite and bounded in terms of $u_1$ and $u_2$. This is indeed correct, but it is far from obvious. Theorem \ref{ThmOver} gives such a bound in terms of the geometry of the divisor $D$ of the $2$-form $u_1\wedge u_2$. Our bound for $m(x)$ is somewhat intricate, but Remark \ref{RmkSharp} shows that in some cases it is sharp.

Bounding $m(x)$ turns out to be crucial, since we will show in Lemma \ref{LemmaBdNm} that $N\le m(x)+1$ with $N$ as in (II). Together with the zero estimates from Section \ref{SecZeros} and our upper bounds for $|c_j|$ (see (I) above) we will prove the following key estimate in Proposition \ref{PropKeyU} 
\begin{equation}\label{EqnIntroU}
\#\Gamma\cap X(\Q_p)\cap U_x\le 1+ m(x)\cdot\frac{p-1}{p-2}.
\end{equation}

Finally, the geometric bound for $m(x)$ proved in Theorem \ref{ThmOver} is applied to \eqref{EqnIntroU}, and then added over $x\in X'(\F_p)$. When $x$ is not in the support of $D=\divi(u_1\wedge u_2)$ we show $m(x)=0$, thus, $\#\Gamma\cap X(\Q_p)\cap U_x\le 1$.  The contribution for $x$ in the support of $D$ is more complicated and it corresponds to counting $\F_p$-points in the support of  $D$ with suitable weights. The divisor $D$ can be non-reduced and highly singular, so this counting problem does not directly follow from Weil's estimates for points in curves over finite fields. We address this problem by studying certain modified Zeta functions in Section \ref{SecZeta}, which allows us to conclude the argument.

A technical point that requires attention is that we make heavy use of $\omega$-integrality for schemes over  rings or fields of positive characteristic. The classical analytic approach is not useful for us and we need purely algebraic methods. Fortunately such a study was carried out by Garcia-Fritz \cite{GFthesis, GFtams, GFijnt} in the context of her generalization of Vojta's explicit version of Bogomolov's approach to quasi-hyperbolicity, see \cite{VojtaBuchi, Deschamps}. While Garcia-Fritz's work is mostly over $\C$, her algebraic methods easily extend to an arbitrary base. See Section \ref{Secwint}. 

A major technical difficulty is that at various points of the argument we need to restrict differential forms to subvarieties of abelian varieties in a way that preserves non-triviality. For example, while $\omega'_1$ and $\omega'_2$ are independent on $A'$, is is not clear $u_1\wedge u_2$ is not the zero form on $X'$, and this is necessary even to define the divisor $D=\divi(u_1\wedge u_2)$. Part of the difficulties are due to the fact that $X$ does not have a particularly convenient presentation that allows for explicit computations with differentials, unlike the case of symmetric powers of curves. The required non-vanishing results for restriction of forms are not difficult to obtain in characteristic zero by analytic means, but we need them in positive characteristic. This is achieved in Section \ref{SecRestrictions} by means of intersection theory.

%%%%%%%%%%%%%%%%%%%%%%%%%%%%%%%%%%%%%%
%%%%%%%%%%%%%%%%%%%%%%%%%%%%%%%%%%%%%%
%%%%%%%%%%%%%%%%%%%%%%%%%%%%%%%%%%%%%%
%%%%%%%%%%%%%%%%%%%%%%%%%%%%%%%%%%%%%%
%%%%%%%%%%%%%%%%%%%%%%%%%%%%%%%%%%%%%%
%%%%%%%%%%%%%%%%%%%%%%%%%%%%%%%%%%%%%%
%%%%%%%%%%%%%%%%%%%%%%%%%%%%%%%%%%%%%%
%%%%%%%%%%%%%%%%%%%%%%%%%%%%%%%%%%%%%%
%%%%%%%%%%%%%%%%%%%%%%%%%%%%%%%%%%%%%%

\section{Notation}

\subsection{General notation}\label{NotationFields} $L$ is a field and $L^{alg}$ is an algebraic closure of it. The symbol $F$ will denote a field, a function on a variety, or a power series, and this will be explicitly stated.  The symbol $K$ will denote a $p$-adic field or a canonical divisor in a variety, which will be explicitly stated.  $k$ will always denote an algebraically closed field and $\kk$ will always denote a finite field.

We write $\N=\Z_{\ge 0}$. For $\alpha=(\alpha_1,...,\alpha_n)\in \N$ we define $\|\alpha\|=\alpha_1+...+\alpha_n$. If $\xx=(x_1,...,x_n)$ is an $n$-tuple of variables, we write $\xx^{\alpha}=x_1^{\alpha_1}\cdots x_n^{\alpha_n}$.

The exponential function on $\R$ is $\exp$ and the logarithm on $\R_{>0}$ is $\log$.
\subsection{Geometry}\label{notationGeo} A variety over $L$ is a reduced $L$-scheme which is separated of finte type over $L$. We remark that we are not requiring irreducibility as part of the definition. A curve (resp. surface) over $L$ is a variety over $L$ of pure dimension $1$ (resp. 2). 

The normalization of an integral domain $A$ is denoted by $\widetilde{A}$, and the normalization of a variety $Z$ over $k$ is denoted by $\widetilde{Z}$. 

If $X$ is a scheme and $x$ is a (schematic) point in $X$, the maximal ideal of the local ring $\Ocal_{X,x}$ is denoted by $\mfrak_{X,x}$. The completion of $\Ocal_{X,x}$ at $\mfrak_{X,x}$ is denoted by $\widehat{\Ocal}_{X,x}$. 

If $f:X\to Y$ is a morphism of schemes, $x$ is a point of $X$ and $y=f(x)$, the induced map on local rings is $f^\#_x : \Ocal_{Y,y}\to \Ocal_{X,x}$. If $X$ consists of exactly one point, we just write $f^\#$ instead of $f^\#_x$.

For an irreducible projective curve $C$ over the algebraically closed field $k$, the arithmetic genus of $C$ is denoted by $\gfrak_a(C)$ and the geometric genus is denoted by $\gfrak_g(C)$. Namely, $\gfrak_g(C)=\gfrak_a(\widetilde{C})$. 

If $L$ is a perfect field and $C$ is a geometrically irreducible projective curve over $L$, we define the arithmetic and geometric genus of $C$ by base change to $L^{alg}$.

Given a geometrically irreducible, smooth, projective surface $S$ over $L$, we have a $\Z$-valued intersection product on divisors of $S$, which we write $(D.D')$, $(D.D')_S$, or $D.D'$ for two divisors $D$ and $D'$ on $S$, and we write $D^2=(D.D)$ for a self-intersection. The first Chern number $c_1^2(S)$ of $S$ is the self-intersection number of a canonical divisor on $S$. 

%%%
%%%
\subsection{Local fields}\label{NotationLoc} In the context of local fields, the symbol $p$ denotes a prime number, $K$ is a finite extension of $\Q_p$, $R$ is the ring of integers of $K$, $\pfrak$ is the maximal ideal of $R$, $\varpi\in \pfrak$ is a uniformizer, $\kk=R/\pfrak$ is the residue field of $R$, $q=\# \kk$, $\efrak$ is the ramification exponent of $K/\Q_p$, and $\ffrak=[\kk : \F_p]$ is the residue degree. We denote by $|-|$ the absolute value on $K$ normalized by $|\varpi|=1/q$, or equivalently, $|p|=p^{-[K:\Q_p]}$.

For $n\ge 1$ we put the following norm on $K^n$:
$$
|\vv | = \max_{1\le j\le n}|v_j|, \quad \vv =(v_1,...,v_n)\in K^n.
$$
For $r>0$ we define the balls $B_n(r)=\{\vv \in K^n : |\vv |<r\}$ and  $B_n[r]=\{\vv \in K^n : |\vv |\le r\}$.

%%%%%%%%%%%%%%%%%%%%%%%%%%%%%%%%%%%%%%
%%%%%%%%%%%%%%%%%%%%%%%%%%%%%%%%%%%%%%
%%%%%%%%%%%%%%%%%%%%%%%%%%%%%%%%%%%%%%
%%%%%%%%%%%%%%%%%%%%%%%%%%%%%%%%%%%%%%
%%%%%%%%%%%%%%%%%%%%%%%%%%%%%%%%%%%%%%
%%%%%%%%%%%%%%%%%%%%%%%%%%%%%%%%%%%%%%
%%%%%%%%%%%%%%%%%%%%%%%%%%%%%%%%%%%%%%
%%%%%%%%%%%%%%%%%%%%%%%%%%%%%%%%%%%%%%
%%%%%%%%%%%%%%%%%%%%%%%%%%%%%%%%%%%%%%

%%%%%%%%%%%%%%%%%%%%%%%%%%%%%%%%%%%%%%
%%%%%%%%%%%%%%%%%%%%%%%%%%%%%%%%%%%%%%
%%%%%%%%%%%%%%%%%%%%%%%%%%%%%%%%%%%%%%
%%%%%%%%%%%%%%%%%%%%%%%%%%%%%%%%%%%%%%
%%%%%%%%%%%%%%%%%%%%%%%%%%%%%%%%%%%%%%
%%%%%%%%%%%%%%%%%%%%%%%%%%%%%%%%%%%%%%

\section{Curves and surfaces} 

%%%%
%%%%
%%%%
\subsection{Curves} Let $C$ be an irreducible curve over the algebraically closed field $k$ and let $x\in C(k)$. The \emph{order of singularity} of $C$ at $x$ is $\delta(C,x)=\dim_k(\widetilde{\Ocal}_{C,x}/\Ocal_{C,x})$, where the quotient $\widetilde{\Ocal}_{C,x}/\Ocal_{C,x}$ is taken as $k$-vector spaces. One has the following basic fact:
\begin{lemma} $\delta(C,x)$ is a non-negative integer. It is equal to $0$ if and only if $C$ is regular at $x$.
\end{lemma}

As a special case of Theorem 2 in \cite{Hironaka} we have

\begin{lemma}[Hironaka]\label{LemmaHironakaGenus} If $C$ is a projective irreducible curve over $k$, then
$$
\gfrak_a(C)=\gfrak_g(C) + \sum_{x\in C(k)} \delta(C,x).
$$

\end{lemma}

The \emph{multiplicity} of $C$ at  $x\in C(k)$ is defined as the multiplicity of the local ring $\Ocal_{C,x}$ as defined, for instance, in Ch.  5, Sec. 14 of \cite{Matsumura}. We denote it by $\mult_x(C)$. A basic fact:
\begin{lemma} $\mult_x(C)$ is a positive integer. It is equal to $1$ if and only if $C$ is regular at $x$.
\end{lemma}

From Exercise 14.5 in \cite{Matsumura} we get a simple expression for $\mult_x(C)$ in a special case.

\begin{lemma}\label{LemmaMultFormula} Let $S$ be a surface over $k$ and let $C$ be a closed irreducible curve in $S$. Let $x\in C(k)$ be a point such that $S$ is regular at $x$. Let $F\in \Ocal_{S,x}$ be a local equation for $C$ at $x$. Then 
$$
\mult_x(C)=\max\{n\ge 1 : F\in \mfrak_{S,x}^n\}.
$$
\end{lemma}

We will also need the following result.
\begin{lemma}[Hironaka] \label{LemmaHironakaMult} Let $S$ be a projective surface and let $C$ be a projective irreducible curve contained in $S$. Let $x\in C(k)$ be a point such that $S$ is regular at $x$. Then we have
$$
\delta(C,x)\ge \frac{1}{2}\cdot \mult_x(C)\cdot (\mult_x(C)-1).
$$
\end{lemma}
\begin{proof} By Theorem 1 in \cite{Hironaka}, after dropping the contribution of higher order neighboring points.
\end{proof}
%%
%%

%%%%
%%%%
%%%%
\subsection{A canonical genus bound} 

The following result will be useful to control the genus of the components of certain canonical divisors on surfaces. 

\begin{lemma}[Canonical genus bound] \label{LemmaCanonicalGenusBd} Let $S$ be a smooth projective irreducible surface over $k$. Let $D$ be a canonical divisor on $S$. Assume that $D>0$ and that $D$ is numerically effective. Let $C_1,...,C_\ell$ be the irreducible components of the support of $D$ and let $a_1,...,a_\ell$ be positive integers defined by $D=a_1C_1 +...+a_\ell C_\ell$. Then we have
$$
\sum_{j=1}^\ell a_j\cdot (\gfrak_a(C_j)-1)\le c_1^2(S).
$$
\end{lemma}
\begin{proof} Let us rearrange the indices so that for certain $0\le r\le \ell$ we have $C_j^2>0$ for each $j\le r$ and $C_j^2\le 0$ for each $j>r$. By the adjunction formula we find
\begin{equation} \label{EqnCanBd1} 
\begin{aligned}
2\sum_{j=1}^\ell a_j\cdot (\gfrak_a(C_j)-1)&=\sum_{j=1}^\ell a_j\cdot ((D.C_j) +C_j^2) = D^2 + \sum_{j=1}^\ell a_j\cdot  C_j^2\\
&\le D^2 + \sum_{j=1}^r a_j \cdot C_j^2 \le D^2 + \sum_{j=1}^r a_j^2 \cdot C_j^2.
\end{aligned}
\end{equation}
For any given $j$ we have $(D-a_jC_j).C_j=(\sum_{i\ne j} a_iC_i).C_j\ge 0$ since the $C_i$ are different irreducible curves. Also, $D.C_j\ge 0$ since $D$ is nef. Therefore
$$
\begin{aligned}
D^2=\sum_{j=1}^\ell a_j D.C_j = \sum_{j=1}^r a_j^2 \cdot C_j^2 + \sum_{j=1}^r  a_j\cdot (D-a_jC_j).C_j + \sum_{j=r+1}^\ell a_j D.C_j\ge \sum_{j=1}^r a_j^2 \cdot C_j^2 .
\end{aligned}
$$
Together with \eqref{EqnCanBd1}, this gives $2\sum_{j=1}^\ell a_j\cdot (\gfrak_a(C_j)-1)\le 2\cdot D^2=2c_1^2(X)$.
\end{proof}
%%
%%

%%%%%
%%%%%
%%%%%
%%%%%

\subsection{Branches} Let $S$ be a surface over $k$ and let $x\in S(k)$ be a smooth point of $S$. Given local parameters $s_1,s_2\in \mfrak_{S,x}$, the completed local ring at $x$ can be presented as $\widehat{\Ocal}_{S,x}\simeq  k[[s_1,s_2]]$.

 Let $C\subseteq S$ be a closed irreducible curve passing through $x$ with local equation $F\in \mfrak_{S,x}$. Let us factor $F=F_1\cdots F_r$ where each $F_j\in \widehat{\Ocal}_{S,x}$ is irreducible. 
\begin{lemma} For each $i\ne j$, the power series $F_i$ and $F_j$ are non-associated.
\end{lemma} 
\begin{proof} The special case $\Ocal_{S,x}\simeq k[s_1,s_2]_{(s_1,s_2)}$ is in \cite{KunzBranches}  Theorem 16.5.  As observed in \cite{GFthesis} Remark 2.58, the general case follows from a result of Chevalley. Indeed, $k$ is perfect since it is algebraically closed. By \cite{ZariskiSamuel2} VIII.13 Theorem 31, the local ring $\Ocal_{C,x}=\Ocal_{S,x}/(F)$ is analytically unramified. Thus, the completion $\widehat{\Ocal}_{C,x} = \widehat{\Ocal}_{S,x}/(F)$ does not have nilpotent elements.
\end{proof}
Hence, the prime ideals $\qfrak_j=(F_j)\subseteq \widehat{\Ocal}_{S,x}$ are different. They will be called \emph{branches} of $C$ at $x$.

Let $\nu: \widetilde{C}\to S$ be the composition of the normalizaton map $\widetilde{C}\to C$ and the inclusion $C\to S$. For each $y\in \nu^{-1}(x)$ we have a map induced on local rings $\nu^\#_y: \Ocal_{S,x}\to \Ocal_{\widetilde{C},y}$ which extends continuously to completed local rings $\widehat{\nu}^\#_y: \widehat{\Ocal}_{S,x}\to \widehat{\Ocal}_{\widetilde{C},y}$. We will need the following:

\begin{lemma}[Relation between normalization and branches] \label{LemmaBranches} With the previous notation, there is a bijection between the set of branches of $C$ at $x$ and the points $y\in \nu^{-1}(x)$. More precisely, the ideals $\ker(\widehat{\nu}^\#_y)\subseteq \widehat{\Ocal}_{S,x}$ for $y\in \nu^{-1}(x)$ are exactly the branches of $C$ at $x$.
\end{lemma}
\begin{proof} In the special case $\Ocal_{S,x}\simeq k[s_1,s_2]_{(s_1,s_2)}$ this is \cite{KunzBranches} Theorem 16.14; a similar argument works in general. Alternatively, the general case is \cite{GFthesis} Theorem 2.69 assuming that $k$ has characteristic $0$. This assumption is needed in other parts of \cite{GFthesis}, but not  for this particular result. In fact, the argument in \cite{GFthesis} reduces this result to \cite{DKatz} Corollary 5, where characteristic $0$ is not required.
\end{proof}
\begin{lemma}[Bound for fibres of normalization] \label{LemmaBounddelta} With the previous notation, and assuming that $S$ is a projective surface, we have $\#\nu^{-1}(x)\le \delta(C,x)+1$.
\end{lemma}
\begin{proof} From Lemma \ref{LemmaMultFormula}, it follows that the number of branches of $C$ at $x$ is at most $\mult_x(C)$. By Lemma \ref{LemmaBranches} we get $\nu^{-1}(x)\le \mult_x(C)$. The result follows from Lemma \ref{LemmaHironakaMult}.
\end{proof}
As the isomorphism $\widehat{\Ocal}_{S,x}\simeq  k[[s_1,s_2]]$ allows us to study branches by means of power series computations, the following result will be useful. It is a direct consequence of \cite{Ploski} Theorem 2.1.
\begin{lemma} \label{LemmaPloski} Let $f(x_1,x_2)\in k[[x_1,x_2]]$ be an irreducible power series. Let $\phi_1(t),\phi_2(t)\in t\cdot k[[t]]$ be such that $f(\phi_1,\phi_2)=0$. Then $\ord_t\phi_1(t)\ge \ord_{x_2} f(0,x_2)$ and $\ord_t\phi_2(t)\ge \ord_{x_1} f(x_1,0)$.
\end{lemma}
%%
%%

%%%%%
%%%%%
%%%%%
%%%%%

\subsection{Surfaces of general type} Let $S$ be a smooth, irreducible, projective surface over $k$.  The surface $S$ is of \emph{general type} if the canonical sheaf $\Omega^2_{S/k}$ is big, i.e. if $\dim_kH^0(S,(\Omega^2_{S/k})^{\otimes n})$ grows quadratically on $n$. More generally, if $X$ is a smooth, geometrically irreducible, projective surface over $L$, we say that $X$ is of general type if $X\otimes L^{alg}$ is a surface of general type over $L^{alg}$.

The following result is classical. Here, ``minimal'' means that $S$ contains no $(-1)$-curves (i.e. smooth rational curves with self-intersection $-1$).

\begin{lemma} \label{LemmaNef} If $S$ is a minimal smooth projective surface of general type over $k$, then $c_1^2(S)\ge 1$ and the canonical sheaf $\Omega^2_{S/k}$ is nef. 
\end{lemma}

The following simple condition ensures that the canonical sheaf of a general type surface is ample.

\begin{lemma}\label{LemmaGTA} Let $S$ be a smooth projective surface of general type. If $S$ contains no smooth rational curves with self-intersection $-1$ or $-2$, then the canonical sheaf of $S$ is ample.  In particular, this is the case when $S$ contains no smooth curve of genus $0$.
\end{lemma}
\begin{proof} This result is well-known. The argument is similar to that in \cite{MumfordApp}. We present the proof for the convenience of the reader.

 Let $K$ be a canonical divisor on $S$. Since $S$ does not contain $(-1)$-curves, it is minimal. Lemma \ref{LemmaNef} implies that $K^2>0$ and that $K$ is nef. Let $C\subseteq X$ be an irreducible curve. By the Nakai-Moishezon criterion, it suffices to show that $K.C\ne 0$. 

For the sake of contradiction, suppose that $K.C=0$. As $K^2>0$, the Hodge Index Theorem gives that either $C^2<0$ or $C$ is numerically trivial. The latter case cannot happen because $C.H>0$ for any ample divisor $H$, so we have $C^2<0$. By the adjunction formula $2\gfrak_a(C)-2 = C^2 +K\cdot C = C^2< 0$, which gives $\gfrak_a(C) = 0$. Thus $C$ is smooth and rational by Lemma \ref{LemmaHironakaGenus}. Furthermore, $C^2=C^2 +K\cdot C=2\gfrak_a(C)-2=-2$. Contradiction. So, $K$ is ample.
\end{proof}
Let us recall that a compact complex manifold $M$ is Brody hyperbolic if every holomorphic map $h:\C\to M$ is constant. As we are in the compact case, this is equivalent to Kobayashi hyperbolicity by Brody's theorem \cite{Brody}, and we simply say that $M$ is hyperbolic. It is conjectured that if $V$ is a smooth projective variety over $\C$ and  $V(\C)$ is hyperbolic, then $V$ is of general type. This is known for surfaces; see \cite{BHPV} Chapter VIII, Theorem 24.1.
\begin{lemma}\label{LemmaGenTypeHyp} Let $S$ be a smooth, irreducible,  projective surface defined over $\C$. If $S(\C)$ is hyperbolic, then $S$ is of general type.
\end{lemma}

\begin{lemma}\label{LemmaGenTypeTranslate} Let $A$ be an abelian variety over $\C$ and let $X\subseteq A$ be a smooth projective surface. If $X$ is not an abelian surface and it does not contain elliptic curves, then $X$ is of general type.
\end{lemma}
\begin{proof}
By \cite{Yau} Theorems 1 and 1'. Alternatively:  $X$ is hyperbolic \cite{Green}, then apply  Lemma \ref{LemmaGenTypeHyp}.
\end{proof}

%%%%%
%%%%%
%%%%%
%%%%%

\subsection{Point counting over finite fields} \label{SecZeta}

In this paragraph we let $\kk$ be a finite field of characteristic $p$ with $q$ elements. Let $D$ be a projective curve defined over $\kk$; we do not require that $D$ be irreducible, connected, or smooth. Let $\kk'$ be a finite extension of $\kk$ such that all geometric irreducible components of $D$ are defined over $\kk'$, and all singular points of $D$ as well as all points of the normalization of $D\otimes \kk'$ mapping to them are $\kk'$-rational.

Let $C_1,...,C_r$ be the geometric irreducible components of $D$, which are defined over $\kk'$. Let $\nu_j:\widetilde{C}_j\to D\otimes \kk'$ be the normalization of $C_j$ composed with the inclusion $C_j\to D\otimes \kk'$. 

Our goal is to show the following estimate

\begin{lemma}\label{LemmaWeilBound} With the previous notation, we have
$$
\sum_{x\in D(\kk)} \sum_{j=1}^r \#\{y\in \widetilde{C}_j(\kk^{alg}) : \nu_j(y)=x\} \le (q+1)r + 2q^{1/2}\cdot \sum_{j=1}^r \gfrak_g(C_j).
$$
\end{lemma}
More generally, for a finite extension $L/\kk$ it will be convenient to define 
$$
A_D(L)=\sum_{x\in D(L)} \sum_{j=1}^r \#\{y\in \widetilde{C}_j(\kk^{alg}) : \nu_j(y)=x\}
$$
and $a_D(L)=A_D(L)-\#D(L)$. Note that $a_D(L)\ge 0$ and by the definition of $\kk'$ we see that the number $c_D:=a_D(\kk')$ satisfies $a_D(L)\le c_D$ for each finite extension $L/\kk$.

\begin{lemma}\label{LemmaLbig} If $L/\kk$ is a finite extension and $\kk'\subseteq L$, then $a_D(L)=c_D$ and
$$
A_D(L)=\sum_{j=1}^r \#\widetilde{C}_j(L).
$$
\end{lemma}
\begin{proof} The quantity $a_D(L)$ counts preimages of singularities of $D$ in the normalization, so it stabilizes when $L$ contains $\kk'$. Let $C$ be the disjoint union of the curves $C_j$ and let $\nu:C\to D\otimes \kk'$ be the morphism over $\kk'$ determined by all the maps $\nu_j$. Then $\nu$ is an isomorphism over $\kk'$ away from the preimages of singularities of $D$, all of which are $\kk'$-rational. Thus,  $A_D(L)=\#C(L)$.
\end{proof}

Note that $\#D(\kk)\le A_D(\kk) \le \# D(\kk) + c_D$. There is related work on the Weil conjectures for singular curves such as \cite{AubryPerret} and a natural approach to Lemma \ref{LemmaWeilBound} would be to use an upper bound on $\#D(\kk)$ together with an upper bound on $c_D$ coming from counting singularities. However, we will consider $\#D(L)+c_D$ at once, as this leads to sharper estimates.

For an integer $n\ge 1$ and $F$ a finite field of characteristic $p$, let $F_n\subseteq F^{alg}$ be the only extension of $F$ of degree $n$. Given a quasi-projective variety $V$ over $F$ (not necessarily smooth or irreducible) let $N_{V,F}(n)=\# V(F_n)$ and define the Zeta function on the variable $T$ by
$$
Z_{V,F}(T)=\exp\left(\sum_{n=1}^\infty \frac{N_{V,F}(n)}{n}\cdot T^n\right).
$$
In this generality, it is a theorem of Dwork \cite{Dwork} that $Z_{V,F}(T)\in \Q(T)$.

Let $N^*_{D,\kk}(n)=\# D(\kk_n)+c_D$ and observe that for all $n\ge 1$ we have 
\begin{equation}\label{EqnBdAD}
A_D(\kk_n)\le N^*_{D,\kk}(n). 
\end{equation}
Let us define
$$
Z^*_{D,\kk}(T)=\exp\left(\sum_{n=1}^\infty \frac{N^*_{D,\kk}(n)}{n}\cdot T^n\right).
$$
Then $Z^*_{D,\kk}(T)=Z_{D,\kk}(T)/(1-T)^{c_D}\in \Q(T)$. Writing $d=[\kk':\kk]$, we have
\begin{lemma}\label{LemmaZprod} Let $\mu_d$ be the set of complex $d$-th roots of $1$. Then
$$
\prod_{\zeta\in \mu_d} Z^*_{D,\kk}(\zeta\cdot T) = \prod_{j=1}^r Z_{\widetilde{C}_j,\kk'}(T^d).
$$
\end{lemma}
\begin{proof} We note that for a given $n\ge 1$
$$
S_n:=\sum_{\zeta\in \mu_d}\frac{N^*_{D,\kk}(n)}{n}\cdot \zeta^nT^n=
\begin{cases}
0 & \mbox{ if } d\nmid n\\
d\cdot \frac{N^*_{D,\kk}(n)}{n}\cdot T^m & \mbox{ if } d|n.
\end{cases}
$$
In the second case we write $n=md$ and observe that $S_{md} = m^{-1}N^*_{D,\kk}(md) T^{md}$. Using $\kk_{md}=\kk'_m\supseteq \kk'$ and Lemma \ref{LemmaLbig} we get
$$
N^*_{D,\kk}(md) =\# D(\kk'_m) + c_D=\# D(\kk'_m)+ a_D(\kk'_m)=A_D(\kk'_m)=\sum_{j=1}^r N_{\widetilde{C}_j,\kk'}(m)
$$
which proves the result.
\end{proof}

\begin{proof}[Proof of Lemma \ref{LemmaWeilBound}] The Weil conjectures for curves (proved by Weil) applied to the curves $\widetilde{C}_j$ over $\kk'$ together with Lemma \ref{LemmaZprod} show that $Z^*_{D,\kk}(T)$ on $\C$ has $r$ poles of modulus $q$, a total of $2\cdot \sum_{j=1}^r \gfrak_g(C_j)$ zeros of modulus $q^{1/2}$, $r$ poles of modulus $1$, and no other zeros or poles in $\C$. A standard computation taking the logarithmic derivative of $Z^*_{D,\kk}(T)$, comparing coefficients of $T^n$, and applying the triangle inequality, shows that for each $n\ge 1$ we have $N^*_{D,\kk}(n)\le r(q^n+1) + 2q^{n/2} \sum_{j=1}^r \gfrak_g(C_j)$. We conclude by \eqref{EqnBdAD} and taking $n=1$. 
\end{proof}

%%%%%%%%%%%%%%%%%%%%%%%%%%%%%%%%%%%%%%
%%%%%%%%%%%%%%%%%%%%%%%%%%%%%%%%%%%%%%
%%%%%%%%%%%%%%%%%%%%%%%%%%%%%%%%%%%%%%
%%%%%%%%%%%%%%%%%%%%%%%%%%%%%%%%%%%%%%
%%%%%%%%%%%%%%%%%%%%%%%%%%%%%%%%%%%%%%
%%%%%%%%%%%%%%%%%%%%%%%%%%%%%%%%%%%%%%

\section{$\omega$-integrality} \label{Secwint}

%%%%%
%%%%%
%%%%%
%%%%%

\subsection{$\omega$-integral morphisms} Let $T$ be a scheme and let $\phi:X\to Y$ be a morphism of $T$-schemes. Using the canonical morphisms of sheaves $\Omega^1_{Y/T}\to \phi_*\phi^*\Omega^1_{Y/T}$ and $\phi^*\Omega^1_{Y/T}\to \Omega^1_{X/T}$, we define 
$$
\phi^\bullet : H^0(Y,\Omega^1_{Y/T})\to H^0(X,\Omega^1_{X/T})
$$ 
as the composition
$$
H^0(Y,\Omega^1_{Y/T})\to H^0(Y,\phi_*\phi^*\Omega^1_{Y/T})=H^0(X,\phi^*\Omega^1_{Y/T})\to H^0(X,\Omega^1_{X/T}).
$$
A basic property of this construction is that it respects composition. 
\begin{lemma}\label{LemmaCompositionInt} Let $\phi: X\to Y$ and $\psi: Y \to Z$ be morphisms of $T$-schemes. Then $(\psi\circ \phi)^\bullet = \phi^\bullet \circ \psi^\bullet$.
\end{lemma}
\begin{proof} First, we remark that this is not obvious. In particular, we claim that $(\psi\circ \phi)^\bullet$  and $\phi^\bullet \circ \psi^\bullet$ are equal, not just equal up to composition with an isomorphism. The result follows from Proposition 3.29 in \cite{GFthesis}; let us explain the details.

We consider the functors $\phi_*$, $\phi_*$, $\psi_*$, $\psi^*$, $(\psi\phi)_*$, and $(\psi\phi)^*$ mapping between the categories of $\Ocal_X$-modules, $\Ocal_Y$-modules, and $\Ocal_Z$-modules. We have $(\psi\phi)_*=\psi_*\phi_*$, while the functors $(\psi\phi)^*$ and $\phi^*\psi^*$ are isomorphic.

It is a routine exercise to check that for each $\Ocal_Z$-module $\Fcal$, the diagram 
$$
\begin{tikzcd}
\Fcal \arrow{d} \arrow{rr} & & \psi_*\psi^*\Fcal\arrow{d} \\
(\psi\phi)_*(\psi\phi)^*\Fcal \arrow{r}{\simeq} & (\psi\phi)_*\phi^*\psi^*\Fcal \arrow{r}{=} & \psi_*\phi_*\phi^*\psi^*\Fcal  
\end{tikzcd}
$$
commutes, where the left vertical arrow is the canonical map, and the right vertical arrow is obtained by applying $\psi_*$ to the canonical map of $\Ocal_Y$-modules $\psi^* \Fcal\to \phi_*\phi^*\psi^*\Fcal$. In particular, the following diagram commutes
\begin{equation}\label{EqnDiag1}
\begin{tikzcd}
H^0(Z,\Omega^1_{Z/T}) \arrow{d} \arrow{rr} & & H^0(Z,\psi_*\psi^*\Omega^1_{Z/T})\arrow{d} \\
H^0(Z,(\psi\phi)_*(\psi\phi)^*\Omega^1_{Z/T}) \arrow{r}{\simeq} \arrow{d}{=}& H^0(Z,(\psi\phi)_*\phi^*\psi^*\Omega^1_{Z/T}) \arrow{r}{=} \arrow{d}{=}& H^0(Z,\psi_*\phi_*\phi^*\psi^*\Omega^1_{Z/T} ) \arrow{d}{=}\\
H^0(X,(\psi\phi)^*\Omega^1_{Z/T}) \arrow{r}{\simeq} & H^0(X,\phi^*\psi^*\Omega^1_{Z/T}) \arrow{r}{=} & H^0(Y,\phi_*\phi^*\psi^*\Omega^1_{Z/T}  ).
\end{tikzcd}
\end{equation}
Applying the functor $\phi_*\phi^*$ to the map $\psi^*\Omega^1_{Z/T}\to \Omega^1_{Y/T}$ we deduce that the diagram
\begin{equation}\label{EqnDiag2}
\begin{tikzcd}
H^0(Z,\psi_*\psi^*\Omega^1_{Z/T})\arrow{r}{=}  \arrow{d} &H^0(Y,\psi^*\Omega^1_{Z/T})\arrow{r} \arrow{d} & H^0(Y,\Omega^1_{Y/T}) \arrow{d} \\
H^0(Z,\psi_*\phi_*\phi^*\psi^*\Omega^1_{Z/T})\arrow{r}{=} &H^0(Y,\phi_*\phi^*\psi^*\Omega^1_{Z/T})\arrow{r} & H^0(Y,\phi_*\phi^*\Omega^1_{Y/T}).
\end{tikzcd}
\end{equation}
commutes. By Proposition 3.29 in \cite{GFthesis}, the following diagram commutes
\begin{equation}\label{EqnDiag3}
\begin{tikzcd}
H^0(X,(\psi\phi)^*\Omega^1_{Z/T})\arrow{r}{\simeq} \arrow{d} &H^0(X,\phi^*\psi^*\Omega^1_{Z/T})  \arrow{d}   &H^0(Y,\phi_*\phi^*\psi^*\Omega^1_{Z/T}) \arrow{l}{=} \arrow{d} \\
H^0(X,\Omega^1_{X/T}) &H^0(X,\phi^*\Omega^1_{Y/T}) \arrow{l} &H^0(Y,\phi_*\phi^*\Omega^1_{Z/T}) \arrow{l}{=}
\end{tikzcd}
\end{equation}
where the middle vertical arrow is obtained by applying $\phi^*$ to $\psi^*\Omega^1_{Z/T}\to \Omega^1_{Y/T}$. The result follows by combining \eqref{EqnDiag1}, \eqref{EqnDiag2}, and  \eqref{EqnDiag3}. 
\end{proof}
\begin{remark} Lemma \ref{LemmaCompositionInt} will be used very frequently, so we will not quote it at each application. 
\end{remark}

Given $T$-schemes $X$ and $Y$ and a differential $\omega\in H^0(Y, \Omega^1_{Y/T})$, we say that a morphism of $T$-schemes $\phi:X\to Y$ is \emph{$\omega$-integral} if $\phi^\bullet(\omega)=0$.

Let us consider points $x\in X$, $y=\phi(x)\in Y$ and let $t\in T$ be the image of $x$ (and $y$) in $T$. The map $\phi^\#_x:\Ocal_{Y,y}\to \Ocal_{X,x}$ makes $\Ocal_{X,x}$ into an $\Ocal_{Y,y}$-algebra, and both are $\Ocal_{T,t}$-algebras. Thus, we have natural maps $\Omega^1_{\Ocal_{Y,y}/\Ocal_{T,t}}\to \Omega^1_{\Ocal_{Y,y}/\Ocal_{T,t}}\otimes_{\Ocal_{Y,y}} \Ocal_{X,x}$ and $ \Omega^1_{\Ocal_{Y,y}/\Ocal_{T,t}}\otimes_{\Ocal_{Y,y}} \Ocal_{X,x}\to  \Omega^1_{\Ocal_{X,x}/\Ocal_{T,t}}$ induced by $\phi^\#_x$. Using these morphisms, we define the map 
$$
\phi^\bullet_x : \Omega^1_{\Ocal_{Y,y}/\Ocal_{T,t}}\to \Omega^1_{\Ocal_{X,x}/\Ocal_{T,t}}
$$
as the composition
$$
\Omega^1_{\Ocal_{Y,y}/\Ocal_{T,t}}\to \Omega^1_{\Ocal_{Y,y}/\Ocal_{T,t}}\otimes_{\Ocal_{Y,y}} \Ocal_{X,x}\to  \Omega^1_{\Ocal_{X,x}/\Ocal_{T,t}}.
$$
The following lemma is a straightforward verification. A more general result (stated over $\C$) can be found in \cite{GFthesis} Section 3.2.
\begin{lemma}\label{LemmaCommLoc} With the previous notation, the following diagram commutes
$$
\begin{tikzcd}
H^0(Y,\Omega^1_{Y/T}) \arrow{d}   \arrow{r}{\phi^\bullet} & H^0(X,\Omega^1_{X/T})\arrow{d}\\
\Omega^1_{\Ocal_{Y,y}/\Ocal_{T,t}} \arrow{r}{\phi^\bullet_x} & \Omega^1_{\Ocal_{X,x}/\Ocal_{T,t}}
\end{tikzcd}
$$
where the vertical arrows are localization maps.
\end{lemma}
%%
%%

%%
%%
%%%%
%%%%
%%%%
\subsection{Overdetermined $\omega$-integrality}\label{SecOver} Recall that $k$ is an algebraically closed field. Given a smooth surface $S$ over $k$ and two regular differentials $\omega_1,\omega_2$ on $S$, in general, one expects that there is no curve $C$ over $k$ with a non-constant morphism $C\to S$ which is $\omega_i$-integral for both $i=1,2$. In classical terms, the map $C\to S$ would determine a solution to an overdetermined system of differential equations, which is not a likely event. Our goal in this section is to bound the order of infinitesimal solutions to such an overdetermined system.

 Let $z$ be a variable. For $m\ge 0$ we define $E_m^k=k[z]/(z^{m+1})$ and $V_m^k = \Spec(E_m^k)$. The only point of $V_m^k$ is denoted by $\zeta$. 
\begin{theorem}[The ``overdetermined'' bound] \label{ThmOver} Let $S$ be a smooth irreducible surface over $k$ and let $x\in S(k)$. Let $\omega_1,\omega_2\in H^0(S, \Omega^1_{S/k})$ be such that $\omega_1\wedge \omega_2\in H^0(S,\Omega^2_{S/k})$ is not the zero section. Let $D=\divi_S(\omega_1\wedge \omega_2)$ and let us write $D=a_1C_1+...+a_\ell C_\ell$ where $a_j$ are positive integers and $C_j$ are irreducible curves for each $j$ (possibly, $\ell=0$ if $D=0$). For each $j$, let $\nu_j:\widetilde{C}_j\to S$ be the normalization map of $C_j$ composed with the inclusion $C_j\to S$. 

Let $m\ge 0$ be an integer such that there is a closed immersion $\phi: V_m^k\to S$ supported at $x$ (i.e. with $\phi(\zeta)=x$) which is $\omega_i$-integral for both $i=1$ and $i=2$. Then for every $\omega_0\in H^0(S,\Omega^1_{S/k})$ of the form $\omega_0=c_1\omega_1+c_2\omega_2$ with $c_1,c_2\in k$, we have
\begin{equation}\label{EqnOver}
m\le \sum_{j=1}^\ell \sum_{y\in \nu_j^{-1}(x)} a_j\cdot \left(\ord_y (\nu_j^\bullet (\omega_0))+1 \right).
\end{equation}
\end{theorem}

\begin{remark}\label{RmkOver1} If $x$ is not in the support of $D$, then the sum in \eqref{EqnOver} is empty and it gives the upper bound $m\le 0$, that is, $m=0$. This case is simpler and it is proved separately in Corollary \ref{Corom0}.
\end{remark}

\begin{remark} If $\nu_j^\bullet(\omega_0)=0$, then $\ord_y(\nu_j^\bullet(\omega_0))=+\infty$ and the upper bound for $m$ is useless. Thus, when we apply Theorem \ref{ThmOver} it is crucial to make sure that the maps $\nu_j$ are not $\omega_0$-integral.
\end{remark}

\begin{remark} \label{RmkSharp} The bound \eqref{EqnOver} is optimal in some non-trivial cases. For instance, assume that $k$ has characteristic different from $2$ and $3$.  We take $S=\A^2_k=\Spec k[s_1,s_2]$, $x=(0,0)$, $\omega_1=ds_1+s_1^2ds_2$, and $\omega_2=ds_1+s_2^2ds_2$. Thus, $\omega_1\wedge\omega_2 = (s_2-s_1)(s_2+s_1)ds_1\wedge ds_2$. Let $C_1$ and $C_2$ be defined by $s_1-s_2=0$ and $s_1+s_2=0$ respectively, so that $D=C_1+C_2$. Let $m=2$ and let $\phi :V_2^k\to \A^2_k$ be the map induced by the $k$-algebra morphism $k[s_1,s_2]\to E_2^k=k[z]/(z^3)$, $s_1\mapsto z\bmod z^3$, $s_2\mapsto 0$. The map $\phi$ is a closed immersion supported at $x$ which is $\omega_i$-integral for $i=1,2$, as $\Omega^1_{E_2^k/k}=(k[z]/(z,3z^2))dz=(k[z]/(z^2))dz$. For $\omega_0=\omega_1,\omega_2$,  \eqref{EqnOver} becomes $m\le 1\cdot (0+1) + 1\cdot (0+1)=2$.
\end{remark}

The rest of this section is devoted to the proof of Theorem \ref{ThmOver}, and we keep the same notation and assumptions as in the statement. 

%%%%
%%%%
%%%%
\subsection{Local expressions} First we study the map $\phi_\zeta^\#:\Ocal_{S,x}\to E_m^k$ induced on local rings.
\begin{lemma}[Generators for $\ker(\phi^\#_\zeta)$] \label{LemmaGenOver} There are local parameters $s_1,s_2\in \mfrak_{S,x}$ at $x$ such that $\phi^\#_\zeta(s_1)=z \bmod z^{m+1}$ and $\phi^\#_\zeta(s_2)=0$. Furthermore, $\ker(\phi^\#_\zeta)=(s_1^{m+1},s_2)$.
\end{lemma}
\begin{proof} As $\phi$ is a closed immersion, $\phi^\#_\zeta:\Ocal_{S,x}\to E_m^k$ is surjective and we can take $s_1\in \mfrak_{S,x}$ with $\phi^\#_\zeta(s_1)=z\bmod z^{m+1}$. In fact, $s_1\in \mfrak_{S,x}\smallsetminus \mfrak_{S,x}^2$ because $\phi^\#_\zeta$ is a local map. 

Let $s\in \mfrak_{S,x}$ with $\mfrak_{S,x}=(s_1,s)$, which is possible because $S$ is smooth.  One can write $\phi^\#_\zeta(s)=c_1z+...+c_{m+1} z^{m+1}\bmod z^{m+1}$ for certain $c_j\in k$, and we define $s_2=s -  (c_1s_1+...+c_{m+1} s_1^{m+1})$. Note that $\mfrak_{S,x}=(s_1,s)=(s_1,s_2)$ and $\phi^\#_\zeta(s_2)=0$. 

Clearly $(s_1^{m+1},s_2)\subseteq \ker(\phi^\#_\zeta)$. Equality holds because  $\dim_k \Ocal_{S,x}/(s^{m+1},s_2)=m+1$ and $\dim_k \Ocal_{S,x}/\ker(\phi^\#_\zeta)=\dim_k E_m^k=m+1$, due to smoothness of $S$.
\end{proof}
From now on we fix a choice of $s_1,s_2$ as in Lema \ref{LemmaGenOver}. With this choice we obtain $\widehat{\Ocal}_{S,x}=k[[s_1,s_2]]$. 

\begin{lemma}[Local expression for differentials] \label{LemmaLocDiff} Let $\omega\in H^0(S, \Omega^1_{S/k})$ be such that the map $\phi:V_m^k\to S$ is $\omega$-integral. Then there are $G_1,G_2,G_3\in \Ocal_{S,x}$ such that the germ of $\omega$ at $x$ is 
$$
\omega_x=(G_1s_1^m + G_2s_2)ds_1 + G_3ds_2 \in \Omega^1_{\Ocal_{S,x}/k}.
$$
\end{lemma}
\begin{proof} Since $\phi^\bullet(\omega)=0$, Lemma \ref{LemmaCommLoc} gives $\phi^\bullet_{\zeta}(\omega_x)=0$. Hence, 
$$
\omega_x\otimes 1_{E_m^k}\in \ker (\Omega^1_{\Ocal_{S,x}/k}\otimes E_m^k\to \Omega^1_{E_m^k/k}). 
$$
Let $I=\ker(\phi^\#_\zeta)\subseteq \Ocal_{S,x}$. As $\phi$ is a closed immersion, we have the exact sequence
$$
I/I^2\to \Omega^1_{\Ocal_{S,x}/k}\otimes E_m^k\to \Omega^1_{E_m^k/k}\to 0
$$
and it follows that $\omega_x\otimes 1_{E_m^k}$ is in the image of the first map. This means that there is $f\in I$ such that $(df-\omega_x)\otimes 1_{E_m^k}=0$. Thus, there are $g_1,g_2\in I$ with $df-\omega_x=g_1ds_1+g_2ds_2$. By Lemma \ref{LemmaGenOver} we have $I=(s_1^{m+1},s_2)$. Since $f,g_1,g_2\in I$, the result follows by a direct computation.
\end{proof}

Let us define $F\in \Ocal_{S,x}$ by the formula $\omega_{1,x}\wedge\omega_{2,x} = F\cdot ds_1\wedge ds_2\in \Omega^2_{\Ocal_{S,x}/k}$. Thus, $F$ is a local equation for $D$ at $x$. 

\begin{lemma}[Local expression for $D$]\label{LemmaLocD} We have $F\in (s_1^m,s_2)\subseteq \Ocal_{S,x}$.
\end{lemma}
\begin{proof} This follows by applying Lemma \ref{LemmaLocDiff} to $\omega_1$ and $\omega_2$, and then expanding $\omega_{1,x}\wedge \omega_{2,x}$.
\end{proof}

\begin{corollary}[The case $x\notin D$] \label{Corom0} If $x$ is not in the support of $D$, then $m=0$.
\end{corollary}
\begin{proof} In this case $F$ is a unit in $\Ocal_{S,x}$. By Lemma \ref{LemmaLocD}, this can only happen for $m=0$.
\end{proof}
%%%%
%%%%
%%%%
\subsection{Bound on a single branch} For each pair $(j,y)$ with $y\in \nu_j^{-1}(x)$ (if any) let $\pfrak_{j,y}\subseteq \widehat{\Ocal}_{S,x}=k[[s_1,s_2]]$ be the branch of $C_j$ at $x$ associated to $y$ by Lemma \ref{LemmaBranches}. Thus, there is a factorization 
$$
F=U\cdot \prod_{j=1}^\ell \prod_{y\in \nu^{-1}(x)} F_{j,y}^{a_j}
$$
where $U\in k[[s_1,s_2]]^\times$ and $F_{j,y}\in k[[s_1,s_2]]$ is irreducible with $\pfrak_{j,y}=(F_{j,y})$ for each pair $(j,y)$ as above. Here we recall that $D=\sum_{j=1}^\ell a_jC_j$.

Let $\pi:k[[s_1,s_2]]\to k[[s_1]]$ be the quotient by $(s_2)$. 

\begin{lemma}[Individual bound] \label{LemmaIndividual} Let $\omega\in H^0(S,\Omega^1_{S/k})$ be such that $\phi:V_m^k\to S$ is $\omega$-integral. Take $j$ with $1\le j\le \ell$ and fix a choice of $y\in \nu_j^{-1}(x)$, if any. Then we have
$$
\ord_y (\nu_j^\bullet(\omega))\ge \min\{m, \ord_{s_1}(\pi(F_{j,y})) -1\}.
$$
\end{lemma}
\begin{proof} Let $t\in \mfrak_{\widetilde{C}_j, y}$ be a local parameter. Thus, $\widehat{\Ocal}_{\widetilde{C}_j,y}=k[[t]]$. The morphism $\nu_{j}:\widetilde{C}_j\to S$ induces the map $\nu_{j,y}^\#: \Ocal_{S,x}\to \Ocal_{\widetilde{C}_j,y}$ and the map on completions $\widehat{\nu}_{j,y}^\#:k[[s_1,s_2]]\to k[[t]]$. 

For $i=1,2$ let $h_i=\nu_{j,y}^\#(s_i)\in \mfrak_{\widetilde{C}_j,y}\subseteq t\cdot k[[t]]$. Lemma \ref{LemmaLocDiff} gives $\omega_x=(G_1s_1^m + G_2s_2)ds_1 + G_3ds_2$ for some $G_1,G_2,G_3\in \Ocal_{S,x}$. By Lemma \ref{LemmaCommLoc} we get
$$
(\nu_j^\bullet(\omega))_y = \nu_{j,y}^\bullet(\omega_x) = \left(\nu_{j,y}^\#(G_1)h_1^m +\nu_{j,y}^\#(G_2)h_2\right)h_1'dt + \nu_{j,y}^\#(G_3)h_2'dt
$$
where $h_i'\in \Ocal_{\widetilde{C}_j,y}$ are defined by the relation $dh_i=h_i'dt$ in $\Omega^1_{\Ocal_{\widetilde{C}_j,y}/k}$ (as $dt$ generates). It follows that
\begin{equation}\label{EqnKeyVal1}
\ord_y(\nu^\bullet(\omega)) = \ord_t\left((\nu_{j,y}^\#(G_1)h_1^m +\nu_{j,y}^\#(G_2)h_2)h_1' + \nu_{j,y}^\#(G_3)h_2'\right).
\end{equation}
Note that $F_{j,y}(h_1,h_2)=\widehat{\nu}_{j,y}^\#(F_{j,y})=0$ in $k[[t]]$, since $(F_{j,y})\subseteq k[[s_1,s_2]]$ is the branch of $C_j$ corresponding to $y$. As $F_{j,y}\in k[[s_1,s_2]]$ is irreducible, Lemma \ref{LemmaPloski} gives 
\begin{equation}\label{EqnKeyVal2}
\ord_t (h_2) \ge \ord_{s_1} (\pi(F_{j,y})).
\end{equation}
The relation $dh_i=h_i'dt$ in $\Omega^1_{\Ocal_{\widetilde{C}_j,y}/k}$  shows that in $\widehat{\Ocal}_{\widetilde{C}_j,y}=k[[t]]$, the power series $h_i'$ is the formal derivative of $h_i$. Hence, for $i=1,2$ we have
\begin{equation}\label{EqnKeyVal3}
\ord_t (h_i') \ge \ord_t (h_i)-1.
\end{equation}
Using \eqref{EqnKeyVal1}, \eqref{EqnKeyVal2}, and \eqref{EqnKeyVal3} we finally deduce
$$
\begin{aligned}
\ord_y(\nu_j^\bullet(\omega))&\ge \min\left\{ \ord_t\left(\nu_{j,y}^\#(G_1)h_1^mh_1' \right), \ord_t\left(  \nu_{j,y}^\#(G_2)h_2h_1' + \nu_{j,y}^\#(G_3)h_2'  \right)   \right\} \\
& \ge \min\left\{ (m+1) \ord_t(h_1) -1, \ord_t(h_2)-1   \right\} \\
& \ge \min\left\{  m , \ord_{s_1}(\pi(F_{j,y}))-1  \right\}.
\end{aligned}
$$
\end{proof}

%%%%
%%%%
%%%%
\subsection{Proof of the overdetermined bound} We keep the notation introduced in this section and the assumptions from the statement of Theorem \ref{ThmOver}.
\begin{proof}[Proof of Theorem \ref{ThmOver}] By Corollary \ref{Corom0}, it is enough to assume that $x$ is in the support of $D$.  Let 
$$
\Scal= \sum_{j=1}^\ell \sum_{y\in \nu_j^{-1}(x)} a_j\cdot \left(\ord_y (\nu_j^\bullet (\omega_0))+1 \right). 
$$
Since $\phi$ is $\omega_i$-integral for $i=1,2$, it is also $\omega_0$-integral. By Lemma \ref{LemmaIndividual} we have
$$
\Scal\ge \sum_{j=1}^\ell \sum_{y\in \nu_j^{-1}(x)} a_j\cdot \min\left\{ m+1 , \ord_{s_1}(\pi(F_{j,y}))\right\}
$$
Since $x$ is in the support of $D$, the sum is non-empty. If there is at least one $(j,y)$ with  $1\le j\le \ell$, $y\in \nu_j^{-1}(x)$ and $\ord_{s_1}(\pi(F_{j,y}))\ge m+1$, then we get $\Scal\ge m+1>m$. 

So we can assume that for each $j$ and each $y\in\nu_j^{-1}(x)$ we have $\ord_{s_1}(\pi(F_{j,y}))\le m$. In this case
$$
\Scal\ge \sum_{j=1}^\ell \sum_{y\in \nu_j^{-1}(x)} a_j\cdot \ord_{s_1}(\pi(F_{j,y})) = \ord_{s_1} (\pi(F)).
$$
By Lemma \ref{LemmaLocD} we have $F\in (s_1^m,s_2)$, which finally gives $\ord_{s_1} (\pi(F))\ge m$.
\end{proof}

%%%%%%%%%%%%%%%%%%%%%%%%%%%%%%%%%%%%%%
%%%%%%%%%%%%%%%%%%%%%%%%%%%%%%%%%%%%%%
%%%%%%%%%%%%%%%%%%%%%%%%%%%%%%%%%%%%%%
%%%%%%%%%%%%%%%%%%%%%%%%%%%%%%%%%%%%%%
%%%%%%%%%%%%%%%%%%%%%%%%%%%%%%%%%%%%%%
%%%%%%%%%%%%%%%%%%%%%%%%%%%%%%%%%%%%%%

\section{Differentials on abelian varieties} \label{SecRestrictions}

We recall that $k$ is an algebraically closed field. Several results in this section will be stated only for positive characteristic as we just need them in this case and the counterpart for characteristic $0$ is either easier or well-known.  

Let $A$ be an abelian variety over $k$ of dimension $n\ge 1$. The neutral point of $A$ is denoted by $e$. For a point $x\in A(k)$, the map of translation by $x$ is denoted by $\tau_x$. If $Z_1$ and $Z_2$ are algebraic cycles of dimensions $r_1$ and $r_2$, the intersection $Z_1.Z_2$ is an algebraic cycle of dimension $2n-r_1-r_2$ defined up to algebraic equivalence. For $0$-cycles we have a $\Z$-valued degree map denoted by $\deg$.

If $X$ is a subvariety of $A$, the smallest translate $B$ of an abelian subvariety of $A$ with $X\subseteq B$ is denoted by $\langle X\rangle$. If $e\in X$ then $\langle X\rangle$ is an abelian subvariety of $A$.

\begin{lemma}\label{LemmaAmpleAbVar} Let $D$ be a non-zero effective divisor on $A$. If $\supp(D)$ contains no translate of positive dimensional abelian subvarieties of $A$, then $D$ is ample.
\end{lemma}
\begin{proof} Define $H(D)=\{x\in A(k) : \tau_x^*D=D\}$ where $\tau_x^*D=D$ is an equality of divisors, not just up to equivalence. Then $H(D)$ is the set of $k$-points of a Zariski-closed subgroup of $A$ which we also denote by $H(D)$. Taking any $x_0\in \supp(D)$ we get $H(D)+x_0\subseteq \supp(D)$. Our assumption on $\supp(D)$ implies that $H(D)$ is zero dimensional, hence, finite. By Application 1 in p.60 of \cite{Mumford} we get that $D$ is ample.
\end{proof}
%%
%%

%%%%
%%%% 
\subsection{Degrees} Let $D$ be a divisor on $A$. It induces a morphism $\phi_D: A\to A^\vee$ by the rule $\phi_D(x) = [\tau_x(D)-D]$. We have 
$$
\deg(\phi_D)= \left(\frac{\deg(D^n)}{n!}\right)^2.
$$
See \cite{Mumford} p.150. If $D$ is ample, then $\phi_D$ is an isogeny.

Given $\psi\in \End(A)$ and a prime $\ell$ different from the characteristic of $k$, we consider the polynomial $P_\psi(t)=\det(t\cdot \Id_A-\psi | T_\ell (A))$ where $T_\ell(A)$ is the $\ell$-adic Tate module of $A$. The polynomial $P_\psi(t)$ is in $\Z[t]$, it is monic, independent of $\ell$, and it has degree $2n$. For $r\in \Z$ one has $P_\psi(r)=\deg([r]_A -\psi)$ where $[r]_A\in \End(A)$ is the map of multiplication by $r$. Thus, $P_\psi(0)=\deg(\psi)$. See \cite{Mumford} p.180. 

We define $\tr(\psi)=\tr(\psi| T_\ell A)$, which is the negative of the coefficient of $t^{2n-1}$ in $P_\psi(t)$.

Let $\Sbf$ be the addition map from the group of $0$-cycles on $A$ to $A(k)$. For a $1$-cycle $Z$ and a divisor $D$ we define $\alpha_A(Z,D)\in \End(A)$ by $\alpha_A(Z,D)(x)=\Sbf(Z.(\tau_x(D)-D))$, see \cite{Morikawa, Matsusaka}. By \cite{Matsusaka} p.8, 
\begin{equation}\label{EqnTrAlpha}
\tr(\alpha_A(Z,D)) = 2\cdot \deg(Z.D).
\end{equation}

For a variety $X$ and an irreducible curve $C\subseteq X$ we let $\nu_{C/X}: \widetilde{C}\to X$ be the composition of the normalization map $\widetilde{C}\to C$ and the inclusion $C\to X$.

Let $C\subseteq A$ an irreducible curve. Fix a point $x\in \widetilde{C}$ and consider the associated jacobian embedding $j_x:\widetilde{C}\to J_{\widetilde{C}}$. By functoriality, we get maps $\nu_{C/A,*} : J_{\widetilde{C}}\to A$ and $\nu_{C/A}^* : A^\vee \to J_{\widetilde{C}}$ where $\nu_{C/A} = \nu_{C/A,*}\circ j_x$. Given a divisor $D$ on $A$, from \cite{Matsusaka} Lemma 3 we get
\begin{equation}\label{EqnDefAlpha}
\alpha_A(C,D) = \nu_{C/A,*}\circ \nu_{C/A}^*\circ \phi_D.
\end{equation}
The following result follows from Lemma 3.6 in \cite{Debarre}.
\begin{lemma}\label{LemmaSquare} Let $C$ be an irreducible curve with $\langle C\rangle =A$ and let $D$ be an ample divisor on $A$. There is a monic polynomial $Q_{C,D}(t)\in \Z[t]$ of degree $n$ such that $P_{\alpha_A(C,D)}(t)=Q_{C,D}(t)^2$. The roots of $Q_{C,D}(t)$ are all real and positive.
\end{lemma}
If $Z$ is an $r$-cycle on $A$, the degree of $Z$ relative to a divisor $H$ is the intersection number $\deg_H (Z)= \deg(H^r.Z)$. If $X\subseteq A$ is a smooth subvariety of dimension $d\ge 1$, the \emph{canonical degree} of $X$ relative to $H$ is $\cdeg_H(X)=\deg(H^{d-1}.K_X)$, where $K_X$ is a canonical divisor on $X$ considered as an algebraic cycle in $A$. For instance, if $B$ is an abelian subvariety of $A$, then $\cdeg_H(B)=0$.

\begin{lemma}[Degrees for surfaces in abelian $3$-folds]\label{LemmaDegXA} Let $X$ be a smooth projective surface contained in an abelian threefold $A$. Then $X|_X$ defines a canonical divisor on $X$ and 
$$
\deg_X(X)=\cdeg_X(X)=\deg(X^3)_A=c_1^2(X).
$$
\end{lemma}
\begin{proof} As $A$ has trivial canonical sheaf, $X|_X$ gives a canonical divisor on $X$ by the adjunction formula (\cite{Hartshorne} Proposition II.8.20). Thus,  $\deg(X^3)_A=\deg(X.K_X)_A=\deg(K_X.K_X)_X=c_1^2(X)$.
\end{proof}
%%
%%

%%%% 
\subsection{Restriction of $1$-forms} We assume that $k$ has positive characteristic $p$.

\begin{lemma}[Injectivity for curves]\label{LemmaInjCurves} Let $H$ be an ample divisor on $A$. Let $C$ be an irreducible curve in $A$ with $\langle C\rangle =A$. Suppose that
$$
p>\frac{n!}{n^n}\cdot \frac{\deg_H(C)^n}{\deg(H^n)}.
$$  
Then the map $\nu_{C/A}^\bullet : H^0(A,\Omega^1_{A/k})\to H^0(\widetilde{C}, \Omega^1_{\widetilde{C}/k})$ is injective.
\end{lemma}
\begin{proof} Fix a point $x\in \widetilde{C}$ and consider the associated embedding $j: \widetilde{C}\to J=J_{\widetilde{C}}$. We have an isomorphism $j^\bullet : H^0(J,\Omega^1_{J/k})\to H^0(\widetilde{C}, \Omega^1_{\widetilde{C}/k})$ which is canonical in the sense that it does not depend on the choice of $x$ (because the regular differentials on $J$ are translation invariant).

Let $\alpha =\alpha_A(C,H)$ and $\beta = \nu_{C/A,*}\circ \nu_{C/A}^* : A^\vee \to A$. Thus, $\alpha=\beta\circ \phi_H$ by \eqref{EqnDefAlpha}. 

Since $H$ is ample, Lemma \ref{LemmaSquare} gives $P_\alpha(t)=Q(t)^2$ for a monic polynomial $Q(t)\in \Z[t]$ of degree $n$, whose roots $\gamma_1,...,\gamma_n$ (counting multiplicity) are real and positive. Hence $\deg(\alpha)=P_\alpha(0)=Q(0)^2=(\gamma_1\cdots \gamma_n)^2 >0$. It follows that $\alpha:A\to A$ is an isogeny, and so is $\beta: A^\vee\to A$. Furthermore, $\deg(\phi_H)=\deg(H^n)^2/n!^2$ so we get 
$$
\deg(\beta)=\frac{\deg(\alpha)}{\deg(\phi_H)}=\left(\frac{n! Q(0)}{\deg(H^n)}\right)^2.
$$
This integer is the square of a rational number, so $M=n! |Q(0)|/\deg(H^n)$ is a positive integer. By the arithmetic-geometric mean inequality we have
$$
M = \frac{n!}{\deg(H^n)}\cdot \gamma_1\cdots \gamma_n\le \frac{n!}{\deg(H^n)}\left(\frac{\gamma_1+...+ \gamma_n}{n}\right)^n = \frac{n!}{\deg(H^n)}\left(\frac{\deg(C.H)}{n}\right)^n
$$
where the last equality uses $\tr(\alpha)=-2(\gamma_1+... + \gamma_n)$ and \eqref{EqnTrAlpha}. Therefore, $p>M$.

Since $M$ is a positive integer and $\deg(\beta)=M^2$, we conclude that $p$ does not divide $\deg(\beta)$. So, $\beta$ is a separable isogeny, giving that $\beta^\bullet : H^0(A,\Omega^1_{A/k})\to H^0(A^\vee,\Omega^1_{A^\vee/k})$ is an isomorphsm. As $\beta = \nu_{C/A,*}\circ \nu_{C/A}^*$, we get that $(\nu_{C/A,*})^\bullet : H^0(A,\Omega^1_{A/k})\to H^0(J,\Omega^1_{J/k})$ is injective. Since $\nu_{C/A} = \nu_{C/A,*}\circ j$ and $j^\bullet$ is an isomorphism, we conclude that $\nu_{C/A}^\bullet$ is injective.
\end{proof}
We remark that if $n=3$ and $C$ is smooth, then the condition on $p$ can be dropped by a result of Nakai \cite{NakaiNotes}. However, this is not enough for us because we will consider possibly singular curves.

\begin{lemma}[Semi-injectivity on canonical divisors]\label{LemmaSemiInj} Let $X$ be a smooth surface contained in $A$. Let $D$ be an effective canonical divisor of $X$ and let $C\subseteq X$ be an irreducible curve contained in the support of $D$. Consider the map $\nu_{C/A}^\bullet: H^0(A,\Omega^1_{A/k})\to H^0(\widetilde{C},\Omega^1_{\widetilde{C}/k})$. We have the following:
\begin{itemize}
\item[(i)] If $n=3$, $X$ is an ample divisor on $A$ containing no elliptic curves, and %%
\begin{equation}\label{EqnSemi1}
p> c_1^2(X),
\end{equation} 
then $\dim \ker(\nu_{C/A}^\bullet)\le 1$. 
\item[(ii)] If $A$ is simple and there is an ample divisor $H$ on $A$ satisfying 
\begin{equation}\label{EqnSemi2}
p>\frac{n!}{n^n}\cdot \frac{\cdeg_H(X)^n}{\deg(H^n)},
\end{equation} 
then $\nu_{C/A}^\bullet$ is injective.
\end{itemize} 
\end{lemma}
\begin{proof} If (i) holds we consider the ample divisor $H=X$ and note that Lemma \ref{LemmaDegXA} gives
$$
\deg_H(X)=\deg(X^3)_A=\cdeg(X)=c_1^2(X).
$$
Thus, condition \eqref{EqnSemi2} in case (i) simplifies to $p>2c_1^2(X)/3$, which is implied by \eqref{EqnSemi1}.

Let $B=\langle C\rangle$. If (ii) holds then $B=A$, while if (i) holds then $B=A$ or $\dim B=2$ since $C$ is not an elliptic curve. Thus, there are two cases to consider: 
\begin{itemize}
\item[(a)] $A=B$ and either (i) or (ii) holds.
\item[(b)] $\dim B=2$ and (i) holds.
\end{itemize}

%%%%
%%%% CASE (a)
%%%%
In case (a),  \eqref{EqnSemi1} and \eqref{EqnSemi2} directly allow us to apply Lemma \ref{LemmaInjCurves} to get injectivity of $\nu_{C/A}^\bullet$.

%%%%
%%%% CASE (b)
%%%%

In case (b) we note that $D$ is ample since it is linearly equivalent to $X|_X$ on $X$ (by Lemma \ref{LemmaDegXA}) and $X$ is ample in $A$. In particular, $D$ is numerically effective and non-zero, so Lemma \ref{LemmaCanonicalGenusBd} applies.

The curve $C$ is an ample divisor on the abelian surface $B$ because $\langle C\rangle =B$. Hence, the adjunction formula for $C$ in $B$
$$
\begin{aligned}
\frac{2!}{2^2}\cdot \frac{\deg_{C}(C)^2}{\deg(C.C)_B}=\frac{\deg(C.C)_B}{2}=\gfrak_a(C)-1\le c_1^2(X)
\end{aligned}
$$
where the last bound is due to Lemma \ref{LemmaCanonicalGenusBd}. Thus, by \eqref{EqnSemi1}  Lemma \ref{LemmaInjCurves} gives that $\nu_{C/B}^\bullet:H^0(B,\Omega^1_{B/k})\to H^0(\widetilde{C},\Omega^1_{\widetilde{C}/k})$ is injective. Let $i_{B/A}:B\to A$ be the inclusion of $B$ in $A$. Since $B$ is (up to translation) an abelian subvariety of $A$, the map $i_{B/A}^\bullet:H^0(A,\Omega^1_{A/k})\to H^0(B,\Omega^1_{B/k})$ is surjective, hence, its kernel has dimension $1$. This proves $\dim_k \ker(\nu_{C/A}^\bullet)=1$ in case (b).
\end{proof}

We also need an injectivity result for surfaces. 

\begin{lemma}[Injectivity for surfaces]\label{LemmaInjSurfaces} Let $X$ be a smooth surface contained in $A$. Let $\iota:X\to A$ be the inclusion. Assume that either of the following conditions holds:
\begin{itemize}
\item[(i)] $n=3$ and $\langle X \rangle =A$.
\item[(ii)] $A$ is simple and there is an ample divisor $H$ on $A$ such that 
$$
p>\frac{3^n\cdot n!}{n^n}\cdot \frac{\deg_H(X)^n}{\deg(H^n)}.
$$
\end{itemize}
Then the map  $\iota^\bullet : H^0(A,\Omega^1_{A/k})\to H^0(X,\Omega^1_{X/k})$ is injective.
\end{lemma}
\begin{proof} If (i) holds, the desired injectivity follows from \cite{NakaiOThDifAlgVar} Theorem 5(I), since the latter result covers the case of generating hypersurfaces in abelian varieties. 

Let us assume (ii). By \cite{Mumford} p.163, $3H$ is a very ample divisor on $A$.  By Bertini's theorem (\cite{Hartshorne} II.8.18 and III.7.9.1) we can choose a divisor $D\sim 3H$ on $A$ such that $C=D.X$ is a smooth irreducible curve. Let $\nu : C\to A$ be the inclusion; this is the same as $\nu_{C/A}$ because $C=\widetilde{C}$. Since $\nu$ factors through $\iota:X\to A$, it suffices to show that $\nu^\bullet : H^0(A,\Omega^1_{A/k})\to H^0(C,\Omega^1_{C/k})$ is injective. 

As $A$ is simple, $\langle C\rangle =A$. We note that $\deg_H(C) =\deg(H.C)= \deg(3H.X.H)=3\deg_H(X)$, so the condition on $p$ given by (ii) allows us to conclude by Lemma \ref{LemmaInjCurves}.
\end{proof}
%%
%%

%%%% 
%%%%
%%%%
\subsection{Restriction of $2$-forms} We keep the assumption that $k$ has positive characteristic $p$.

\begin{lemma}[Non-vanishing of $2$-forms on surfaces]\label{Lemma2forms} Assume $n\ge 3$. Let $X$ be a smooth surface in $A$, let $\iota:X\to A$ be the inclusion, and let $\omega_1,\omega_2\in H^0(A,\Omega^1_{A/k})$ be differentials satisfying  $\omega_1\wedge\omega_2\ne 0$ in $H^0(A,\Omega^2_{A/k})$. Assume that either of the following conditions holds:
\begin{itemize}
\item[(i)] $n=3$, $X$ is an ample divisor on $A$ containing no elliptic curves,  and 
\begin{equation}\label{Eqn2formi}
p>(128/9)\cdot c_1^2(X)^2.
\end{equation}
\item[(ii)] $A$ is simple and there is an ample divisor $H$ on $A$ satisfying 
\begin{equation}\label{Eqn2formii}
p>\frac{ n!}{n^n}\cdot \frac{(3\deg_H(X)+\cdeg_H(X))^n}{\deg(H^n)}.
\end{equation}
\end{itemize}
Then $\iota^\bullet (\omega_1)\wedge \iota^\bullet(\omega_2)\in H^0(X,\Omega^2_{X/k})$ is not the zero section. 
\end{lemma}
\begin{proof}  Let us write $u_1=\iota^\bullet (\omega_1)$ and $u_2=\iota^\bullet (\omega_2)$, which are elements of $H^0(X,\Omega^1_{X/k})$.  We note that if (i) holds, then $\langle X\rangle=A$ since $X$ is ample. Thus, in cases (i) and (ii) we see that $u_1$ and $u_2$ are linearly independent over $k$, by Lemma \ref{LemmaInjSurfaces}.

For the sake of contradiction, suppose that 
\begin{equation}\label{Eqncontr}
u_1\wedge u_2=0 \mbox{ in } H^0(X,\Omega^2_{X/k}). 
\end{equation}
Then, there is a rational function $f\in k(X)$ such that $u_1 = f \cdot u_2$ and we have that $f$ is non-constant because $u_1,u_2\in H^0(X,\Omega^1_{X/k})$ are linearly independent over $k$. 

Let $C$ be an irreducible curve in the support of the divisor of zeros of $f$, which is non-trivial because $f$ is non-constant.  As $C$ is in the zero locus of $f$ and $u_1=f\cdot u_2$, we get that $C$ is contained in the zero locus of $u_1$, and in particular 
\begin{equation}\label{Eqnw1int}
\nu_{C/A}^\bullet(\omega_1)=\nu_{C/X}^\bullet(u_1)=0.
\end{equation}

If (i) holds, we make the notation more uniform by choosing $H=X$ on $A$.   Lemma \ref{LemmaDegXA} yields  
\begin{equation}\label{Eqnci1}
\deg_H(X)=\deg(X^3)_A=\cdeg_H(X)=c_1^2(X)
\end{equation}
and since $n=3$ when (i) holds, we see that in this case the conditions \eqref{Eqn2formi} and \eqref{Eqn2formii} are the same.

Let us construct an auxiliary curve $D$. Since $H$ is ample, $3H$ is very ample (\cite{Mumford} p.163) and by Bertini (\cite{Hartshorne} Theorem II.8.18) there is a smooth irreducible curve $D\subseteq X$ linearly equivalent to $3H|_X$ as divisors in $X$. Given a non-zero reduced effective divisor $Z$ in $X$ which contains $C$  (possibly $Z=C$) we can require that $D$ meets $Z$ transversely at smooth points of $Z$ and that $D$ does not pass through points in the intersection of $C$ with the divisor of poles of $f$. We remark that the intersection $D.C$ is non-empty because $H$ is ample. A precise choice of $Z$ will be made later, but this free parameter will not affect the computations below.

By the adjunction formula we have
\begin{equation}\label{Eqn2form0}
2\gfrak_g(D)-2 = \deg ((D+K_X).D)_X = \deg (3H.(D+K_X))_A = 9\deg_H(X) + 3\cdeg_H(X).
\end{equation}

We point out that if (i) holds, then \eqref{Eqnci1}  shows that this expression simplifies to
\begin{equation}\label{Eqn2form0bis}
2\gfrak_g(D)-2 = 12c_1^2(X).
\end{equation}

We claim that $\langle D\rangle =A$. This holds in case (ii) because $A$ is simple. In case (i) note that $\langle D\rangle$ is not an elliptic curve since $D\subseteq X$, so, if $\langle D\rangle \ne A$, then $\langle D\rangle$ is an abelian surface. The latter would imply that $\langle D\rangle . D=0$ in $A$ since we can translate the abelian surface $\langle D\rangle$ so that it does not meet $D\subseteq\langle D\rangle$. But it is not possible that $\langle D\rangle . D=0$ because in case (i) we have $D=3H.X=3X.X$ where $X$ is ample. Thus, $\langle D\rangle =A$ in case (i) and (ii).

We note that by \eqref{Eqn2formi} if (i) holds, and by \eqref{Eqn2formii} if (ii) holds, we have
$$
\frac{n!}{n^n}\cdot \frac{\deg_H(D)^n}{\deg (H^n)}=\frac{n!}{n^n}\cdot \frac{\deg(H^{n-2}.3H.X)^n}{\deg (H^n)}=\frac{n! }{n^n}\cdot \frac{3^n\deg_H(X)^n}{\deg (H^n)}<p. 
$$
In view of this bound and the fact that $\langle D\rangle =A$, Lemma \ref{LemmaInjCurves} gives injectivity of the map 
\begin{equation}\label{Eqn2form1}
\nu_{D/A}^\bullet : H^0(A,\Omega^1_{A/k})\to H^0(D,\Omega^1_{D/k}).
\end{equation}

Since the map $\nu_{D/A}^\bullet$ from \eqref{Eqn2form1} is injective, we see that $0\ne \nu_{D/A}^\bullet(\omega_1) = \nu_{D/X}^\bullet(u_1)$. Since $C$ is contained in the zero locus of $u_1$, we see that $\nu_{D/X}^\bullet(u_1)$ vanishes at each point of $D\cap C\subseteq D$.  As this intersection is transverse and $\nu_{D/X}^\bullet(u_1)\ne 0$, from \eqref{Eqn2form0} we deduce
\begin{equation}\label{EqnBdInt1}
\deg((D.C)_X)=\# D\cap C\le \deg_D( \nu_{D/X}^\bullet(u_1))=2\gfrak_g(D)-2=9\deg_H(X) + 3\cdeg_H(X).
\end{equation}
Note that if (i) holds, then \eqref{Eqn2form0bis} gives
\begin{equation}\label{EqnBdInt2}
\deg((D.C)_X)\le 12 c_1^2(X).
\end{equation}

Let $B=\langle C\rangle$. As in the proof of Lemma \ref{LemmaSemiInj}, we note that if (ii) holds then $B=A$, while if (i) holds then $B=A$ or $\dim B=2$ since $C$ is not an elliptic curve. This leaves two cases: 
\begin{itemize}
\item[(a)] $A=B$ and either (i) or (ii) holds.
\item[(b)] $\dim B=2$ and (i) holds.
\end{itemize}

%%%%
%%%% CASE (a)
%%%%

First we consider case (a), and let us recall that if (i) holds then we are choosing $H=X$. Since $(D.C)_X=(3H.C)_A$, the bound \eqref{EqnBdInt1} implies
$$
\deg_H(C)=(H.C)_A\le 3\deg_H(X) + \cdeg_H(X).
$$
By  \eqref{Eqn2formi} and \eqref{Eqn2formii}, it follows that
$$
\frac{n!}{n^n}\cdot \frac{\deg_H(C)^n}{\deg(H^n)} \le  \frac{n!}{n^n}\cdot \frac{(3\deg_H(X) + \cdeg_H(X))^n}{\deg(H^n)}<p. 
$$
Since $A=B$ (we are in case (a)), we can apply Lemma \ref{LemmaInjCurves}  to conclude that 
$$
\nu_{C/A}^\bullet :H^0(A,\Omega^1_{A/k})\to H^0(\widetilde{C},\Omega^1_{\widetilde{C}/k})
$$
is injective. This is a contradiction by \eqref{Eqnw1int}. Therefore we cannot have \eqref{Eqncontr} in case (a).

%%%%
%%%% CASE (b)
%%%%

Let us now consider case (b). We begin by applying Lemma \ref{LemmaInjCurves} to $C$ and the abelian surface $B=\langle C\rangle$, choosing the ample divisor in $B$ given by the $1$-cycle $X'=(X.B)_A$.  Using \eqref{EqnBdInt2} we find
$$
\deg((X'.C)_B) = \deg((X.C)_A) =\deg((K_X.C)_X)=\frac{1}{3}\deg((D.C)_X)\le 4c_1^2(X).
$$
On the other hand, since $C$ is a component of the effective $1$-cycle $X'=(X.B)_A$ we get
$$
\deg((X'.X')_B) \ge \deg((X'.C)_B).
$$
From  \eqref{Eqn2formi} we deduce
$$
\frac{2!}{2^2}\cdot \frac{\deg_{X'}(C)^2}{\deg((X'.X')_B)}=\frac{1}{2}\cdot\frac{\deg((X'.C)_B)^2}{\deg((X'.X')_B)}\le \frac{1}{2}\deg((X'.C)_B)\le 2c_1^2(X)<p.
$$
Thus, Lemma \ref{LemmaInjCurves} implies the injectivity of the map
\begin{equation}\label{EqnnuCB}
\nu^\bullet_{C/B}:H^0(B,\Omega^1_{B/k})\to H^0(\widetilde{C},\Omega^1_{\widetilde{C}/k}).
\end{equation}

Let $i_{B/A}:B\to A$ be the inclusion. This is a closed immersion, hence it induces a surjection on cotangent spaces at each point. By translation with the group structure of $A$ we see that the map
$$
i_{B/A}^\bullet : H^0(A,\Omega^1_{A/k})\to H^0(B,\Omega^1_{B/k})
$$
is surjective. Hence, $\dim_k \ker(i_{B/A}^\bullet) = 1$. Moreover, $\omega_1\in\ker (\nu_{C/A}^\bullet)$ by \eqref{Eqnw1int}. This, together with the injectivity of \eqref{EqnnuCB} implies
\begin{equation}\label{EqnKerIsw1}
\ker \left(i_{B/A}^\bullet:H^0(A,\Omega^1_{A/k})\to H^0(B,\Omega^1_{B/k})\right) =\langle \omega_1\rangle.
\end{equation}
This, together with the injectivity of \eqref{EqnnuCB}, gives
\begin{equation}\label{EqnKerCA}
\ker \left(\nu_{C/A}^\bullet:H^0(A,\Omega^1_{A/k})\to H^0(\widetilde{C},\Omega^1_{\widetilde{C}/k})\right) =\langle \omega_1\rangle.
\end{equation}

Let $B_0$ be a translate of $B$ such that $0_A\in B_0$. Then $B_0$ is an abelian subvariety of $A$. Define the elliptic curve $E=A/B_0$ and the quotient map $\pi:A\to E$. Since $\ker(\pi)=B_0$ is smooth over $k$, the quotient map $\pi$ is smooth and we deduce that
\begin{equation}\label{Eqnpiinj}
\pi^\bullet: H^0(E,\Omega^1_{E/k})\to H^0(A,\Omega^1_{A/k})
\end{equation}
is injective. Let $\omega_E\in H^0(E,\Omega^1_{E/k})$ be a generator. Since $\pi|_B:B\to E$ is constant, we have 
$$
i_{B/A}^\bullet(\pi^\bullet(\omega_E))=(\pi|_B)^\bullet(\omega_E)=0.
$$
From \eqref{EqnKerIsw1} we deduce  $\pi^\bullet(\omega_E)\in \ker(i_{B/A}^\bullet)=\langle \omega_1\rangle$ and by injectivity of $\pi^\bullet$ (from \eqref{Eqnpiinj}) we get 
$$
\im\left(\pi^\bullet:H^0(E,\Omega^1_{E/k})\to H^0(A,\Omega^1_{A/k})\right)=\langle\omega_1\rangle\subseteq H^0(A,\Omega^1_{A/k}).
$$
Thus, replacing $\omega_E$ by a suitable non-zero multiple we may assume $\pi^\bullet(\omega_E)=\omega_1$ and in particular
\begin{equation}\label{EqnwE}
(\pi|_X)^\bullet(\omega_E)=\iota^\bullet(\pi^\bullet(\omega_E))=\iota^\bullet(\omega_1)=u_1.
\end{equation}

We claim that $\nu_{C/X}:\widetilde{C}\to X$ is $u_2$-integral.

Let us choose $Z$  as the (reduced) support of $B\cap X$. Note that $C\subseteq Z$ and that $Z$ has dimension $1$ since $X$ is not contained in $B$ ($X$ is ample in $A$). As explained in the construction of $D$, the intersection $C\cap D$ is non-empty. Let $x\in C\cap D$ and recall from the construction of $D$ that $x$ is a smooth point of $Z$ (in particular, of $C$) and it is not contained in the pole divisor of $f$ in $X$.

 Let $s_1,s_2\in \Ocal_{X,x}$ be a system of local parameters for $X$ at $x$ such that $s_1$ is a local equation for $C$ at $x$. Then $f=s_1^a\cdot \theta$ for some $a\ge 1$ and $\theta\in \Ocal_{X,x}^\times$ and in particular, $u_1=s_1^a\cdot \theta\cdot u_2$ in $\Omega^1_{X/k,x}$. 

Let $y=\pi(x)\in E$ and let $t\in \Ocal_{E,y}$ be a local parameter for $E$ at $y$. Then $\omega_E=\sigma dt$ for certain $\sigma\in \Ocal_{E,y}^\times$ because the invariant differential $\omega_E$ has no zeros. Let $\varphi=(\pi|_X)^\#_x : \Ocal_{E,y}\to \Ocal_{X,x}$. By \eqref{EqnwE} we have $u_1=\varphi(\sigma)d\varphi(t)$ in $\Omega^1_{X/k, x}$.

Note that $\pi(B)=\{y\}$, so,  $Z$ is the support of $B\cap X=(\pi|_X)^{-1}(y)$. As $x$ is a smooth point of $Z$ and $C$ is a component of $Z$, we get $\varphi(t)=s_1^b\cdot \gamma$ for certain $b\ge 1$ and $\gamma\in \Ocal_{X,x}^\times$. Thus, 
$$
s_1^a\cdot \theta\cdot u_2 = u_1 = \varphi(\sigma)d\varphi(t)= \varphi(\sigma)\left( bs_1^{b-1}\gamma d s_1+ s_1^{b}d\gamma \right)
$$
where we recall that $a,b\ge 1$. Since $\theta$ and $\sigma$ are units in their respective local rings, we have that $\tau=\theta\cdot \varphi(\sigma^{-1})\in \Ocal_{X,x}^\times$. With this notation,
\begin{equation}\label{Eqn2formkey}
s_1^a \cdot \tau\cdot u_2 = s_1^{b-1}\cdot \left( b\gamma ds_1 + s_1 d\gamma \right).
\end{equation} 
Let us consider two cases depending on the values of $a$ and $b$:
\begin{itemize}
\item $a<b$. In this case $s_1^{b-1-a}$ is regular at $x$ and we get
\begin{equation}\label{Eqn2formkey1}
 u_2 = \tau^{-1}\cdot s_1^{b-1-a}\cdot \left( b\gamma ds_1 + s_1 d\gamma \right).
\end{equation}
Note that $C$ is smooth at $x$ and $\Ocal_{C,x}\simeq \Ocal_{X,x}/(s_1)$. Let $x'\in \widetilde{C}$ be the only preimage of $x$ under $\nu_{C/X}$.  Then $\Omega^1_{\widetilde{C}/k,x'}\simeq \Omega^1_{C/k,x}  \simeq \Omega^1_{X/k,x}/(s_1, ds_1)$ and under this isomorphism the local map $\nu_{C/X,x'}^\bullet:\Omega^1_{X/k,x}\to \Omega^1_{\widetilde{C}/k,x'}$ becomes just the quotient $\Omega^1_{X/k,x}\to \Omega^1_{X/k,x}/(s_1, ds_1)$. The image of $b\gamma ds_1 + s_1 d\gamma$ under this quotient is $0$, so $\nu_{C/X,x'}^\bullet(u_2)=0$ by \eqref{Eqn2formkey1}. By Lemma \ref{LemmaCommLoc} and injectivity of $H^0(\widetilde{C},\Omega^1_{\widetilde{C}/k})\to  \Omega^1_{\widetilde{C}/k,x'}$ we deduce $\nu^\bullet_{C/X}(u_2)=0$.

\item $a\ge b$. In this case \eqref{Eqn2formkey} shows that $s_1$ divides $b\gamma ds_1$ in $\Omega^1_{X/k,x}$. As $s_1$ is part of system of local parameters at $x$ and $\gamma$ is invertible in $\Ocal_{X,x}$, this means that $b\equiv 0\bmod p$. In particular, $a\ge b\ge p$ because $a,b\ge 1$. Recall that $D$ is smooth, $C$ is smooth at $x$, and $D,C$ meet transversely at $x$. Let $x_0\in D$ be the only preimage of $x$ under $\nu_{D/X}$. By \eqref{Eqn2form0bis} and recalling that $u_1=s_1^a\cdot \theta\cdot u_2$ we find
$$
a = \ord_{x_0}(\nu_{D/X}^*(s_1^a))\le \ord_{x_0}(\nu_{D/X}^\bullet (u_1))\le \deg_D(\nu_{D/X}^\bullet (u_1))= 12c_1^2(X)<p
$$
by \eqref{Eqn2formi}. This is not possible, so the case $a\ge b$ cannot occur.
\end{itemize}
This proves that $a<b$ and $\nu^\bullet_{C/X}(u_2)=0$. Thus, $\nu_{C/X}$ is $u_2$-integral as claimed.

Finally, we have $u_2\in \ker (\nu_{C/X}^\bullet)$ which implies $\omega_2\in \ker (\nu_{C/A}^\bullet)$. By \eqref{EqnKerCA}, this implies $\omega_2\in \langle \omega_1\rangle$ in $H^0(A,\Omega^1_{A/k})$. This is the desired contradiction in case (b). Therefore, \eqref{Eqncontr} cannot hold.
\end{proof}
%%
%%

%%%%%%%%%%%%%%%%%%%%%%%%%%%%%%%%%%%%%%
%%%%%%%%%%%%%%%%%%%%%%%%%%%%%%%%%%%%%%
%%%%%%%%%%%%%%%%%%%%%%%%%%%%%%%%%%%%%%
%%%%%%%%%%%%%%%%%%%%%%%%%%%%%%%%%%%%%%
%%%%%%%%%%%%%%%%%%%%%%%%%%%%%%%%%%%%%%
%%%%%%%%%%%%%%%%%%%%%%%%%%%%%%%%%%%%%%
%%%%%%%%%%%%%%%%%%%%%%%%%%%%%%%%%%%%%%
%%%%%%%%%%%%%%%%%%%%%%%%%%%%%%%%%%%%%%
%%%%%%%%%%%%%%%%%%%%%%%%%%%%%%%%%%%%%%
%%%%%%%%%%%%%%%%%%%%%%%%%%%%%%%%%%%%%%
%%%%%%%%%%%%%%%%%%%%%%%%%%%%%%%%%%%%%%
%%%%%%%%%%%%%%%%%%%%%%%%%%%%%%%%%%%%%%
%%%%%%%%%%%%%%%%%%%%%%%%%%%%%%%%%%%%%%
%%%%%%%%%%%%%%%%%%%%%%%%%%%%%%%%%%%%%%
%%%%%%%%%%%%%%%%%%%%%%%%%%%%%%%%%%%%%%
%%%%%%%%%%%%%%%%%%%%%%%%%%%%%%%%%%%%%%
%%%%%%%%%%%%%%%%%%%%%%%%%%%%%%%%%%%%%%
%%%%%%%%%%%%%%%%%%%%%%%%%%%%%%%%%%%%%%

%%%%%%%%%%%%%%%%%%%%%%%%%%%%%%%%%%%%%%
%%%%%%%%%%%%%%%%%%%%%%%%%%%%%%%%%%%%%%
%%%%%%%%%%%%%%%%%%%%%%%%%%%%%%%%%%%%%%
%%%%%%%%%%%%%%%%%%%%%%%%%%%%%%%%%%%%%%
%%%%%%%%%%%%%%%%%%%%%%%%%%%%%%%%%%%%%%
%%%%%%%%%%%%%%%%%%%%%%%%%%%%%%%%%%%%%%

\section{Zeros of $p$-adic power series} \label{SecZeros}

We keep the notation from Section \ref{NotationLoc}.

%%%%
%%%%
%%%%
\subsection{One variable} Let $z$ be a variable. Given a power series $h(z)=\sum_{n=0}^\infty a_j z^j\in K[[z]]$ the radius of convergence of $h$ is defined by $\rho_h=\sup\{ r : \lim_j |a_j|r^j=0 \}$.

If $\rho_h>0$ then $h(z)$ converges on $B_1(\rho_h)$, where it defines an analytic function. In particular, the vanishing order $\ord_{z_0}(h)$ is defined at each $z_0\in B_1(\rho_h)$. For each $0<r<\rho_h$ we define $|h|_r=\max_j |a_j|r^j$, $\nu(h,r)=\max\{m : |a_m|r^m=|h|_r\}$, and 
$$
\nn_0(h,r)=\sum_{|z_0|\le r} \ord_{z_0}(h).
$$
The following classical bound will suffice for our purposes. See, for instance,  Theorem 1.21 in \cite{HuYang}.
\begin{lemma}\label{LemmaZeros1var} If $h\ne 0$ and $\rho_h>0$, then for each $0<r<\rho_h$ we have $\nn_0(h,r)\le \nu(h,r)$.
\end{lemma}
%%
%%

%%%%
%%%%
%%%%
\subsection{Several variables} Let $n\ge 1$ and let $\xx=(x_1,...,x_n)$ be an $n$-tuple of variables. Given 
$$
H(\xx)=\sum_{\alpha\in \N^n} c_\alpha \xx^\alpha\in K[[\xx ]]
$$ 
the radius of convergence $\rho_H$ is defined by $\rho_H= \sup \{r : \lim_m \max_{\|\alpha\| = m} |c_\alpha|r^m=0\}$. We note that this definition agrees with the one variable case when $n=1$. If $\rho_H>0$ then $H$ converges to an analytic function on $B_n(\rho_H)$. For each $j\ge 0$ let $P_{H,j}(\xx)=\sum_{\|\alpha\|=j} c_\alpha \xx^\alpha$; this is the homogeneous degree $j$ part of $H$. For $\uu\in K^n$ with $|\uu|=1$ we define
$$
H_\uu(z)=H(z\cdot \uu)=\sum_{j=0}^\infty P_{H,j}(\uu) \cdot z^j\in K[[z]].
$$
It follows that $\rho_{H_\uu}\ge \rho_H$ for each $\uu\in K^n$ with $|\uu|=1$.
\begin{lemma}[Multivariable zero estimate]\label{LemmaMVzeros} Let $H(\xx)=\sum_\alpha c_\alpha \xx^\alpha\in K[[\xx]]$ and let $\uu\in K^n$ with $|\uu|=1$. Suppose that there is an integer  $N\ge 1$ and a real number $M\ge 1$ such that
\begin{itemize}
\item[(i)] For each $\alpha$ we have $|c_\alpha|\le M^{\|\alpha\|-1}$.
\item[(ii)] $\max\{|P_{H,j}(\uu)| : 0\le j\le N\}\ge 1$ (in particular, $H_\uu\ne 0$). 
\end{itemize}
Then $\rho_H\ge M^{-1}$ and for every $0<r<M^{-1}$ we have
$$
\nn_0(H_\uu,r)\le \left(N-\frac{\log M}{\log (r^{-1})}\right)\left(1-\frac{\log M}{\log (r^{-1})}\right)^{-1}.
$$
\end{lemma}
\begin{proof} 
By (i), $\rho_H\ge M^{-1}$. Fix $0<r<M^{-1}$, in particular $r<1$. From (ii) and $r<1$ we get
$$
r^N\le \max_{0\le j\le N} |P_{H,j}(\uu)|r^j.
$$
Let $\lambda = (\log M)/\log(r^{-1})$ and note that $0<\lambda<1$. For each $m> (N-\lambda)/(1-\lambda)$ we have
$$
M^{m-1}r^m=M^{-1}(Mr)^m<M^{-1}(Mr)^{(N-\lambda)/(1-\lambda)}=r^N
$$
because $Mr<1$. By (i) and these observations, for each $m> (N-\lambda)/(1-\lambda)$ we get
\begin{equation}\label{EqnZeros1}
|P_{H,m}(\uu) |r^m \le M^{m-1} r^m  < r^N \le \max_{0\le j\le N} |P_{H,j}(\uu )|r^j.
\end{equation}
Let $N'=\lfloor (N-\lambda)/(1-\lambda)\rfloor$ and note that $N'\ge N$.  From \eqref{EqnZeros1} we deduce 
$$
|H_{\uu}|_r\le \max\{|P_{H,j}(\uu )|r^j : 0\le j\le N'\}\quad \mbox{ and }\quad \nu(H_\uu, r)\le N'.
$$
The result follows from Lemma \ref{LemmaZeros1var}.
\end{proof}
%%%%%%%%%%%%%%%%%%%%%%%%%%%%%%%%%%%%%%
%%%%%%%%%%%%%%%%%%%%%%%%%%%%%%%%%%%%%%
%%%%%%%%%%%%%%%%%%%%%%%%%%%%%%%%%%%%%%
%%%%%%%%%%%%%%%%%%%%%%%%%%%%%%%%%%%%%%
%%%%%%%%%%%%%%%%%%%%%%%%%%%%%%%%%%%%%%
%%%%%%%%%%%%%%%%%%%%%%%%%%%%%%%%%%%%%%

\section{Exponential and logarithm} \label{SecLogExp}

Let us keep the notation from Section \ref{NotationLoc}. Let $\Acal$ be an abelian variety over $\Spec R$ of relative dimension $n\ge 1$. 
%%%%
%%%%
%%%%
\subsection{Local parameters} \label{SecLogExp1}

Let $\sigma:\Spec R\to \Acal$ be a section with image $\Zcal=\sigma(\Spec R)$. The closed point of $\Zcal$ is $x_0=\sigma(\pfrak)$. Let $\widehat{\Ocal}_{\Acal ,x_0,\Zcal}$ be the completion of $\Ocal_{\Acal,x_0}$ along $\Zcal$.  That is, if $\Ical_\Zcal$ is the ideal sheaf of $\Zcal$, then $\widehat{\Ocal}_{\Acal ,x_0,\Zcal}$ is the completion of $\Ocal_{\Acal,x_0}$ with respect to the ideal $\Ical_{\Zcal,x_0}\subseteq\Ocal_{\Acal,x_0}$.  

 Let $y_1,...,y_n$ be variables. We say that elements $t_1,...,t_n\in \mfrak_{\Acal,x_0}$  are \emph{$R$-local parameters} along $\sigma$ (or $\Zcal$) if the rule $y_j\mapsto t_j$ induces a continuous isomorphism of $R$-algebras $\widehat{\Ocal}_{\Acal,x_0,\Zcal}\simeq R[[y_1,...,y_n]]$.
\begin{lemma}[Choice of local parameters]\label{LemmaChoiceLocParam} With the previous notation, let $\Xcal$ be a regular closed subscheme of $\Acal$ containing $\Zcal$ which is smooth of relative dimension $d\ge 0$ over $\Spec R$.  There are $t_j\in \mfrak_{\Acal,x_0}$ for $1\le j\le n$ such that $t_1,...,t_n$ are $R$-local parameters along $\Zcal$ and $t_1,...,t_{n-d}$ are a system of local equations for $\Xcal$ in $\Acal$ along $\Zcal$. In particular, $R$-local parameters along $\Zcal$ exist.
\end{lemma}
\begin{proof} By \cite{SGA1} Expos\'e II, Th\'eor\`eme 4.15 with $X=\Acal$, $Y=\Xcal$, $S=\Spec R$, and $x=x_0$  we get the existence of  a regular sequence $t_1,...,t_{n-d}\in \mfrak_{\Acal,x_0}$ that generates the ideal $\Ical_{\Xcal, x_0}$ of $\Xcal$ in $\Ocal_{\Acal, x_0}$. Applying \cite{SGA1} Expos\'e II, Corollaire 4.17 with $X=\Acal$, $Y=\Spec R$, $i=\sigma$, and $y=\pfrak\in \Spec R$ we get that  $\widehat{\Ocal}_{\Acal,x_0,\Zcal}$ is a power series ring over $R$ in $n$ variables. Analyzing the proof, one sees that the variables of this power series ring can be constructed by extending any regular sequence in $\Ical_{\Zcal, x_0}$, see \cite{SGA1} Expos\'e II, Remarques  4.14. We conclude since $\Ical_{\Xcal,x_0}\subseteq \Ical_{\Zcal,x_0}$.
\end{proof}
Let $\ttt=(t_1,...,t_n)$ be a choice of $R$-local parameters along the identity section $e$. These $R$-local parameters determine a formal group $\Fcal=(\Fcal_1,...,\Fcal_n)$ with $\Fcal_j\in R[[X_1,...,X_n,Y_1,...,Y_n]]$ for each $1\le j\le n$, where the $X_j$ and $Y_j$ are variables, see \cite{SilvermanHindry} Lemma C.2.4. 

Given an integer $m$, let $\Psi^{[m]}=(\Psi^{[m]}_1,...,\Psi^{[m]}_n)$ be the power series expansion of the morphism $[m]:\Acal\to \Acal$ of multiplication by $m$ in terms of $\ttt$. Then, integrality of the formal group $\Fcal$ gives
\begin{lemma}\label{LemmaPsiR} $\Psi^{[m]}_j\in R[[\ttt]]$ for each $j=1,...,n$. 
\end{lemma}

Although the notation does not explicitly indicate so, the coefficients of the power series $\Psi^{[m]}_j$  depend upon the choice of local parameters $\ttt$. Nevertheless, since $[1]=\Id$ we observe:
\begin{lemma}\label{LemmaLinearPart} $\Psi^{[1]}_j =t_j$ for each $j=1,...,n$.
\end{lemma}
  From \cite{Bourbaki} III.5.3 Proposition 2(i) we get
\begin{lemma}\label{LemmaVanishPsi} The terms of degree less than $m$ in $\Psi^{[m]}_j$ vanish for each $j=1,...,n$.
\end{lemma}
%%
%%

%%%%
%%%%
%%%%
\subsection{Power series construction} \label{SecLogExp2} Let $\psi^{[m]}_j$ be the homogeneous part of degree $m$ in $\Psi_j^{[m]}$. Thus, $\psi_j^{[1]} = t_j$ by Lemma \ref{LemmaLinearPart}. Let $x_1,...,x_n$ and $y_1,...,y_n$ be variables and for each $j=1,...,n$ we define 
$$
\begin{aligned}
\Log^{\ttt}_j(\yy) & = \sum_{m=1}^\infty \frac{(-1)^m}{m}\cdot \Psi^{[m]}_j(\yy)\in K[[\yy]]\\
\Exp^{\ttt}_j(\xx) & = \sum_{m=1}^\infty \frac{1}{m!}\cdot \psi^{[m]}_j(\xx)\in K[[\xx]].
\end{aligned}
$$
These power series depend upon the choice of local parameters $\ttt$ as the notation indicates. When this choice is clear, we will simply write
$$
\Log^\ttt_j(\yy)=\sum_\alpha b_{j,\alpha}\cdot \yy^\alpha \quad \mbox{ and }\quad \Exp^\ttt_j(\xx)=\sum_\alpha c_{j,\alpha}\cdot \xx^\alpha.
$$
The coefficients $b_{j,\alpha}$ and $c_{j,\alpha}$ also depend on the choice of $\ttt$, although the notation is not explicit.

By \cite{Bourbaki} III.5.4 Prop. 3, these are the power series expansions of the exponential and logarithm maps of the $p$-adic Lie group $A(K)$ at the identity point $e$, with respect to the choice of local parameters $\ttt$. They converge on some $p$-adic neighborhood of $e$, but for our purposes we need a more precise discussion on convergence.

\begin{lemma}[Convergence of  logarithm]\label{LemmaCvLog} For each $j=1,...,n$, the radius of convergence of $\Log^{\ttt}_j$ is at least $1$. Furthermore, for each $\alpha\in \N^n$ we have $|b_{j,\alpha}|\le \|\alpha\|^{[K:\Q_p]}$.
\end{lemma}
\begin{proof} From Lemmas \ref{LemmaPsiR} and \ref{LemmaVanishPsi} we get $|b_{j,\alpha}|\le \max\{ |m^{-1}|: 1\le m\le\|\alpha\|\} \le \|\alpha\|^{[K:\Q_p]}$. The claim on the radius of convergence follows.
\end{proof}
\begin{lemma}[Convergence of the exponential] \label{LemmaCvExp} For each $j=1,...,n$ and each $\alpha\in \N^n$, we have $m!\cdot c_{j,\alpha}\in R$ where $m=\|\alpha\|$. Furthermore, $|c_{j,\alpha}|\le p^{[K:\Q_p](m-1)/(p-1)}$. In particular, the radius of convergence of $\Exp^{\ttt}_j$ is at least $p^{-[K:\Q_p]/(p-1)}$.
\end{lemma}
\begin{proof} Since $c_{j,\alpha}$ is a coefficient of $m!^{-1} \psi_j^{[m]}(\xx)$, we get $m!\cdot c_{j,\alpha}\in R$ from Lemma \ref{LemmaPsiR}. Letting $s_p(m)$ be the sum of the digits of $m$ in base $p$, we get
$$
|c_{j,\alpha}|\le \frac{1}{|m!|} =p^{[K:\Q_p](m-s_p(m))/(p-1)}\le p^{[K:\Q_p](m-1)/(p-1)}.
$$
The claim on the radius of convergence follows.
\end{proof}
%%
%%

%%%%
%%%%
%%%%
\subsection{Local linearization}\label{SecLocLin} We consider the open set $V = \pfrak\times ...\times \pfrak\subseteq K^n$ and the additive Lie group $\Tcal$ given by $V$ with the usual addition. The variables $\xx=(x_1,...,x_n)$ and $\yy=(y_1,...,y_n)$ will be considered as coordinates on $\Tcal$ and $V$ respectively.

Since $\Acal\to \Spec R$ is proper, we have the reduction map $\red : A(K)=\Acal(K)\to A'(\kk)$ where $A'=\Acal\otimes \kk$. This is a group morphism. Then $U_e=\ker(\red)$ is a Lie subgroup of $A(K)$. 

The $R$-local parameters $\ttt=(t_1,...,t_n)$ define functions $t_j:U_e\to \pfrak$ which determine a $K$-analytic local chart $\chi^\ttt:U_e\to V$. The map $\chi^\ttt$ is bijective, although it does not need to be a group morphism.

\begin{lemma}\label{LemmaKeyLogExp} Suppose that $p>\max\{\efrak +1 , \exp(\efrak/\exp(1))\}$. Then we have the following:
\begin{itemize}
\item[(i)]  $\Log^\ttt_j$ and $\Exp^\ttt_j$ have radius of convergence larger than $1/q$ for each $j=1,...,n$. 
\item[(ii)] The power series $\Log^\ttt(\yy)=(\Log^\ttt_j(\xx))_{j=1}^n$ and $\Exp^\ttt(\xx)=(\Exp^\ttt_j(\yy))_{j=1}^n$ give analytic maps
$$
\Log^\ttt : V\to \Tcal\quad  \mbox{ and }\quad \Exp^\ttt:\Tcal\to V
$$
which are inverse to each other.
\item[(iii)] The maps $\widetilde{\Log}^\ttt$ and $\widetilde{\Exp}^\ttt$ defined by
$$
\widetilde{\Log}^\ttt=\Log^\ttt \circ \chi^\ttt: U_e\to \Tcal\quad  \mbox{ and }\quad \widetilde{\Exp}^\ttt=(\chi^\ttt)^{-1}\circ \Exp^\ttt:\Tcal\to U_e
$$
are isomorphisms of Lie groups, inverse to each other.
\end{itemize}
\end{lemma}
\begin{proof} (i)  Lemma \ref{LemmaCvLog} implies the result for $\Log^\ttt_j$. The condition $p>\efrak+1$ gives $[K:\Q_p]<(p-1)\ffrak$, hence $p^{[K:\Q_p]/(p-1)}<q$. Lemma \ref{LemmaCvExp} shows that $\Exp^\ttt_j$ has radius of convergence at least $p^{-[K:\Q_p]/(p-1)}$, which is larger than $1/q$.

(ii) As open sets of $K^n$, we have $V=\Tcal=B_n[1/q]$. Thus, (i) gives that $\Log^\ttt$ and $\Exp^\ttt$ are analytic on the domains $V$ and $\Tcal$ respectively, as maps to $K^n$.

Let us show that $\Log^\ttt$ maps $V$ to $\Tcal$. Since $B_n[1/q]=B_n(1)$ in $K^n$, it suffices to check that for all $\vv\in V$ and $j=1,...,n$ we have $|\Log^\ttt_j(\vv)|<1$. By Lemma \ref{LemmaCvLog} we have 
$$
|\Log^\ttt_j(\vv)|\le \max_{m\ge 1} m^{[K:\Q_p]}/q^m =\max\left\{1/q , \max_{m\ge 2} m^{[K:\Q_p]}/q^m\right\}. 
$$
By $p>\exp(\efrak/\exp(1))$ we have $m/\log m> \exp(1)> \efrak/\log p$ for all $m\ge 2$. Thus, $[K:\Q_p]\log m =\efrak\ffrak\log m < \ffrak m\log p=m\log q$. Hence, $m^{[K:\Q_p]}/q^m<1$ for each $m\ge 2$, proving  $|\Log^\ttt(\vv)_j|<1$.

To show that $\Exp^\ttt$ maps $\Tcal$ to $V$, it suffices to check that for each $\vv\in \Tcal$ and $j=1,...,n$ we have $|\Exp^\ttt_j(\vv)|<1$. Lemma \ref{LemmaCvExp} gives 
$$
|\Exp^\ttt_j(\vv)|\le \sup_{m\ge 1} p^{[K:\Q_p]\frac{m-1}{p-1}}/q^m = \sup_{m\ge 1} p^{[K:\Q_p]\frac{m-1}{p-1} -\ffrak m}.
$$
Since $p>\efrak+1$, we get $(m-1)\efrak < (p-1)m$. This gives $[K:\Q_p]\frac{m-1}{p-1} -\ffrak m<0$, proving $|\Exp^\ttt_j(\vv)|<1$.

The fact that $\Log^\ttt$ and $\Exp^\ttt$ are inverse to each other is a power series identity, see \cite{Bourbaki} III 5.4.

Finally, (iii) also follows from \cite{Bourbaki} III 5.4, where a neighborhood of $e$ is identified with an open set of $K^n$ by the choice of local parameters made in \cite{Bourbaki} III 5.3.
\end{proof}
%%
%%

%%%%%%%%%%%%%%%%%%%%%%%%%%%%%%%%%%%%%%
%%%%%%%%%%%%%%%%%%%%%%%%%%%%%%%%%%%%%%
%%%%%%%%%%%%%%%%%%%%%%%%%%%%%%%%%%%%%%
%%%%%%%%%%%%%%%%%%%%%%%%%%%%%%%%%%%%%%
%%%%%%%%%%%%%%%%%%%%%%%%%%%%%%%%%%%%%%
%%%%%%%%%%%%%%%%%%%%%%%%%%%%%%%%%%%%%%

\section{Analytic $1$-parameter subgroups} \label{Sec1PS}

Using the results on the logarithm and exponential maps from the previous section, in this section introduce analytic $1$-parameter subgroups of abelian varieties in the non-archimedian setting and study some fundamental properties. The naive idea is that a $1$-parameter subgroup of an abelian variety should be an analytic curve at the identity element which is locally a subgroup. However, it is more convenient for our purposes to take a different approach.

Again, we let $\Acal$ be an abelian variety over $\Spec R$ of relative dimension $n\ge 1$. We keep the notation introduced in Sections \ref{NotationLoc} and  \ref{SecLogExp}. We assume 
$$
p>\max\{\efrak +1 , \exp(\efrak/\exp(1))\}
$$ 
throughout this section, so that Lemma \ref{LemmaKeyLogExp} applies.

%%%%
%%%%
%%%%
\subsection{Definitions} If $G$ is an abelian $K$-analytic Lie group, the $K$-vector space of invariant differentials will be denoted by $\Omega(G)$. A morphism of abelian Lie groups $f:G_1\to G_2$ functorially induces a linear map on invariant differentials, which we denote by $f^\bullet:\Omega(G_2)\to \Omega(G_1)$.

A \emph{$1$-parameter subgroup} of $\Acal$ is a pair $\gamma=(\ttt,\uu)$ where $\ttt=(t_1,...,t_n)$ is a choice of $R$-local parameters at $e\in A(K)$ and $\uu\in K^n$ satisfies $|\uu|=1$.

 The space of invariant differentials on the Lie group $\Tcal$ is $\Omega(\Tcal) = \bigoplus_{j=1}^n K\cdot dx_j$ and the space of invariant differentials on $A(K)$ is $ \Omega(A(K))=H^0(A,\Omega^1_{A/K})$. 
 
 Given $\ttt$  a choice of $R$-local parameters at $e\in A(K)$, the isomorphism of Lie groups $\widetilde{\Exp}^\ttt:\Tcal\to U_e$ (cf. Lemma \ref{LemmaKeyLogExp}) induces an isomorphism of $K$-vector spaces on invariant differentials 
\begin{equation}\label{Eqnzxc}
(\widetilde{\Exp}^\ttt)^\bullet : H^0(A,\Omega^1_{A/K})\to \Omega(\Tcal).
\end{equation}
Given $\uu=(u_1,...,u_n)\in K^n$ with $|\uu|=1$, we define the $K$-linear map 
$$
\epsilon^\uu: \Omega(\Tcal)\to K
$$ 
by the rule $dx_j\mapsto u_j$. For a $1$-parameter subgroup $\gamma=(\ttt,\uu)$ we get a non-trivial $K$-linear map 
$$
\epsilon^\uu\circ (\widetilde{\Exp}^\ttt)^\bullet: H^0(A,\Omega^1_{A/K})\to K.
$$
Let us define $\Hcal(\gamma)=\ker(\epsilon^\uu\circ (\widetilde{\Exp}^\ttt)^\bullet)$ and note that it is a hyperplane in $H^0(A,\Omega^1_{A/K})$.

A $1$-parameter subgroup $\gamma=(\ttt,\uu)$ determines a morphism of Lie groups 
\begin{equation}\label{Eqntildegamma}
\tilde{\gamma}:\pfrak \to U_e, \quad \tilde{\gamma}(z)=\widetilde{\Exp}^\ttt(u_1z,...,u_nz)
\end{equation}
where $z$ denotes the variable on the additive Lie group $\pfrak$. The space of invariant differentials on $\pfrak$ is $\Omega(\pfrak) = K\cdot dz$. One immediately verifies
\begin{lemma}\label{LemmaHisKernel} The map $\tilde{\gamma}^\bullet : H^0(A,\Omega^1_{A/K})\to \Omega(\pfrak)$ induced by $\tilde{\gamma}$  is given by $\omega\mapsto (\epsilon^\uu\circ (\widetilde{\Exp}^\ttt)^\bullet)(\omega)\cdot dz$. In particular, $\ker(\tilde{\gamma}^\bullet)=\Hcal(\gamma)$.
\end{lemma}
The \emph{image} of a $1$-parameter subgroup $\gamma$ is defined as $\im(\gamma)=\tilde{\gamma}(\pfrak)$. We observe that
\begin{lemma}\label{LemmaSaturated} If $\gamma=(\ttt,\uu)$ is a $1$-parameter subgroup of $\Acal$, then $\im(\gamma)$ is a Lie subgroup of $U_e$. The image of $\im(\gamma)$ under the isomorphism $\widetilde{\Log}^\ttt:U_e\to \Tcal$ is $\pfrak\cdot \uu$, which is a closed, saturated $R$-submodule of $\Tcal$, free of rank $1$ over $R$, and generated by $\varpi\cdot \uu$.
\end{lemma}
%%
%%

%%%%
%%%%
%%%%
\subsection{Equivalence} Given $\gamma$ and $\gamma'$ two $1$-parameter subgroups  of $\Acal$, we say that they are \emph{equivalent} if $\Hcal(\gamma)=\Hcal(\gamma')$. This is denoted by $\gamma\sim\gamma'$.

\begin{lemma} \label{LemmaVaryu1PS} Let $\ttt$ be a choice of $R$-local parameters at $e$ and let $\uu,\uu'\in K^n$ with $|\uu|=|\uu'|=1$. Let $\gamma=(\ttt,\uu)$ and $\gamma'=(\ttt,\uu')$. We have $\gamma\sim \gamma'$ if and only if there is $\eta\in R^\times$ with $\uu'=\eta\cdot \uu$.
\end{lemma}
\begin{proof} As \eqref{Eqnzxc} is an isomorphism, we have $\Hcal(\gamma)=\Hcal(\gamma')$ if and only if $\ker(\epsilon^\uu)=\ker(\epsilon^{\uu'})$. The latter holds if and only if $\uu'=\eta\cdot \uu$ for some $\eta\in K^\times$, in which case  $\eta\in R^\times$ as $|\uu|=|\uu'|=1$.
\end{proof}
\begin{lemma} \label{LemmaEquiv1PS} Let $\gamma$ and $\gamma'$ be $1$-parameter subgroups of $\Acal$. The following are equivalent:
\begin{itemize}
\item[(i)] $\gamma\sim\gamma'$.
\item[(ii)] $\im(\gamma)=\im(\gamma')$.
\item[(iii)] $\im(\gamma)$ and $\im(\gamma')$ have non-trivial intersection, i.e. $\{e\}\subsetneq \im(\gamma)\cap\im(\gamma')$.
\end{itemize}
\end{lemma}
\begin{proof}
Write $\gamma=(\ttt,\uu)$ and $\gamma'=(\ttt',\uu')$. Define $\delta:\Tcal\to \Tcal$ as $\delta=\widetilde{\Log}^\ttt\circ \widetilde{\Exp}^{\ttt'}$. Then $\delta$ is a Lie group automorphism of the additive Lie group $\Tcal$. Thus, $\delta$ is given by $\delta\in \GL_n(R)$ acting on $\Tcal=\pfrak\times...\times \pfrak$ on the left.

After these remarks, Lemma \ref{LemmaSaturated} shows that (ii) and (iii) are equivalent, and in fact, both are equivalent to the following condition: (iv) there exists $\eta\in R^\times$ with $\delta(\uu') = \eta\cdot \uu$ (note that as an element of $\GL_n(R)$, $\delta$ acts on $K^n$).

Let us show that (iv) is equivalent to (i). Note that the explicit definition of $\epsilon^{\uu}$ and the fact that $|\uu|=|\uu'|=1$ give that (iv) is equivalent to the condition $\ker(\epsilon^\uu) = \ker(\epsilon^{\delta(\uu')})$. Directly evaluating on $dx_i$ we see that $\epsilon^{\delta(\uu')}=\epsilon^{\uu'}\circ \delta^\bullet$, so, (iv) is equivalent to $\ker(\epsilon^\uu) = \ker(\epsilon^{\uu'}\circ \delta^\bullet)$. Using the fact that $\delta^\bullet\circ (\widetilde{\Exp}^{\ttt})^\bullet=(\widetilde{\Exp}^{\ttt}\circ \delta)^\bullet =( \widetilde{\Exp}^{\ttt'})^\bullet$ and that $( \widetilde{\Exp}^{\ttt})^\bullet$ is an isomorphism, we finally see that (iv) is equivalent to
$\ker(\epsilon^\uu\circ( \widetilde{\Exp}^{\ttt})^\bullet)=\ker(\epsilon^{\uu'}\circ( \widetilde{\Exp}^{\ttt'})^\bullet)$, which is exactly (i).
\end{proof}

\begin{lemma}\label{LemmaH1PS} The rule $\gamma\mapsto \Hcal(\gamma)$ defines a bijection between equivalence classes of $1$-parameter subgroups of $\Acal$ and $K$-linear hyperplanes of $H^0(A,\Omega^1_{A/K})$.
\end{lemma}
\begin{proof} By Lemma \ref{LemmaEquiv1PS}, it only remains to show that given any hyperplane $H$ of $H^0(A,\Omega^1_{A/K})$ there is a $1$-parameter subgroup $\gamma$ with $H=\Hcal(\gamma)$. Choose $R$-local parameters $\ttt$ at $e$. Let $H'=(\widetilde{\Exp}^\ttt)^\bullet(H)$ and note that it is a hyperplane in $\Omega(\Tcal)$. We can choose $\uu\in K^n$ with $|\uu|=1$ such that $\ker(\epsilon^\uu)=H'$, hence $H=\ker(\epsilon^\uu\circ (\widetilde{\Exp}^\ttt)^\bullet)$ as desired.
\end{proof}

\begin{lemma}\label{Lemma1PSgen} Let $\xi\in U_e$ with $\xi\ne e$ and let $\ttt$ be a choice of $R$-local parameters at $e$. There exists $\uu\in K^n$ with $|\uu|=1$ such that $\xi \in \im(\gamma)$ for $\gamma=(\ttt,\uu)$. Furthermore, $\uu$ is unique for this $\xi$ and $\ttt$, up to multiplication by elements of $R^\times$.
\end{lemma}
\begin{proof} Let $\vv=\widetilde{\Log}(\xi)\in \Tcal$. We have $\vv\ne 0$ because $\xi\ne e$ and $\widetilde{\Log}:U_e\to \Tcal$ is an isomorphism of Lie groups. Take $m\in \Z$ such that $|\varpi^m|=|\vv|$. We can choose $\uu=\varpi^{-m}\cdot \vv$. 

If $\gamma=(\ttt,\uu)$ and $\gamma'=(\ttt,\uu')$ satisfy that $\xi\in \im(\gamma)\cap\im(\gamma')$ then $\gamma\sim\gamma$ by Lemma \ref{LemmaEquiv1PS}, and we conclude by Lemma \ref{LemmaVaryu1PS}.
\end{proof}
%%%%
%%%%
%%%%
\subsection{Infinitesimal $1$-parameter subgroups} For an integer $m\ge 0$ we define 
\begin{itemize}
\item $\Ecal_m= R[z]/(z^{m+1})$, $\Vcal_m=\Spec \Ecal_m$, 
\item $E_m=\Ecal_m\otimes_R K$, $V_m=\Spec E_m$, 
\item $E'_m=\Ecal_m\otimes_R \kk$, $V'_m=\Spec E'_m$.
\end{itemize}

Let $\gamma=(\ttt,\uu)$ be a $1$-parameter subgroup of $\Acal$. The power series expansion of the Lie group morphism $\tilde{\gamma}:\pfrak\to U_e$ (see \eqref{Eqntildegamma}) with respect to $z$ and $\ttt$ induces a morphism of formal schemes $\widehat{\gamma}: \Spf K[[z]]\to  \Spf \widehat{\Ocal}_{A,e}$ determined by the following map on completed local rings:
\begin{equation}\label{Eqnhatgamma}
\widehat{\gamma}^\#: \widehat{\Ocal}_{A,e}\to K[[z]], \quad t_j\mapsto \Exp^\ttt_j(u_1z,...,u_nz).
\end{equation}
For each $m\ge 0$, the map $\widehat{\gamma}$ induces a map $\widehat{\gamma}_m:V_m\to \Spf  \widehat{\Ocal}_{A,e}$ giving the morphism of $K$-schemes
$$
\tilde{\gamma}_m:V_m\to A
$$
supported at $e$. Explicitly, $\widehat{\gamma}_m$ induces 
\begin{equation}\label{Eqninf1ps}
\widehat{\gamma}_m^\#: \widehat{\Ocal}_{A,e}\to E_m, \quad t_j\mapsto \Exp^\ttt_j(u_1z,...,u_nz)\bmod z^{m+1}
\end{equation}
on completed local rings, and the restriction to $\Ocal_{A,e}$ is the map $\tilde{\gamma}_m^\#:\Ocal_{A,e}\to E_m$ induced by $\tilde{\gamma}_m$. In summary, the following diagram commutes:
\begin{equation}\label{Eqngammas}
\begin{tikzcd}
\widehat{\Ocal}_{A,e}   \arrow{r}{\widehat{\gamma}^\#} \arrow{rd}{\widehat{\gamma}^\#_m}   &  K[[z]]  \arrow{d}  \\
\Ocal_{A,e}  \arrow[hook]{u}  \arrow[swap]{r}{\tilde{\gamma}^\#_m }&   E_m.
\end{tikzcd}
\end{equation}
The morphism $\tilde{\gamma}_m$ is the $m$-th jet of $\tilde{\gamma}$, which might be thought of as an infinitesimal $1$-parameter subgroup. We remark that, although $\tilde{\gamma}$ is an analytic maps of Lie groups, the maps $\tilde{\gamma}_m:V_m\to A$ are scheme morphisms.

For each integer $h\ge 0$ it will be convenient to define the polynomial $P^{\ttt}_{j,h}(\xx)\in K[\xx]$ as the homogeneous part of degree $h$ of $\Exp^\ttt_j(\xx)$, which in the notation of Section \ref{SecLogExp2} can be written as
$$
P^{\ttt}_{j,h}(\xx)=\frac{1}{m!}\cdot \psi^{[h]}_j(\xx)=\sum_{\|\alpha\|=h} c_{j,\alpha}\cdot \xx^\alpha.
$$
In particular, $P^\ttt_{j,0}=0$. With this notation, the map $\widehat{\gamma}^\#$ from \eqref{Eqninf1ps} can be expressed as
\begin{equation}\label{Eqninf1psP}
\widehat{\gamma}^\#_m(t_j)=\sum_{h=0}^{m} P^\ttt_{j,h}(\uu)\cdot z^h\bmod z^{m+1}. 
\end{equation}
\begin{lemma}[Integrality and reduction modulo $\pfrak$]\label{LemmaTruncate1ps} Let $\gamma=(\ttt,\uu)$ be a $1$-parameter subgroup of $\Acal$. If $m<p$, the morphism $\tilde{\gamma}_m:V_m\to A$ extends to an $R$-morphism $\Vcal_m\to \Acal$ supported along the identity section $\sigma_e:\Spec R\to \Acal$. Hence, upon base change to $\kk$, it determines a $\kk$-morphism $V'_m\to A'$ supported at $e$. 
\end{lemma}
\begin{proof} Lemma \ref{LemmaCvExp} shows that $|P^\ttt_{j,h}(\uu)|\le 1$ for each $h\le m$, since $m<p$. We conclude by applying this estimate to \eqref{Eqninf1psP} and the fact that $\ttt$ is a choice of $R$-local parameters for $\Acal$ along $e$.
\end{proof}
Given a $1$-parameter subgroup $\gamma$ and an integer $0\le m<p$, the morphisms provided by Lemma \ref{LemmaTruncate1ps} will be denoted by
$$
\tilde{\gamma}^R_m:\Vcal_m\to \Acal\quad \mbox{ and }\quad \tilde{\gamma}'_m : V'_m\to A'.
$$
\begin{lemma}[Closed immersions]\label{LemmaClosedImm} Let $\gamma=(\ttt,\uu)$ be a $1$-parameter subgroup of $\Acal$. For each $m\ge 0$, the morphism $\tilde{\gamma}_m:V_m\to A$ is a closed immersion. If $m<p$, then the maps $\tilde{\gamma}^R_m:\Vcal_m\to \Acal$ and $\tilde{\gamma}'_m:V'_m\to A'$ are also closed immersions.
\end{lemma}
\begin{proof} For $m=0$ we have $E_0=K$, $\Ecal_0=R$, and $E'_0=\kk$, so, the result is clear. Assume $m\ge 1$.

By Lemma \ref{LemmaLinearPart} we have $P^\ttt_{j,1}(\xx)=x_j$. Thus, $P^\ttt_{j,1}(\uu)=u_j$ and since $|\uu|=1$ we conclude that for some $j=1,...,n$ we have $P^\ttt_{j,1}(\uu)\in R^\times$. By \eqref{Eqninf1psP} and the fact that $t_j\in \mfrak_{\Acal,e}\subseteq \Ocal_{\Acal,e}\subseteq \Ocal_{A,e}$ (cf. Section \ref{SecLogExp1}), this implies that the image of $\tilde{\gamma}_m^\#$ contains an element of the form $uz+z^2h\bmod z^{m+1}$ for some $u\in R^\times$ and some $h\in K[z]$. As such an element generates $E_m$ as a $K$-algebra, we obtain that $\tilde{\gamma}_m^\#:\Ocal_{A,e}\to E_m$ is surjective. Thus, $\tilde{\gamma}_m$ is a closed immersion. 

Since $u\in R^\times$, an analogous argument works for $\tilde{\gamma}^R_m$ and $\tilde{\gamma}'_m$ if $m<p$.
\end{proof}
%%
%%

%%%%
%%%%
%%%%
\subsection{$\omega$-integrality} 
\begin{lemma}\label{LemmaDiffExpl} We have $H^0(V_m, \Omega^1_{V_m/K})=(K[z]/(z^m))dz$. Furthermore, if $m\le p-2$ then $H^0(\Vcal_m,\Omega^1_{\Vcal_m/R}) = (R[z]/(z^m))dz$ and $H^0(V'_m, \Omega^1_{V'_m/\kk})=(\kk[z]/(z^m))dz$. In particular, for $m\le p-2$ we have that $H^0(\Vcal_m,\Omega^1_{\Vcal_m/R})=\Omega^1_{\Ecal_m/R} $ is a free $R$-module of rank $m$.
\end{lemma}
\begin{proof}  Note that $H^0(V_m,\Omega^1_{V_m/K})=\Omega^1_{E_m/K}=K[z]dz/(z^{m+1}, d(z^{m+1}))$. The first claim follows from the fact that $d(z^{m+1})=(m+1)z^mdz$ and $m+1\in K^\times$. The rest is proved similarly. 
\end{proof}

Recall that for the additive Lie group $\pfrak$ the space of invariant differentials is $\Omega(\pfrak)=K\cdot dz$. For each $m$ we define the $K$-linear map 
$$
\rho_m:\Omega(\pfrak)\to \Omega^1_{E_m/K}=H^0(V_m, \Omega^1_{V_m/K})
$$ 
by the rule $\rho_m(dz)=dz\in \Omega^1_{E_m/K}$.

\begin{lemma}[Analytic-algebraic compatibility] \label{LemmaCompatibility} Let $\gamma=(\ttt,\uu)$ be a $1$-parameter subgroup of $\Acal$ and let $m\ge 0$. The map on invariant differentials $\tilde{\gamma}^\bullet : H^0(A,\Omega^1_{A/K})\to \Omega(\pfrak)$ induced by the Lie group morphism $\tilde{\gamma}$ and the map $\tilde{\gamma}_m^\bullet : H^0(A,\Omega^1_{A/K})\to H^0(V_m,\Omega^1_{V_m/K})$ induced by the scheme morphism $\tilde{\gamma}_m:V_m\to A$ satisfy $\rho_m\circ \tilde{\gamma}^\bullet=\tilde{\gamma}_m^\bullet$.
\end{lemma}
\begin{proof} For a $K$-algebra $B$, the module of universally finite differentials over $K$ (cf. \cite{KunzDiff} Sec. 11) is denoted by $\widetilde{\Omega}_{B/K}$. We recall that in the special case $B=K[[x_1,...,x_n]]$ we have $\widetilde{\Omega}_{B/K}=\bigoplus_{j=1}^n B\cdot dx_j$ and the map $d:B\to \widetilde{\Omega}_{B/K}$ is continuous for the $(x_1,...,x_n)$-topology.

Let $\kappa:\widetilde{\Omega}_{\widehat{\Ocal}_{A,e}/K}\to \widetilde{\Omega}_{K[[z]]/K}$ be the map induced by the continuous ring map $\widehat{\gamma}^\# : \widehat{\Ocal}_{A,e}\to K[[z]]$ given by the power series expansion of $\tilde{\gamma}$, cf. \eqref{Eqnhatgamma}. Since the map $\tilde{\gamma}^\bullet$ might be computed from the power series expansion of the Lie group morphism $\tilde{\gamma}: \pfrak\to U_e$ by taking differentials term-by-term, we deduce that the following diagram commutes:
$$
\begin{tikzcd}
  H^0(A,\Omega^1_{A/K}) \arrow{r}{\tilde{\gamma}^\bullet}   \arrow[hook]{d}   & \Omega(\pfrak) \arrow[hook]{d} \\
  \widetilde{\Omega}_{\widehat{\Ocal}_{A,e}/K} \arrow{r}{\kappa} &\widetilde{\Omega}_{K[[z]]/K}.
\end{tikzcd}
$$
The map $H^0(A,\Omega^1_{A/K})\to \widetilde{\Omega}_{\widehat{\Ocal}_{A,e}/K}$ factors through $\Omega^1_{\Ocal_{A,e}/K}$. The quotient $K[[z]]\to E_m$ induces a map $\widetilde{\Omega}_{K[[z]]/K}\to \Omega^1_{E_m/K}$ and one checks (evaluating on $dz$) that this map composed with $\Omega(\pfrak)\to \widetilde{\Omega}_{K[[z]]/K}$ is $\rho_m$. From these two observations we obtain the commutative diagram
$$
\begin{tikzcd}
   & H^0(A,\Omega^1_{A/K}) \arrow{r}{\tilde{\gamma}^\bullet}   \arrow[hook]{d}  \arrow[hook]{ld}  &  \Omega(\pfrak) \arrow[hook]{d} \arrow{rd}{\rho_m} & \\
\Omega^1_{\Ocal_{A,e}/K} \arrow[hook]{r}& \widetilde{\Omega}_{\widehat{\Ocal}_{A,e}/K} \arrow{r}{\kappa} &\widetilde{\Omega}_{K[[z]]/K}\arrow{r} & \Omega^1_{E_m/K}.
\end{tikzcd}
$$
The composition $\Omega^1_{\Ocal_{A,e}/K}\to \Omega^1_{E_m/K}$ of the bottom maps is the morphism induced by $\tilde{\gamma}^\#_m$, by \eqref{Eqngammas}. Finally, this map composed with the inclusion $H^0(A,\Omega^1_{A/K})\to \Omega^1_{\Ocal_{A,e}/K}$ is $\tilde{\gamma}^\bullet_m$, by Lemma \ref{LemmaCommLoc}.
\end{proof}
\begin{lemma}\label{Lemmawintgamma} Let $\gamma=(\ttt,\uu)$ be a $1$-parameter subgroup of $\Acal$, let $m\ge 0$, and let $\omega \in \Hcal(\gamma)$. Then $\tilde{\gamma}_m:V_m\to A$ is $\omega$-integral. Furthermore, if $m\le p-2$ and $\omega$ extends to a section $\widetilde{\omega}\in H^0(\Acal,\Omega^1_{\Acal/R})$, then $\tilde{\gamma}'_m:V'_m\to A'$ is $\omega'$-integral, where $\omega'$ is the restriction of $\widetilde{\omega}$ to $A'$. 
\end{lemma}
\begin{proof} By Lemma \ref{LemmaHisKernel} we have $\tilde{\gamma}^\bullet (\omega)=0$. Lemma \ref{LemmaCompatibility} gives $\tilde{\gamma}^\bullet_m(\omega)=\rho_m(\tilde{\gamma}^\bullet (\omega))=0$. Therefore, $\tilde{\gamma}_m:V_m\to A$ is $\omega$-integral.

Suppose now that $m\le p-2$. The morphism $\tilde{\gamma}^R_m:\Vcal_m\to \Acal$ induces a map on differentials $(\tilde{\gamma}^R_m)^\bullet : H^0(\Acal,\Omega^1_{\Acal/R})\to H^0(\Vcal_m,\Omega^1_{\Vcal_m/R})$. By base change we get a commutative diagram
$$
\begin{tikzcd}
H^0(A',\Omega^1_{A/\kk})  \arrow{d}{(\tilde{\gamma}'_m)^\bullet} &    H^0(\Acal,\Omega^1_{\Acal/R}) \arrow{l} \arrow{r} \arrow{d}{(\tilde{\gamma}^R_m)^\bullet} &   H^0(A,\Omega^1_{A/K}) \arrow{d}{\tilde{\gamma}_m^\bullet} \\
H^0(V'_m,\Omega^1_{V'_m/\kk})   &   H^0(\Vcal_m,\Omega^1_{\Vcal_m/R}) \arrow{l} \arrow{r}& H^0(V_m,\Omega^1_{V_m/K})
\end{tikzcd}
$$

Since $H^0(\Vcal_m,\Omega^1_{\Vcal_m/R})$ is free as an $R$-module (cf. Lemma \ref{LemmaDiffExpl}), the base change morphism $H^0(\Vcal_m,\Omega^1_{\Vcal_m/R}) \to H^0(V_m,\Omega^1_{V_m/K})$  is injective. Hence, $(\tilde{\gamma}^R_m)^\bullet(\widetilde{\omega})=0$ because $\tilde{\gamma}^\bullet_m(\omega)=0$ and the right square of the previous diagram commutes.

By commutativity of the left square of the previous diagram, we deduce that $(\tilde{\gamma}'_m)^\bullet(\omega')=0$. Hence $\tilde{\gamma}'_m:V'_m\to A'$ is $\omega'$-integral.
\end{proof}
%%
%%

%%%%%%%%%%%%%%%%%%%%%%%%%%%%%%%%%%%%%%
%%%%%%%%%%%%%%%%%%%%%%%%%%%%%%%%%%%%%%
%%%%%%%%%%%%%%%%%%%%%%%%%%%%%%%%%%%%%%
%%%%%%%%%%%%%%%%%%%%%%%%%%%%%%%%%%%%%%
%%%%%%%%%%%%%%%%%%%%%%%%%%%%%%%%%%%%%%
%%%%%%%%%%%%%%%%%%%%%%%%%%%%%%%%%%%%%%
%%%%%%%%%%%%%%%%%%%%%%%%%%%%%%%%%%%%%%
%%%%%%%%%%%%%%%%%%%%%%%%%%%%%%%%%%%%%%
%%%%%%%%%%%%%%%%%%%%%%%%%%%%%%%%%%%%%%
%%%%%%%%%%%%%%%%%%%%%%%%%%%%%%%%%%%%%%
%%%%%%%%%%%%%%%%%%%%%%%%%%%%%%%%%%%%%%
%%%%%%%%%%%%%%%%%%%%%%%%%%%%%%%%%%%%%%
%%%%%%%%%%%%%%%%%%%%%%%%%%%%%%%%%%%%%%
%%%%%%%%%%%%%%%%%%%%%%%%%%%%%%%%%%%%%%
%%%%%%%%%%%%%%%%%%%%%%%%%%%%%%%%%%%%%%
%%%%%%%%%%%%%%%%%%%%%%%%%%%%%%%%%%%%%%
%%%%%%%%%%%%%%%%%%%%%%%%%%%%%%%%%%%%%%
%%%%%%%%%%%%%%%%%%%%%%%%%%%%%%%%%%%%%%

%%%%%%%%%%%%%%%%%%%%%%%%%%%%%%%%%%%%%%
%%%%%%%%%%%%%%%%%%%%%%%%%%%%%%%%%%%%%%
%%%%%%%%%%%%%%%%%%%%%%%%%%%%%%%%%%%%%%
%%%%%%%%%%%%%%%%%%%%%%%%%%%%%%%%%%%%%%
%%%%%%%%%%%%%%%%%%%%%%%%%%%%%%%%%%%%%%
%%%%%%%%%%%%%%%%%%%%%%%%%%%%%%%%%%%%%%

\section{The main result} \label{SecMainResult}

\subsection{Statement and first steps} We keep the notation from Section \ref{NotationLoc}. Let $\Acal/R$ be an abelian variety of relative dimension $n\ge 3$. Let $A=\Acal_K$ be the generic fibre and $A'=\Acal_\kk$ the special fibre. Let $\Xcal$ be an integral subscheme of $\Acal$ which is smooth, proper, of relative dimension $2$ over $R$ with geometrically irreducible fibres. Let $X=\Xcal_K$ and $X'=\Xcal_\kk$. Let $G\le A(K)$ be a finitely generated subgroup and let $\Gamma\subseteq A(K)$ be the $p$-adic closure of $G$ in $A(K)$. Our main result is:

\begin{theorem}[Main result]\label{ThmMain} With the previous notation, suppose that $\rk (G)\le 1$, that 
\begin{equation}\label{Eqnplarge}
p>\max\left\{\efrak+1, \exp(\efrak/\exp(1))\right\} , 
\end{equation}
and that either of the following conditions holds:
\begin{itemize}
\item[(i)] $n=3$, $X$ is of general type, $X'$ contains no elliptic curves over $\kk^{alg}$, and 
$$
p> (128/9)\cdot c_1^2(X)^2.
$$ 
\item[(ii)] $A'$ is simple over $\kk^{alg}$ and there is an ample divisor $H$ on $A$ such that
$$
p> \max\left\{3 c_1^2(X)+2,\frac{n!\cdot (3\deg_H(X) + \cdeg_H(X))^n}{n^n\cdot \deg(H^n)}\right\}.
$$
\end{itemize}
Then $\Gamma\cap X(K)$ is finite and we have
\begin{equation}\label{EqnMainBound}
\# \Gamma\cap X(K) \le \# X'(\kk) + \left(1-\frac{\efrak}{p-1}\right)^{-1}\cdot (q+4q^{1/2}+3)\cdot c_1^2(X).
\end{equation}
\end{theorem}

We keep the notation and assumptions of Theorem \ref{ThmMain} for the rest of this section. In addition we let $k=\kk^{alg}$. First we observe 

\begin{lemma}\label{LemmaRk1} It suffices to prove Theorem \ref{ThmMain} under the assumption $\rk(G)=1$.
\end{lemma}
\begin{proof} Suppose that $\rk(G)=0$ and let $\xi\in A(K)$ be a non-torsion point ($A(K)$ is uncountable). We may replace $G$ by the group generated by $G$ and $\xi$, which has rank $1$.
\end{proof}

Therefore, \emph{from now on we assume $\rk(G)=1$.}

The hypotheses of Theorem \ref{ThmMain} imply the following useful facts.

\begin{lemma}\label{LemmaHyps} If either of (i) or (ii) in Theorem \ref{ThmMain} holds, then $X$ and $X'$ are surfaces of general type and $1\le c_1^2(X')=c_1^2(X)<(p-2)/3$. Furthermore:

\begin{itemize}

\item If assumption (i) in Theorem \ref{ThmMain} holds, then $X'$  is an ample divisor on $A'$. 

\item If assumption (ii) in Theorem \ref{ThmMain} holds, then there is an ample divisor $H'$ on $A'$ such that $\deg_{H'}(X')=\deg_H(X)$, $\cdeg_{H'}(X')=\cdeg_H(X)$, and $\deg((H')^3)_{A'}=\deg(H^3)_A$.
\end{itemize}
\end{lemma}
\begin{proof} If $A'$ is simple over $k$ then $A$ is simple over $K^{alg}$ in which case $X\otimes K^{alg}$ does not contain translates of positive dimensional abelian subvarieties of $A\otimes K^{alg}$, and Lemma \ref{LemmaGenTypeTranslate} implies that $X$ is of general type (note that $K^{alg}\simeq \C$ as fields). Thus, if either of (i) or (ii) holds, $X$ is of general type and Theorem 9.1 in \cite{KatsuraUeno} implies that $X'$ is of general type. Lemma 9.3 in \cite{KatsuraUeno} gives $c_1^2(X')=c_1^2(X)$. We get $c_1^2(X)\ge 1$ by Lemma \ref{LemmaNef}. Either of (i) or (ii) implies $c_1^2(X)< (p-2)/3$.

If (i) holds then $X'$ is not an abelian surface because it is of general type, and since it does not contain elliptic curves over $k$, Lemma \ref{LemmaAmpleAbVar} implies that $X'$ is an ample divisor on $A'$. 

On the other hand, suppose that assumption (ii)  in Theorem \ref{ThmMain} holds. The  theory developed in the appendix to Expos\'e X in \cite{SGA6} (see also Section 20.3 in \cite{FultonInt}) gives a specialization map from the group of algebraic cycles  modulo numerical equivalence of $A$ to that of $A'$ which respects the intersection product. Let $H'$ be the specialization of $H$. Since $A'$ is geometrically simple, the divisor $H'$ is ample. Furthermore, since $\Xcal\to \Spec R$ is smooth, the specialization of a canonical divisor on $X$ is a canonical divisor on $X'$. The required equalities follow.
\end{proof}
\begin{remark}\label{RmkMainToIntro} Theorem \ref{ThmMainIntro} follows from Theorem \ref{ThmMain}. Indeed, taking $K=\Q_p$ in Theorem \ref{ThmMain}, condition \eqref{Eqnplarge} becomes $p\ge 3$. Since $c_1^2(X)\ge 1$ (Lemma \ref{LemmaHyps}), either of (i) and (ii) implies $p\ge 7$, so the condition $p\ge 3$ can be dropped. Finally, \eqref{EqnMainBound} becomes \eqref{EqnMainIntro}.
\end{remark}

We also have the following observation regarding curves in $X'_k$.

\begin{lemma}\label{LemmaGenus2} Let $C$ be an irreducible curve in $X'_k$ defined over $k$. Then $\gfrak_g(C)\ge 2$.
\end{lemma}
\begin{proof} Since $C\subseteq A'_k$ we have $\gfrak_g(C)\ge 1$. If it is $1$, then we get $\nu_{C/A'_k}:\widetilde{C}\to A'_k$ where $\widetilde{C}$ is an elliptic curve. Thus, $\nu_{C/A'_k}$ is, up to translation, a non-constant morphism of abelian varieties, and we get that its image $C$ is is an elliptic curve. But $X'_k$ contains no elliptic curves by assumption. 
\end{proof}
%%
%%

%%%%%%
%%%%%%
%%%%%%

\subsection{Choice of differentials and canonical divisor} \label{SecChoiceDiff} Recall that  we are assuming $\rk(G)=1$.

\begin{lemma}\label{LemmaChoice1PS} Let $\ttt$ be a choice of $R$-local parameters for $A$ at $e$. There is a $1$-parameter subgroup $\gamma=(\ttt,\uu)$ for $A$ such that $\Gamma\cap U_e\subseteq \im(\gamma)$. Moreover, $\gamma$ is uniquely determined, up to equivalence, by the condition that $\im(\gamma)\cap \Gamma$ strictly contains $e$.
\end{lemma}
\begin{proof} By Lemma \ref{Lemma1PSgen} and the fact that $\rk(G)=1$,  there is a $1$-parameter subgroup $\gamma$ of the form $(\ttt,\uu)$ for the given $\ttt$, such that $\im(\gamma)\cap \Gamma$ strictly contains $\{e\}$. By Lemma \ref{LemmaSaturated} and the fact that $\Gamma$ is the $p$-adic closure of the rank $1$ group $G$, the non-triviality of $\im(\gamma)\cap \Gamma$ is equivalent to $\Gamma\cap U_e\subseteq \im(\gamma)$. By Lemma \ref{LemmaEquiv1PS}, the equivalence class of $\gamma$ is unique for this last property.
\end{proof}

In view of Lemma \ref{LemmaChoice1PS} and Lemma \ref{LemmaH1PS}, the group $G$ uniquely determines a $K$-linear hyperplane $\Hcal=\Hcal(\gamma)\subseteq H^0(A, \Omega^1_{A/K})$.

\begin{lemma}\label{LemmaDiff1} There are $\omega_1,\omega_2\in \Hcal$ such that
\begin{itemize}
\item[(i)] $\omega_1,\omega_2$ are $K$-linearly independent.
\item[(ii)] $\omega_1,\omega_2\in H^0(\Acal,\Omega^1_{\Acal/R})$.
\item[(iii)] The differentials $\omega'_1,\omega'_2\in H^0(A',\Omega^1_{A'/\kk})$ obtained by reducing $\omega_1,\omega_2$ modulo $\pfrak$, are $\kk$-linearly independent.
\end{itemize}
\end{lemma}
\begin{proof} The ring $R$ is a DVR, hence $H^0(\Acal,\Omega^1_{\Acal/R})\simeq R^n$. Thus, the problem is reduced to showing the following: Given $n\ge 3$ and a $K$-linear map $f:K^n\to K$, there are $v_1,v_2\in \ker(f)\cap R^n$ which are linearly independent and such that their images in $\kk^n$ are $\kk$-linearly independent. For this, we may assume that $f$ is $R$-valued on $R^n$, and it follows that $\rk_R(\ker(f|_{R^n}))=n-1$. Thus, we can take $v_1,v_2\in \ker(f)\cap R^n$ which are $K$-linearly independent. After scaling, we can assume that $v_1,v_2\in R^n$ are primitive. If there are $b_1,b_2\in R^\times$ with $b_1 v_1+b_2 v_2\in \pfrak R^n\subseteq R^n$ then there are $c_1,c_2\in R$ such that $b_1 v_1+b_2 v_2=\varpi c_1 v_1+\varpi c_2v_2$ which is not possible since $v_1,v_2$ are $K$-linearly independent. Hence the images of  $v_1,v_2$ in $\kk^n$ are $\kk$-linearly independent.
\end{proof}

From now on we fix a choice of $\omega_1,\omega_2$ as in Lemma \ref{LemmaDiff1} and let $\omega'_1,\omega'_2\in H^0(A',\Omega^1_{A'/\kk})$ be their reductions modulo $\pfrak$. Let $\iota:X'\to A'$ be the inclusion map. For $j=1,2$ we let $u_j=\iota^\bullet(\omega'_j)\in H^0(X',\Omega^1_{X'/\kk})$ be the restriction of $\omega'_j$ to $X'$.

\begin{lemma}\label{LemmaDiff2} The $2$-form   $u_1\wedge u_2\in H^0(X',\Omega^2_{X/\kk})$ is not the zero section.
\end{lemma}
\begin{proof} Since $A'$ is an abelian variety,  (iii) in Lemma \ref{LemmaDiff1} shows that $\omega'_1\wedge \omega'_2\in H^0(A',\Omega^2_{A'/\kk})$ is not the zero section. By the assumptions (i) and (ii) in Theorem \ref{ThmMain} and by Lemma \ref{LemmaHyps}, we can apply Lemma \ref{Lemma2forms} to $\omega'_1,\omega'_2$ on $A'$ and the surface $X'\subseteq A'$ after base change to $k$. 
\end{proof}

 By Lemma \ref{LemmaDiff2} we can define $D=\divi_{X'}(u_1\wedge u_2)$ on $X'$. Lemmas \ref{LemmaHyps}, \ref{LemmaGenus2},  and \ref{LemmaGTA} give:
 
 \begin{lemma}\label{LemmaAAA}  $D$ is effective, ample,  and it is a canonical divisor of $X'$ defined over $\kk$. In particular, $D$ is numerically effective and $D> 0$.
\end{lemma}

Let $Z=\supp(D)$ and let $C_1,...,C_\ell$ be the irreducible components of $Z$ over $k$. Thus, $Z$ is a (possibly singular and reducible) curve defined over $\kk$ and we have $D=\sum_{j=1}^\ell a_j C_j$ as divisors on $X'_k=X'\otimes k$ for certain integers $a_j\ge 1$. Let us write $\nu_j=\nu_{C_j/X'_k}:\widetilde{C}_j\to X'_k$ for $j=1,...,\ell$.

\begin{lemma}\label{LemmaChoicew} There is $b\in \kk$ such that the differential $w=u_1+bu_2\in H^0(X',\Omega^1_{X'/\kk})$ satisfies that no map $\nu_j:\widetilde{C}_j\to X'_k$ is $w$-integral.
\end{lemma}
\begin{proof} By Lemmas \ref{LemmaHyps} and \ref{LemmaAAA} (especially, ampleness of $D$) we have 
$$
\ell \le \sum_{j=1}^\ell a_j \le \sum_{j=1}^\ell a_j\cdot (C_j.D)_{X'_k}=c_1^2(X')<p
$$
Given $b\in \kk$ let $w_b=u_1+bu_2\in H^0(X',\Omega^1_{X'/\kk})$. If for every $b\in \F_p\subseteq \kk$ we have that some $\nu_{j}$ is $w_b$-integral, the fact that $\ell<p$ implies that for some $j_0$ there are $b\ne b'$ in $\kk$ such that $w_b,w_{b'}\in \ker(\nu_{j_0}^\bullet)$. Hence, $\omega_1+b\omega_2$ and $\omega_1+b'\omega_2$ are in $\ker (\nu_{C_j/A'_k}^\bullet)$. Since $b\ne b'$ we deduce $\dim_k \ker (\nu_{C_j/A'_k}^\bullet)\ge 2$. By the assumptions (i) and (ii) in Theorem \ref{ThmMain} and Lemma \ref{LemmaHyps}, this contradicts Lemma \ref{LemmaSemiInj}.
\end{proof}

From now on, we fix a choice of $b\in \kk$ and the corresponding $w=u_1+bu_2\in H^0(X',\Omega^1_{X'/\kk})$ as in Lemma \ref{LemmaChoicew}.

%%%%%%
%%%%%%
%%%%%%

\subsection{Bounds from overdetermined $\omega$-integrality}\label{SecBoundsOver} %

Given $x\in X'(k)$ we let $m_{X',\omega'_1,\omega'_2}(x)$ be the supremum (possibly infinite) of all integers $m\ge 0$ with the property that there is a closed immersion $\phi : V_m^k\to X'_k$ supported at $x$ which is $\omega$-integral for both $\omega=u_1,u_2$. Here we recall from Section \ref{SecOver} that $V^k_m=\Spec k[z]/(z^{m+1})$. We observe that, fixing the ambient abelian variety $A'$ over $\kk$,  the quantity $m_{X',\omega'_1,\omega'_2}(x)$ only depends on our choices of $X'$, $\omega'_1$, $\omega'_2$, and $x$. As $X'$ is given  and $\omega'_1,\omega'_2$ are fixed, we may simply write $m(x)=m_{X',\omega'_1,\omega'_2}(x)$ (except in the proof of Lemma \ref{LemmaXY} below, where a reduction argument will require to consider other choices of $X'$).

\begin{lemma}\label{Lemmamx0} If $x\notin Z$ then $m(x)=0$.
\end{lemma}
\begin{proof} By Theorem \ref{ThmOver}. See Remark \ref{RmkOver1} for details in the case $x\notin Z=\supp(D)$.
\end{proof}
\begin{lemma}\label{Lemmamx1} If $x\in Z$ then 
$$
m(x)\le \sum_{j=1}^\ell \sum_{y\in \nu_j^{-1}(x)} a_j\cdot (\ord_y(\nu_j^\bullet(w))+1).
$$
\end{lemma}
\begin{proof} By Theorem \ref{ThmOver}, after choosing $\omega_0=w=u_1+bu_2$.
\end{proof}
\begin{lemma}\label{Lemmamxp} For every $x\in X'(k)$ we have $m(x)\le p-3$. In particular, $m(x)$ is finite.
\end{lemma}
\begin{proof} By Lemma \ref{Lemmamx0} we may assume $x\in Z$. Lemma \ref{Lemmamx1} gives
$$
m(x)\le \sum_{j=1}^\ell  a_j\cdot \deg_{\widetilde{C}_j}(\nu_j^\bullet(w)) + \sum_{j=1}^\ell a_j\cdot \# \nu_j^{-1}(x).
$$
Since $w$ is chosen as in Lemma \ref{LemmaChoicew} we get $\deg_{\widetilde{C}_j}(\nu_j^\bullet(w))=2\gfrak_g(C_j)-2\le 2\gfrak_a(C_j)-2$. On the other hand, Lemma \ref{LemmaGenus2} gives $\gfrak_g(C_j)\ge 2$ for each $j$, and in view of Lemmas \ref{LemmaHironakaGenus} and \ref{LemmaBounddelta} we get
$$
\#\nu^{-1}_j(x)\le \delta(C_j,x)+1\le \gfrak_a(C_j)-1.
$$
Using Lemmas \ref{LemmaCanonicalGenusBd}, \ref{LemmaHyps}, and \ref{LemmaAAA} this gives $m(x)\le 3\cdot  \sum_{j=1}^\ell a_j\cdot (\gfrak_a(C_j)-1)\le 3c_1^2(X')<p-2$.
\end{proof}
%%
%%

%%%%%%
%%%%%%
%%%%%%

\subsection{Bounds on residue disks}\label{SecResDisk} %

Every $\xi\in A(K)$ extends to a section $\sigma_\xi:\Spec R\to \Acal$ because $\Acal\to \Spec R$ is proper. This determines a reduction map $\red: A(K)\to A'(\kk)$ which is a group morphism. For $x\in A'(\kk)$ we let $U_x=\red^{-1}(x)\subseteq A(K)$ be the corresponding residue disk. In particular, $U_e=\ker(\red)$ as in Section \ref{SecLocLin}. The goal of this paragraph is to show the following result (under the assumptions of Theorem \ref{ThmMain}), which can be of independent interest:

\begin{proposition}\label{PropKeyU} Let $\xi\in \Gamma\cap X(K)$ and  $x=\red(\xi)\in X'(\kk)$. Then $\Gamma\cap X(K)\cap U_x$ is finite and
$$
\# \Gamma\cap X(K)\cap U_x\le 1+ m(x)\cdot  \left(1-\frac{\efrak}{p-1}\right)^{-1}.
$$
In particular, if $m(x)=0$, then $U_x$ contains at most one point of $\Gamma\cap X(K)$.
\end{proposition}

Recall that we are assuming $\rk(G)=1$. In addition, we can make the following reduction in the proof of Proposition \ref{PropKeyU}.

\begin{lemma}\label{LemmaXY} It suffices to prove Proposition \ref{PropKeyU} under the assumption $\xi=e$.
\end{lemma} 
\begin{proof} Let $\xi_0\in \Gamma\cap X(K)$ and consider the translation $\Ycal=\Xcal- \sigma_{\xi_0}(\Spec R)$. Note that the generic and special fibres of $\Ycal$ are $Y=X-\xi_0$ and $Y'=X'-x_0$ where $x_0=\red(\xi_0)$. The assumptions of Theorem \ref{ThmMain} also hold for $\Ycal$ instead of $\Xcal$. Since $\Gamma-\xi_0=\Gamma$, we find 
$$
\# \Gamma\cap X(K)\cap U_{x_0} = \# \Gamma \cap Y(K)\cap U_e.
$$ 
On the other hand, as the differentials $\omega'_1,\omega'_2\in H^0(A',\Omega^1_{A'/\kk})$ are translation invariant, we have 
$$
m_{X',\omega'_1,\omega'_2}(x_0)=m_{Y',\tau^*\omega'_1,\tau^*\omega'_2}(e)=m_{Y',\omega'_1,\omega'_2}(e)
$$ 
where $\tau:A'\to A'$ is the translation map $P\mapsto P+x_0$. The differential $\omega'_j$ is determined by $\omega_j\in H^0(A, \Omega^1_{A/K})$ (for $j=1,2$), which in turn is chosen according to Lemmas \ref{LemmaChoice1PS} and  \ref{LemmaDiff1}. Thus, the choice of $\omega'_1,\omega'_2$  only depends on $\Acal$ and $G$, not on the particular choice of $\Xcal$. Therefore, if Proposition \ref{PropKeyU} holds in the case $\xi=e$ for the given $\Acal$ and $G$, and for every choice of $\Xcal$ as in Theorem \ref{ThmMain}, then we can in particular apply it to $\Ycal$ obtaining 
$$
\# \Gamma\cap X(K)\cap U_{x_0} = \# \Gamma \cap Y(K)\cap U_e\le m_{Y',\omega'_1,\omega'_2}(e)=m_{X',\omega'_1,\omega'_2}(x_0).
$$
\end{proof}

For the rest of the current Section \ref{SecResDisk}, let us assume that $e\in \Gamma\cap X(K)$ in order to show Proposition \ref{PropKeyU} for $\xi=e$. This is enough, by Lemma \ref{LemmaXY}.

Let $t_1,...,t_n\in \mfrak_{\Acal, e}$ be a system of $R$-local parameters for $\Acal$ along $\sigma_e$ with the property that $t_1, ..., t_{n-2}$ are a system of local equations for $\Xcal$ along $\sigma_e$. This is possible by Lemma \ref{LemmaChoiceLocParam}. Let $t'_j$ be the restriction of $t_j$ to $A'$. Then $t'_1,...t'_{n-2}$ is a system of local equations for $X'\subseteq A'$ at $e$. In particular, $t_1,...,t_{n-2}$ restricted to $A$ vanish on $X(K)\cap U_e$.

Write $\ttt=(t_1,...,t_n)$ and let $\uu\in K^n$ with $|\uu |=1$ be such that $\gamma=(\ttt,\uu)$ be a $1$-parameter subgroup of $A$ with $\Gamma\cup U_e\subseteq \im(\gamma)$. By Lemma \ref{LemmaChoice1PS}, the vector $\uu$  exists, $\gamma$ is unique up to equivalence, and $\gamma$ corresponds to the hyperplane $\Hcal\subseteq H^0(A,\Omega^1_{A/K})$ from Section \ref{SecChoiceDiff}. We get
$$
\# \Gamma\cap X(K)\cap U_e \le \# X(K)\cap \im(\gamma).
$$

By Lemma \ref{LemmaKeyLogExp} (which is applicable by \eqref{Eqnplarge}), the definition of the map $\tilde{\gamma}:\pfrak\mapsto U_e$ (cf. \eqref{Eqntildegamma}), and the fact that $t_1,...,t_{n-2}$ vanish on $X(K)\cap U_e$, we deduce that for each $j=1,...,n-2$
$$
 \# X(K)\cap \im(\gamma)\le \#\{z_0\in \pfrak : \Exp_j^\ttt(z_0\cdot \uu)=0\}.
$$
Therefore, for each $j=1,..., n-2$ we obtain
\begin{equation}\label{EqnBdUn0}
\# \Gamma\cap X(K)\cap U_e \le \nn_0(\Exp_j^\ttt(z\cdot \uu), 1/q).
\end{equation}
Let $P^\ttt_{j,h}(\xx)\in K[x_1,...,x_n]$ be the homogeneous part of degree $h$ in $\Exp_j^\ttt(\xx)$. Note that $P^\ttt_{j,0}(\xx)=0$.

\begin{lemma}\label{LemmaBdNm} There are positive integers integers  $j_0\le n-2$ and $N\le m(e)+1$ with $|P^{\ttt}_{j_0,N}(\uu)|\ge 1$.
\end{lemma}
\begin{proof} Let $0\le m_0\le p-2$ be an integer satisfying $|P^{\ttt}_{j,h}(\uu)|< 1$ for each $0\le h\le m_0$ and for each $1\le j\le n-2$. We claim that $m_0\le m(e)$.

Since $\omega_1,\omega_2$ are chosen as in Lemma \ref{LemmaDiff1}, we see that Lemmas \ref{LemmaClosedImm} and \ref{Lemmawintgamma} give a closed immersion $\tilde{\gamma}'_{m_0}: V'_m\to A'$ which is $\omega'_i$-integral for $i=1,2$. Let us prove that $\tilde{\gamma}'_{m_0}$ factors through a closed immersion $\phi_{m_0}:V'_{m_0}\to X'$ supported at $x=e$. Indeed, let $G_{j,m_0}(z)\in K[z]$ be the truncation of $\Exp^\ttt_{j}(z\cdot \uu)$ up to degree $m_0$, that is,
$$
G_{j,m}(z)=\sum_{h=0}^m  P^{\ttt}_{j,h}(\uu)\cdot z^h\in K[z].
$$
Our assumption $|P^{\ttt}_{j,h}(\uu)|< 1$ for all $0\le h\le m_0$ and $1\le j\le n-2$ implies that $G_{j,m_0}(z)\in R[z]$ for all $1\le j\le n-2$, and in fact, in this case $G_{j,m_0}(z)\equiv 0 \bmod \pfrak$. In view of Lemma \ref{LemmaTruncate1ps} which gives the construction of $\tilde{\gamma}'_{m_0}$, reducing  \eqref{Eqninf1psP} modulo $\pfrak$ we get  that for each $1\le j\le n-2$
$$
(\tilde{\gamma}'_{m_0})^\#(t'_j)= G_{j,m_0}(z) \bmod (z^{m_0+1},\varpi)=0.
$$
As $t'_1,...,t'_{n-2}$ are local equations for $X'$ at $e$, this shows that $\tilde{\gamma}'_{m_0}$ factors through the inclusion $\iota:X'\to A'$. We obtain the claimed closed immersion $\phi_{m_0}:V'_{m_0}\to X'$ satisfying $\iota\circ \phi_{m_0}=\tilde{\gamma}'_{m_0}$.

Since $\tilde{\gamma}'_{m_0}$ is $\omega'_i$-integral for $i=1,2$, we obtain that $\phi_{m_0}$ is $u_i$-integral for $i=1,2$. The same holds for the base change of $\phi_{m_0}$ to $k$, which is a closed immersion $V^k_{m_0}\to X'_k$ supported at $x=e$. Thus, by definition of $m_{X',\omega'_1,\omega'_2}(x)$ (cf. Section \ref{SecBoundsOver}) we finally get $m_0\le m(e)$, as claimed.

Suppose now that $|P^{\ttt}_{j,h}(\uu)|< 1$ for each $0\le h\le p-2$ and $1\le j\le n-2$. Taking $m_0=p-2$ in our initial claim we would get $p-2\le m(e)$, which is not possible by Lemma \ref{Lemmamxp}. 

Therefore, there is some $0\le h_0\le p-2$ and some $1\le j_0\le n-2$ such that $|P^{\ttt}_{j_0,h_0}(\uu)|\ge 1$. Let $N$ be the least value of such an $h_0$ and choose any $1\le j_0\le n-2$ satisfying $|P^{\ttt}_{j_0,N}(\uu)|\ge 1$. Thus, $N$ exists and $N\le p-2$. Also, $N\ge 1$ since $P^{\ttt}_{j,0}(\uu)=0$ for each $j$. Taking $m_0=N-1\le p-3$ we see that $|P^{\ttt}_{j,h}(\uu)|< 1$ for each $h\le m_0$ and  $j\le n-2$, and out initial claim implies $N-1=m_0\le m(e)$.
\end{proof}

\begin{proof}[Proof of Proposition \ref{PropKeyU}] As proved in Lemma \ref{LemmaXY}, we may assume that $e\in \Gamma\cap X(K)$ and that $\xi=e$. Let $j_0$ and $N$ be as in Lemma \ref{LemmaBdNm}, let $M=p^{[K:\Q_p]/(p-1)}$, and let $r=1/q$. Let us write $\lambda=\efrak/(p-1)$ and observe that $\lambda<1$ by \eqref{Eqnplarge}. We remark that $M>1$, and $r<M^{-1}$ because
\begin{equation}\label{EqnFracMr}
\frac{\log M}{\log r^{-1}} =\frac{\log M}{\log q} = \frac{[K:\Q_p]}{(p-1)\ffrak} = \lambda<1.
\end{equation}
If $c_{j,\alpha}\in K$ is the coefficient of $\xx^\alpha$ in $\Exp^\ttt_{j}(\xx)$, Lemma \ref{LemmaCvExp} gives $|c_{j,\alpha}| \le M^{\| \alpha\|-1}$. Hence, we can apply Lemma \ref{LemmaMVzeros} with these choices to the power series $\Exp^\ttt_{j_0}(\xx)$ obtaining
$$
\nn_0(\Exp_{j_0}^\ttt(z\cdot \uu), 1/q)\le \left(N-\lambda\right)\left(1-\lambda\right)^{-1}
$$
where we used \eqref{EqnFracMr}. By \eqref{EqnBdUn0} and Lemma \ref{LemmaBdNm} we conclude
$$
\# \Gamma\cap X(K)\cap U_e \le  \left(m(e)+1-\lambda\right)\left(1-\lambda\right)^{-1}=1+ m(e)\cdot \left(1-\lambda\right)^{-1}.
$$
\end{proof}

%%%%%%
%%%%%%
%%%%%%

\subsection{Adding over residue disks} 

\begin{proof}[Proof of Theorem \ref{ThmMain}] By Proposition \ref{PropKeyU} and  Lemma \ref{Lemmamx0} we obtain
\begin{equation}\label{EqnFinalU}
\begin{aligned}
\# \Gamma\cap X(K)& =\sum_{x\in X'(\kk)} \# \Gamma \cap X(K)\cap U_x= \# X'(\kk) + \left(1-\frac{\efrak}{p-1}\right)^{-1}\sum_{x\in X'(\kk)} m(x)\\
&= \# X'(\kk) + \left(1-\frac{\efrak}{p-1}\right)^{-1}\sum_{x\in Z(\kk)} m(x).
\end{aligned}
\end{equation}
Recall from Lemma \ref{LemmaChoicew} that for each $j=1,...,\ell$ we have that $\nu_j^\bullet(w)$ is not the zero differential. By Lemma \ref{Lemmamx1} we have have 
\begin{equation}\label{EqnFinalm}
\begin{aligned}
\sum_{x\in Z(\kk)} m(x)&\le \sum_{x\in Z(\kk)} \sum_{j=1}^\ell \sum_{y\in \nu_j^{-1}(x)} a_j\cdot (\ord_y(\nu_j^\bullet(w))+1)\\
&\le \sum_{j=1}^\ell a_j\cdot \deg_{\widetilde{C}_j}(\nu_j^\bullet (w)) + \sum_{x\in Z(\kk)} \sum_{j=1}^\ell a_j\cdot \#\nu_j^{-1}(x).
\end{aligned}
\end{equation}
Lemma \ref{LemmaCanonicalGenusBd} (which is applicable by Lemma \ref{LemmaAAA}) and Lemma  \ref{LemmaHyps} give
\begin{equation}\label{EqnFinalS1}
\sum_{j=1}^\ell a_j\cdot \deg_{\widetilde{C}_j}(\nu_j^\bullet (w)) = \sum_{j=1}^\ell a_j\cdot (2 \gfrak_g({C}_j) -2)\le 2c_1^2(X')=2c_1^2(X).
\end{equation}
Let $D_1,..., D_s$ be the irreducible components of $Z$ over $\kk$ and let $J_1,...,J_s$ be the partition of $\{1,2,...,\ell\}$ determined by the condition $D_i\otimes k=\cup_{j\in J_i} C_j$. Since each $D_i$ is defined over $\kk$, there are integers $b_i\ge 0$ for $1\le i\le s$ such that $a_j=b_i$ for each $j\in J_i$. By Lemma \ref{LemmaWeilBound}   we obtain
$$
\begin{aligned}
\sum_{x\in Z(\kk)} \sum_{j=1}^\ell a_j\cdot \#\nu_j^{-1}(x) & =\sum_{i=1}^s b_i \cdot \sum_{x\in D_i(\kk)} \sum_{j\in J_i} \#\nu_j^{-1}(x) \\
& \le \sum_{i=1}^s b_i \cdot \left(  (q+1)\cdot \#J_i + 2q^{1/2}\cdot \sum_{j\in J_i} \gfrak_g(C_j) \right)\\
&= (q+1)\sum_{j=1}^\ell a_j + 2q^{1/2}\cdot \sum_{j=1}^\ell a_j\cdot \gfrak_g(C_j)\\
&= (q+2q^{1/2}+1)\sum_{j=1}^\ell a_j + 2q^{1/2}\cdot \sum_{j=1}^\ell a_j\cdot (\gfrak_g(C_j)-1).
\end{aligned}
$$
From Lemmas \ref{LemmaGenus2} and \ref{LemmaCanonicalGenusBd} (applicable by Lemma \ref{LemmaAAA}) we deduce
\begin{equation}\label{EqnFinalS2}
\sum_{x\in Z(\kk)} \sum_{j=1}^\ell a_j\cdot \#\nu_j^{-1}(x) \le (q+4q^{1/2}+1)\sum_{j=1}^\ell a_j\cdot (\gfrak_g(C_j)-1)\le (q+4q^{1/2}+1)c_1^2(X).
\end{equation}
Using \eqref{EqnFinalS1} and \eqref{EqnFinalS2} in \eqref{EqnFinalm} we get
$$
\sum_{x\in Z(\kk)} m(x)\le (q+4q^{1/2}+3)c_1^2(X)
$$
which together with \eqref{EqnFinalU} gives the result.
\end{proof}

%%%%%%%%%%%%%%%%%%%%%%%%%%%%%%%%%%%%%%
%%%%%%%%%%%%%%%%%%%%%%%%%%%%%%%%%%%%%%
%%%%%%%%%%%%%%%%%%%%%%%%%%%%%%%%%%%%%%
%%%%%%%%%%%%%%%%%%%%%%%%%%%%%%%%%%%%%%
%%%%%%%%%%%%%%%%%%%%%%%%%%%%%%%%%%%%%%
%%%%%%%%%%%%%%%%%%%%%%%%%%%%%%%%%%%%%%

\section{Applications for rational points} \label{SecApplications}

%%%%%%%%%
%%%%%%%%%
%%%%%%%%%

\subsection{Rational points on surfaces} \label{SecApp1} We begin with the following generalization of Theorems \ref{ThmIntroA} and \ref{ThmIntroB} to the case of an arbitrary number field.

\begin{theorem}\label{ThmFinalMain} Let $F$ be a number field, let $p$ be a prime number, let $\pfrak$ be a prime of $F$ above $p$, let $\efrak$ be the ramification index of $\pfrak$, let $\kk$ be the residue field of $\pfrak$, and let $q=\#\kk$. We assume that 
$$
p>\max\{\efrak+1, \exp(\efrak/\exp(1))\}.
$$
Let $X$ be a smooth projective surface contained in an abelian variety $A$ of dimension $n\ge 3$, both defined over $F$ and having good reduction at $\pfrak$. Let $X'$ and $A'$ be the corresponding reductions modulo $\pfrak$. Suppose  $\rk A(F)\le 1$ and that either of the following conditions holds:
\begin{itemize}
\item[(i)] $n=3$, $X$ is of general type, $X'$ contains no elliptic curves over $\kk^{alg}$, and 
$$
p> (128/9)\cdot c_1^2(X)^2.
$$ 
\item[(ii)] $A'$ is simple over $\kk^{alg}$ and there is an ample divisor $H$ on $A$ such that
$$
p> \max\left\{3 c_1^2(X)+2,\frac{n!\cdot (3\deg_H(X) + \cdeg_H(X))^n}{n^n\cdot \deg(H^n)}\right\}.
$$
\end{itemize}
Then $X(F)$ is finite and we have
$$
\# X(F) \le \# X'(\kk) + \left(1-\frac{\efrak}{p-1}\right)^{-1}\cdot (q+4q^{1/2}+3)\cdot c_1^2(X).
$$
\end{theorem}
\begin{proof} By Theorem \ref{ThmMain}, after base-change to the completion $K=F_\pfrak$ and taking $G=A(F)$. 
\end{proof}
%%
%%

%%%%%%%%%
%%%%%%%%%
%%%%%%%%%

\subsection{Primes of geometrically simple reduction} \label{SecApp2} The next result is due to Chavdarov \cite{Chavdarov}.
\begin{lemma}\label{LemmaCha} Let $A$ be an abelian variety over $\Q$ of odd dimension $n$ satisfying $\End (A_\C)=\Z$. There is a set of primes $\Pcal$ of density $1$ in the primes such that for every $p\in \Pcal$ the reduction of $A$ modulo $p$ is geometrically simple.
\end{lemma}
With this at hand, we have:

\begin{proof}[Proof of Corollary \ref{CoroDim3}] Let $\Pcal$ be the set of primes afforded by Lemma \ref{LemmaCha} applied to $A$ (as $\dim A=3$ is odd), discarding the primes of bad reduction for $X$. For every $p\in \Pcal$ we have that $X\otimes \F_p^{alg}$ does not contain elliptic curves since $A\otimes \F_p^{alg}$ is simple. We conclude by Theorem \ref{ThmIntroA}.
\end{proof}

From the argument it is clear that we can get a similar result over number fields and abelian varieties of dimension $n\ge 3$ using Theorem \ref{ThmFinalMain}, provided that a suitable extension of Chavdarov's theorem is applicable to $A$. The general problem regarding abundance of primes of geometrically simple reduction is a conjecture of Murty and Patankar \cite{MurtyPatankar}. See \cite{Zywina} and the references therein for the formulation of this conjecture and available results.

%%%%%%%%%
%%%%%%%%%
%%%%%%%%%

\subsection{Quadratic points in curves} \label{SecApp3} We write $\equiv$ for numerical equivalence.

Let $C$ be a smooth,  irreducible, projective curve of genus $g\ge 2$ with Jacobian $J$, over an algebraically closed field $k$ of characteristic different from $2$. The symmetric square of $C$,  denoted by $C^{(2)}$, is a smooth, irreducible, projective surface. The quotient map $\pi:C\times C\to C^{(2)}$ is finite of degree $2$ ramified along the diagonal $\Delta\subseteq C\times C$. Let $D=\pi(\Delta)$ and for $\xi \in C(F^{alg})$ we let $V_\xi=\pi(\{\xi\}\times C)$. As we vary $\xi$, all the divisors $V_\xi$ are numerically equivalent to each other and this numerical equivalence class  is denoted by $V$.  Using adjunction for $\pi$ (see \cite{Hartshorne} Proposition II.8.20) and intersections on $C\times C$ we find

\begin{lemma}\label{LemmaKSymm2} We have $K_{C^{(2)}}\equiv (2g-2)V-D/2$. Moreover, $V^2=1$, $(D.V)=2$, $D^2=4-4g$, and $c_1^2(C^{(2)})=(4g-9)(g-1)$.
\end{lemma}

Let $d_0$ be degree $2$ divisor on $C$. The map $C\times C\to J$ given by $(x_1,x_2)\mapsto [x_1+x_2-b_0]$ induces a  map $j:C^{(2)}\to J$ birational onto its image (since $g\ge 2$). Let $\Theta$ be a Theta divisor in $J$ and recall that it is ample.  We define $\theta=j^*\Theta$ on $C^{(2)}$. By Lemma 7 in \cite{KouvidakisTAMS93} and Lemma \ref{LemmaKSymm2} we have
 
 \begin{lemma}\label{LemmaThetaC2} $D\equiv 2((g+1)V-\theta)$. Hence, $\theta\equiv (g+1)V-D/2$ and $K_{\overline{C}^{(2)}}\equiv \theta + (g-3)V$.
\end{lemma}

\begin{corollary} \label{LemmaSymm2GT} For every curve $C$ of genus $g\ge 3$  we have that $C^{(2)}$ has ample canonical divisor. In particular, $C^{(2)}$ is a surface of general type.
\end{corollary}
\begin{proof} Since $\theta$ is ample and all the $V_{\xi}$ are effective, we get the result from Lemma \ref{LemmaThetaC2}.
\end{proof}

By Lemmas \ref{LemmaKSymm2} and \ref{LemmaThetaC2} (namely, $\theta\equiv (g+1)V-D/2$) we find

\begin{lemma}\label{LemmaThetaK} $(\theta.K_{C^{(2)}})=2g(g-2)$. 
\end{lemma}

Let $W_2=j(C^{(2)})$. It is clear that

\begin{lemma}\label{LemmanonHypEll} If $C$ is not hyperelliptic, then $j$ is injective, in which case $C^{(2)}\simeq W_2$.
\end{lemma}

Now consider a field $F$ of characteristic different from $2$, not necessarily algebraically closed and suppose that $C$ is defined over $F$. If the divisor $d_0$ is defined over $F$ then so is the map $j:C^{(2)}\to J$ and its image $W_2$. However, $\Theta$ is only defined in some algebraic extension of $F$ in general. Nevertheless, in this setting we have

\begin{lemma}\label{LemmaTheta} Assuming the existence of an effective degree $2$ divisor $d_0$ on $C$ defined over $F$, there is a divisor $H$ on $J$ defined over $F$ with $H\equiv 2\Theta$. In particular, $H$ is ample.
\end{lemma}
\begin{proof} Write $d_0=x_1+x_2$ where $x_1,x_2$ are either $F$-rational or of degree $2$ over $F$ in which case they are Galois conjugate. Let $\Theta_i$ be the $(g-1)$-fold pointwise sum of $j(V_{x_i})\subseteq J$ with itself, for $i=1,2$. Then the divisor $H=\Theta_1+\Theta_2$ has the required properties.  
\end{proof}

\begin{proof}[Proof of Corollary \ref{CoroQuadratic}] We keep the previous notation in the case $F=\Q$. We can assume the existence of the effective degree $2$ divisor $d_0$ on $C$ defined over $\Q$, for otherwise $C^{(2)}(\Q)$ is empty. Let us apply Theorem \ref{ThmIntroB} to $A=J$, $n=g$,  and $X=W_2\subseteq J$, which is a smooth projective surface satisfying $W_2\simeq C^{(2)}$ (Lemma \ref{LemmanonHypEll}). Furthermore, the reduction of $C^{(2)}$ modulo $p$ is $(C')^{(2)}$ and the isomorphism $C^{(2)}\simeq X$ induced by $j$ reduces to an isomorphism $(C')^{(2)}\simeq X'$ with $X'$ the reduction of $X$, because $C'\otimes \F_p^{alg}$ is not hyperelliptic. The divisor $H$ is taken as in Lemma \ref{LemmaTheta}. 

By the Poincar\'e formulas (see for instance \cite{KouvidakisTAMS93}) we have $\deg(H^g)=2^g\deg(\Theta^g)=2^g\cdot g!$ and $\deg(H^2.X)=4\deg(\Theta^2.W_2)=4g(g-1)$. By Lemma \ref{LemmaThetaK} we obtain $\deg(H.K_X)=2\deg(\Theta.K_{W_2})=2 (\theta.K_{C^{(2)}})=4g(g-2)$. Therefore,
$$
\frac{g!\cdot (3\deg(H^{2}.X) + \deg(H.K_X))^g}{g^g\cdot \deg(H^g)} = \frac{g!\cdot (4g(4g-5))^g}{g^g\cdot 2^g\cdot g!}=(8g-10)^g.
$$
In view of Lemma \ref{LemmaKSymm2} and the fact that $3c_1^2(X)+2=3(4g-9)(g-1)+2<(8g-10)^g$ for $g\ge 3$, the result follows from Theorem \ref{ThmIntroB}.
\end{proof}

\begin{proof}[Proof of Corollary \ref{CoroQuadratic3}] Similar to the previous argument, using Theorem \ref{ThmIntroA} instead of Theorem \ref{ThmIntroB}. Here, $X=W_2\simeq C^{(2)}$ is of general type by Lemma \ref{LemmaSymm2GT}. Lemma \ref{LemmaKSymm2} and the fact that $g=3$ give $c_1^2(X)=c_1^2(C^{(2)})=6$. Thus, $(128/9)c_1^2(X)^2= 512$ and the least prime $p>512$ is $521$.
\end{proof}

%%%%%%%%%%%%%%%%%%%%%%%%%%%%%%%%%%%%%%
%%%%%%%%%%%%%%%%%%%%%%%%%%%%%%%%%%%%%%
%%%%%%%%%%%%%%%%%%%%%%%%%%%%%%%%%%%%%%
%%%%%%%%%%%%%%%%%%%%%%%%%%%%%%%%%%%%%%
%%%%%%%%%%%%%%%%%%%%%%%%%%%%%%%%%%%%%%
%%%%%%%%%%%%%%%%%%%%%%%%%%%%%%%%%%%%%%

\section{Acknowledgments}
We would like to thank Natalia Garcia-Fritz for answering numerous questions regarding branches of curves and $\omega$-integrality. Initially, our results were proved for surfaces contained in abelian threefolds, and we thank Peter Sarnak for asking a question that led us to address the general case.

J.C. was supported by ANID Doctorado Nacional 21190304 and H.P. was supported by FONDECYT Regular grant 1190442.
%
%
%

%%%%%%%%%%%%%%%%%%%%%%%%%%%%%%%%%%%%%%
%%%%%%%%%%%%%%%%%%%%%%%%%%%%%%%%%%%%%%
%%%%%%%%%%%%%%%%%%%%%%%%%%%%%%%%%%%%%%


\begin{thebibliography}{9}      

\bibitem{AubryPerret} Y. Aubry, M. Perret, \emph{On the characteristic polynomials of the Frobenius endomorphism for projective curves over finite fields}. Finite Fields and Their Applications, 10(3), (2004) 412-431.



\bibitem{BBM} J. Balakrishnan, A. Besser, J. M\"uller,  \emph{Computing integral points on hyperelliptic curves using quadratic Chabauty}. Mathematics of Computation, 86(305), (2017) 1403-1434.

\bibitem{BD1} J. Balakrishnan, N. Dogra, \emph{Quadratic Chabauty and rational points, I: $p$-adic heights}. Duke Math. J. 167(11), (2018), 1981-2038.

\bibitem{BDeffChKim} J. Balakrishnan, N. Dogra. \emph{An effective Chabauty-Kim theorem}. Compositio Math. 155 (2019) 1057-1075.


\bibitem{BD2} J. Balakrishnan, N. Dogra. \emph{Quadratic Chabauty and rational points II: Generalised height functions on Selmer varieties}. Int. Math. Res. Not. (2020) doi:10.1093/imrn/rnz362


\bibitem{BDMTV} J. Balakrishnan, N. Dogra, J. M\"uller, J. Tuitman, J. Vonk, \emph{Explicit Chabauty--Kim for the split Cartan modular curve of level 13}. Ann. Math., 189(3), (2019) 885-944.

\bibitem{BHPV} W. Barth,  K. Hulek,  C. Peters,  and  A. Van  de  Ven. \emph{Compact complex surfaces}. Ergebnisse der Mathematik und ihrer Grenzgebiete, vol. 4, second edition.  Springer, 2004.


\bibitem{SGA6} P. Berthelot, A. Grothendieck, L. Illusie, eds. \emph{Theorie des Intersections et Theoreme de Riemann-Roch (SGA6)}. 1966/67, Lecture Notes in Math. 225, Springer (1971).
%

\bibitem{BhargavaShankar} M. Bhargava, A. Shankar, \emph{Ternary cubic forms having bounded invariants, and the existence of a positive proportion of elliptic curves having rank 0}. Ann. Math., 181 (2015), 587-621.

\bibitem{BhargavaSkinner} M. Bhargava, C. Skinner, \emph{A positive proportion of elliptic curves over Q have rank one}. Journal of the Ramanujan Mathematical Society, 29(2), (2014), 221-242.



\bibitem{Bourbaki} N. Bourbaki, \emph{Lie groups and Lie algebras}. Chapters 1-3, Elements of Mathematics (Berlin), Springer-Verlag, Berlin, 1998.

\bibitem{Brody} R. Brody, \emph{Compact Manifolds and Hyperbolicity}. Trans. Amer. Math. Soc. 235 (1978), 213-219.

\bibitem{Chabauty} C. Chabauty, \emph{Sur les points rationnels des courbes alg\'ebriques de genre sup\'erieur \`a l'unit\'e}. C. R. Acad. Sci. Paris 212 (1941), 882-885.

\bibitem{Chavdarov} N. Chavdarov, \emph{The generic irreducibility of the numerator of the zeta function in a family of curves with large monodromy}. Duke Math. J. 87 (1997), no. 1, 151-180.


\bibitem{Coleman} R. Coleman, \emph{Effective Chabauty}. Duke Math. J. 52 (1985), no. 3, 765-770.

\bibitem{Debarre} O. Debarre, \emph{Degrees of curves in abelian varieties}. Bulletin de la S. M. F., tome 122, no 3 (1994), p. 343-361

\bibitem{Deligne} P. Deligne, \emph{La conjecture de Weil: I}. Publications Math\'ematiques de l'IH\'ES. 43: 273-307 (1974).

\bibitem{Deschamps} M. Deschamps, \emph{Courbes de genre g\'eom\'etrique born\'e sur une surface de type g\'en\'eral (d'apr\`es F. A. Bogomolov)}. S\'eminaire Bourbaki 30e ann\'ee, 1977/78, Lecture Notes in Mathematics 710, Springer, 1978, No. 519.


\bibitem{Dwork} B. Dwork, \emph{On the rationality of the zeta function of an algebraic variety}. American Journal of Mathematics, 82(3), (1960), 631-648.


\bibitem{EdixhovenLido} B. Edixhoven, G. Lido, \emph{Geometric quadratic Chabauty}. Preprint, v3 in arXiv: 1910.10752

\bibitem{Faltings1} G. Faltings, \emph{Endlichkeitss\"atze f\"ur abelsche Variet\"aten \"uber Zahlk\"orpern}. (German) [Finiteness theorems for abelian varieties over number fields] Invent. Math. 73 (1983), no. 3, 349-366.

\bibitem{Faltings2} G. Faltings, \emph{Diophantine approximation on abelian varieties}. Ann. Math., 133 (1991), 549-576.
%
\bibitem{Faltings3} G. Faltings, \emph{The general case of S. Lang's conjecture}. Barsotti Symposium in Algebraic Geometry (Abano Terme, 1991). Perspect. Math. 15. Academic Press. San Diego. 1994. p. 175-182


\bibitem{Flyn} E. Flynn, \emph{A flexible method for applying Chabauty's theorem}. Compositio Math. 105 (1997), no. 1, 79-94.



\bibitem{FultonInt} W. Fulton, \emph{Intersection Theory}. Second edition, 2nd ed., Ergebnisse der Mathematik und ihrer Grenzgebiete, vol. 2, Springer-Verlag, Berlin, 1998.



\bibitem{GFthesis} N. Garcia-Fritz,  \emph{Curves of low genus on surfaces and applications to Diophantine problems}. PhD Thesis, Queen's University, 2015.

\bibitem{GFtams} N. Garcia-Fritz, \emph{Sequences of powers with second differences equal to two and hyperbolicity}. Trans. Am. Math. Soc. 370(5), 3441-3466 (2018)

\bibitem{GFijnt} N. Garcia-Fritz, \emph{Quadratic sequences of powers and Mohanty's conjecture}. International Journal of Number Theory 14.02 (2018), 479-507.


\bibitem{Green} M. Green, \emph{Holomorphic Maps to Complex Tori}. American Journal of Mathematics, 100(3), (1978) 615-620.

\bibitem{SGA1} A. Grothendieck, M. Raynaud,  \emph{Rev\^etement \'etales et groupe fondamental (SGA1)}. Lecture Note in Math. 224, Springer (1971).



\bibitem{GM} J. Gunther, J. Morrow, \emph{Irrational points on random hyperelliptic curves}. Preprint, v3 in arXiv:1709.02041

\bibitem{Hartshorne} R. Hartshorne. \emph{Algebraic  geometry}. Graduate texts in Math. vol. 52.,  Springer Science \& Business Media, 2013.



\bibitem{SilvermanHindry} M. Hindry, J. Silveman, \emph{Diophantine Geometry, An Introduction}. Graduate Texts in Mathematics, 201, Springer-Verlag New York (2000).

\bibitem{Hironaka} H. Hironaka, H. \emph{On the arithmetic genera and the effective genera of algebraic curves}. Memoirs of the College of Science, University of Kyoto. Series A: Mathematics, 30(2), (1957) 177-195.

\bibitem{HuYang} P.-C. Hu, C.-C. Yang, \emph{Meromorphic Functions over non-Archimedean Fields}. Kluwer Academic Publishers, 2000.

\bibitem{KatsuraUeno} T. Katsura, K. Ueno, \emph{On elliptic surfaces in characteristic $p$}. Math. Ann. 272(3), (1985) 291-330.

\bibitem{DKatz} D. Katz, \emph{On the number of minimal prime ideals in the completion of a local domain}. The Rocky Mountain Journal of Mathematics, 16(3), (1986) 575-578.

\bibitem{KatzZB} E. Katz, D. Zureick-Brown, \emph{The Chabauty-Coleman bound at a prime of bad reduction and Clifford bounds for geometric rank functions}. Compos. Math. 149 (2013), no. 11, 1818-1838.

\bibitem{KRZ} E. Katz, J. Rabinoff, D. Zureick-Brown, \emph{Uniform bounds for the number of rational points on curves of small Mordell-Weil rank}. Duke Math. J.,  165(16), (2016), 3189-3240.

\bibitem{Kim} M. Kim, \emph{The motivic fundamental group of $\mathbf{P}^1\setminus \{0,1,\infty\}$ and the theorem of Siegel}. Inventiones Mathematicae, 161, (2005) 629-656.

\bibitem{Klassen} M. J. Klassen, \emph{Algebraic points of low degree on curves of low rank}. Thesis, University of Arizona, 1993.



\bibitem{KouvidakisTAMS93} A. Kouvidakis, \emph{Divisors on symmetric products of curves}. Transactions of the American Mathematical Society,  (1993), 337(1), 117-128.

\bibitem{KunzDiff} E. Kunz, \emph{K\"ahler Differentials}. Advanced lectures in mathematics, Friedr. Vieweg and Sohn, Braunschweig, 1986.

\bibitem{KunzBranches} E. Kunz, \emph{Introduction to Plane Algebraic Curves}. Birkh\"auser Boston, Inc., Boston, MA, 2005.


\bibitem{LorenziniTucker} D. Lorenzini, T. Tucker, \emph{Thue equations and the method of Chabauty-Coleman}. Invent. Math. 148 (2002), 47-77.


\bibitem{Matsumura} H. Matsumura, \emph{Commutative ring theory}. Cambridge Studies in Advanced Mathematics (M. Reid, Trans.). Cambridge: Cambridge University Press (1987).

\bibitem{Matsusaka}  T. Matsusaka, \emph{On a characterization of a Jacobian variety}. Mem. College Sci. Univ. Kyoto Ser. A Math. 32 (1959), no. 1, 1-19.

\bibitem{MPsurvey} W. McCallum, B. Poonen,  \emph{The method of Chabauty and Coleman}. Explicit methods in number theory; rational points and diophantine equations, Panoramas et Synth\`eses 36, Soci\'et\'e Math. de France, (2012) 99-117.

\bibitem{Morikawa} H. Morikawa, \emph{Cycles and endomorphisms of abelian varieties}. Nagoya Math. J. 7 (1954), 95-102.


\bibitem{MumfordApp} D. Mumford, appendix to ``\emph{The Theorem of Riemann-Roch for High Multiples of an Effective Divisor on an Algebraic Surface}'' by O. Zariski. Ann. Math., 76(3), (1962) 560-615 


\bibitem{Mumford} D. Mumford, \emph{Abelian Varieties}. Tata Inst. Fund. Res. Stud. Math. 5. Tata Institute of Fundamental Research, Bombay; Oxford Univ. Press, London, 1970.


\bibitem{MurtyPatankar} V. K. Murty, V. Patankar, \emph{Splitting of abelian varieties}. Int. Math. Res. Not., 12 (2008).

\bibitem{NakaiNotes} Y. Nakai, \emph{Notes on invariant differentials on abelian varieties}. J. Math. Kyoto Univ. 3 (1963), no. 1, 127-135. 

\bibitem{NakaiOThDifAlgVar} Y. Nakai, \emph{On the theory of differentials on algebraic varieties}. J. Sci. Hiroshima Univ. Ser. A-I 27 (1963),  7-34. 

\bibitem{Park} J. Park, \emph{Effective Chabauty for symmetric powers of curves}. Preprint, v1 in arXiv:1504.05544

\bibitem{Ploski} A. P\l{}oski, \emph{Introduction to the local theory of plane algebraic curves}. Krasi\'nski T., Spodzieja S. (eds.), Analytic and Algebraic Geometry, \L\'od\'z University Press, \L\'od\'z 2013, 115-134.

\bibitem{Remond1} G. R\'emond,  \emph{D\'ecompte dans une conjecture de Lang.} Inventiones Mathematicae, (2000) 142 (3),  513-545.

\bibitem{Remond2} G. R\'emond, \emph{Sur les sous-vari\'et\'es des tores}. Compositio Mathematica 134.3 (2002) 337-366.


\bibitem{Siksek} S. Siksek, \emph{Chabauty for symmetric powers of curves}. Algebra \& Number Theory,  3(2), (2009),  209-236.


\bibitem{Stollr} M. Stoll, \emph{Independence of rational points on twists of a given curve}. Compositio Math. 142 (2006), 1201-1214

\bibitem{Stoll} M. Stoll, \emph{Uniform bounds for the number of rational points on hyperelliptic curves of small Mordell-Weil rank}. Journal of the European Mathematical Society, 21(3), (2019) 923-956.

\bibitem{VW} S. Vemulapalli, D. Wang, \emph{Uniform bounds for the number of rational points on symmetric squares of curves with low Mordell-Weil rank}. Preprint, v3 in arXiv:1708.07057

\bibitem{VojtaCompact} P. Vojta, \emph{Siegel's theorem in the compact case}. Ann. Math., 1991, p. 509-548.

\bibitem{VojtaBuchi} P. Vojta, \emph{Diagonal quadratic forms and Hilbert's tenth problem}. Hilbert's tenth problem: relations with arithmetic and algebraic geometry (Ghent, 1999), 261-274, Contemp. Math., 270, Amer. Math. Soc., Providence, RI, 2000.


\bibitem{Yau} S.-T. Yau, \emph{Intrinsic measures of compact complex manifolds}. Math. Ann., 212, (1975) 317-329.

\bibitem{ZariskiSamuel2} O. Zariski, P. Samuel, \emph{Commutative Algebra, volume II}. Graduate Texts in Mathematics, 29, Springer-Verlag Berlin Heidelberg (1960).

\bibitem{Zywina} D. Zywina, \emph{The splitting of reductions of an abelian variety}. Int. Math. Res. Not., (2014) Issue 18,  5042-5083.

\end{thebibliography}
\end{document}